\documentclass[10pt,a4paper,reqno]{amsart}

\usepackage{geometry}
\usepackage{amsmath,amsthm,amsfonts,amssymb,euscript,cite,mathrsfs,mathtools}
\usepackage{xcolor} %
\usepackage{slashed}
\usepackage{hyperref}
\usepackage{cleveref}
\crefname{equation}{Equation}{Equations}
\Crefname{equation}{Equation}{Equations}
\crefname{lemma}{Lemma}{Lemmas}
\Crefname{lemma}{Lemma}{Lemmas}
\crefname{theorem}{Theorem}{Theorems}
\Crefname{theorem}{Theorem}{Theorems}
\crefname{proposition}{Proposition}{Propositions}
\Crefname{proposition}{Proposition}{Propositions}
\crefname{corollary}{Corollary}{Corollaries}
\Crefname{corollary}{Corollary}{Corollaries}
\crefname{remark}{Remark}{Remarks}
\Crefname{remark}{Remark}{Remarks}
\crefname{definition}{Definition}{Definitions}
\Crefname{definition}{Definition}{Definitions}

\usepackage{enumitem}
\usepackage{tikz}

\hypersetup{
    colorlinks=true,    %
    linkcolor=black,    %
    citecolor=black,    %
    filecolor=black,    %
    urlcolor=black,     %
}

\renewcommand{\Im}{\operatorname{Im}}
\renewcommand{\Re}{\operatorname{Re}}

\newcommand{\CI}{\mathcal{C}^\infty}
\newcommand{\CcI}{\mathcal{C}_c^\infty}

\newcommand{\RR}{\mathbb{R}}

\newcommand{\NN}{\mathbb{N}}

\renewcommand\Re{\operatorname{Re}}

\newcommand{\A}{\mathcal{A}}

\newcommand{\Hb}{H_\text{b}}
\newcommand{\Hbh}{H_{\text{b},h}}
\newcommand{\ol}{\overline}
\newcommand{\Sph}{\mathbb{S}}

\newcommand{\tface}{{\mathrm{tf}}}
\newcommand{\zface}{{\mathrm{zf}}}
\newcommand{\bface}{{\mathrm{bf}}}

\newcommand{\pa}{\partial}

\newcommand{\ti}{\tilde}

\newcommand{\la}{\langle}
\newcommand{\ra}{\rangle}

\newcommand{\inv}{^{-1}}

\newcommand{\f}{\frac}

\newcommand{\iy}{\infty}
\renewcommand{\S}{{\mathbb S}}

\newcommand{\p}{\phi}

\newcommand{\supp}{\text{supp\,}}

\let\pminus\pm
\let\hat\widehat
\renewcommand{\pm}{\phi_{\le m}}

\def\doi#1{ {\href{http://dx.doi.org/#1}
   {{\mdseries\ttfamily DOI}}}}

\newtheorem{theorem}{Theorem}[section]
\newtheorem{corollary}[theorem]{Corollary}
\newtheorem{lemma}[theorem]{Lemma}
\newtheorem{proposition}[theorem]{Proposition}
\theoremstyle{definition}
\Crefname{claim}{Claim}{Claims}

\newtheorem{definition}[theorem]{Definition}
\newtheorem{example}[theorem]{Example}

\theoremstyle{remark}
\newtheorem{remark}[theorem]{Remark}
\newtheorem{problem}[theorem]{Problem}

\numberwithin{equation}{section}

\renewcommand{\Re}{\mathrm{Re}\,}
\renewcommand{\Im}{\mathrm{Im}\,}

\newcommand{\x}{\alpha}
\newcommand{\xb}{\beta}
\newcommand{\xd}{\delta}

	\newcommand{\xG}{\Gamma}
\newcommand{\eps}{\epsilon}

\newcommand{\xo}{\omega}

\newcommand{\xs}{\sigma}

\newcommand{\C}{\mathbb C}

\newcommand{\N}{{\mathbb N}}

\newcommand{\R}{\mathbb R}

\newcommand{\calA}{\mathcal A}
\newcommand{\cA}{\mathcal A}

\newcommand{\calC}{\mathcal C}

\newcommand{\calF}{\mathcal F}

\newcommand{\calK}{\mathcal K}

\newcommand{\calM}{\mathcal M}

\newcommand{\calY}{\mathcal Y}

\newcommand{\cE}{\mathcal E}

\newcommand{\cV}{\mathcal V}

\newcommand{\ts}{t_*}

\newcommand{\Diffb}{\mathrm{Diff}_\mathrm{b}}

\newcommand{\wh}{\widehat}

\usepackage{tikz}
\usetikzlibrary{arrows.meta, bending, shapes.geometric}

\usepackage{setspace}
\usepackage{graphicx}
\usepackage{color}
\usepackage{transparent}
\usepackage{comment}
\usepackage{fullpage} %

\usepackage{microtype} %

\usepackage{seqsplit}

\definecolor{green}{rgb}{0,0.8,0}

\definecolor{babypink}{rgb}{0.96,0.76,0.76}

\newcommand{\jts}{\la t_*\ra}

\newcommand{\pats}{\pa_{t_*}}

\title{Asymptotic Expansions for Semilinear Waves on Asymptotically Flat Spacetimes}
\author[S. Looi]{Sam Looi}
\email{Looi@caltech.edu}
\address{Div. Physics, Mathematics and Astronomy, California Institute of Technology, Pasadena, CA 91125, USA}
\author[H. Xiong]{Haoren Xiong}
\email{haorenxiong@math.ucla.edu}
\address{Department of Mathematics, University of California, Los Angeles, CA 90095, USA}

\begin{document}

\begin{abstract}
We establish precise asymptotic expansions for solutions to semilinear wave equations with power-type nonlinearities on asymptotically flat spacetimes. Our analysis focuses on two key cases: cubic nonlinearities and higher-order power nonlinearities. For cubic nonlinearities of the form $a(t,x) \, \phi^3$, we prove asymptotic expansions for the solution globally in the spacetime. In the special case of compact spatial regions, solutions exhibit the asymptotic behavior $\phi(t, x) = c \, t^{-2} + \mathcal{O}(t^{-3+})$. For higher-order nonlinearities $a(t,x) \, \phi^p$ with $p \geq 4$, we prove the solution satisfies $\phi(t, x) = d \, t^{-3} + \mathcal{O}(t^{-4+})$, thereby extending the classical Price's law (a late-time tail postulated in 1972) to nonlinear settings in a precise fashion. These results sharpen previous decay estimates for nonlinear waves. We develop a radiation field expansion and a low-energy resolvent expansion adapted to conormal asymptotic inputs, extending Hintz’s approach for linear waves to the semilinear setting. Our methods connect geometric microlocal analysis (b-calculus) with classical physical-space techniques, providing a convenient tool for analyzing asymptotic behavior of nonlinear waves. 
\end{abstract}

\maketitle

\section{Introduction}

We study late-time asymptotics for semilinear wave equations on stationary, asymptotically flat spacetimes. The goal is to go beyond decay rates and obtain asymptotic expansions, including an explicit description of the leading tail and its coefficient. The analysis distinguishes the cubic case from higher powers, and it yields both spatially compact asymptotics and a global description on the resolved spacetime.

The geometric hypotheses are tailored to stationary black hole exteriors. The class of metrics in Definition \ref{def:metric} encodes the long-range $r^{-1}$ correction to Minkowski determined by the mass parameter $m$; for Kerr one uses a minor variant, see \cite[Section 4.1]{Hin22}. We also assume spectral admissibility (Definition \ref{def:specadm}), namely mode stability and high-energy resolvent control. These assumptions hold for subextremal Kerr exteriors: mode stability is established in \cite{WhitingKerrModeStability, ShlapentokhRothmanModeStability}, and the required resolvent bounds follow from microlocal estimates at the normally hyperbolic trapped set \cite{WunschZworskiResolventTrapped,Dyatlov2015,DyatlovSpectralGaps} together with radial point estimates at the horizons and at spatial infinity \cite{VasyMicroKerrdS,VasyLAPLag}.

Quantitative information about late-time tails in the exterior, in particular polynomial lower bounds along the event horizon, are a key input in the spherically symmetric strong cosmic censorship analysis: when propagated into the black hole interior, they lead to instability at the Cauchy horizon; see \cite{Luk2019a,Luk2019b}. The analysis in this paper is restricted to the stationary, asymptotically flat exterior region. We do not study horizon crossing or the interior problem.

Let the coordinate $t_*$ be a globally defined time function that interpolates smoothly between the standard asymptotically flat time coordinate far from the black hole and a regular time coordinate near the event horizon. Specifically, $t_*$ can be chosen such that
\[
t_* \approx 
\begin{cases}
t + r_* & \text{near the event horizon } r = 2m, \\
t - r_* & \text{for large } r,
\end{cases}
\]
where $r_* = r + 2m \log(r - 2m)$ is the tortoise coordinate. This coordinate is regular across the entire exterior spacetime while asymptotically matching standard coordinates. We remark that the gradient of $t_*$ is everywhere timelike, so each level set $\Sigma_{t_*}$ is spacelike.

\medskip

We study the following initial value problem:

\begin{problem}[Initial Value Problem]
Let $(M, g)$ be a stationary and asymptotically flat spacetime with mass $m \in \R$, as defined in Definition \ref{def:metric}. Suppose that the metric $g$ is spectrally admissible (see Definition \ref{def:specadm}). Let $t_*$ be a globally defined time function that is future timelike and asymptotically matches the standard asymptotically flat (retarded) time coordinate far from any potential black hole region. We consider the following nonlinear wave equation:
\begin{equation}
\label{IVP intro}
\begin{cases}
\Box_g \phi = a(t_*,x) \phi^p, & \text{on } M, \\
\phi(0, x) = \phi_0(x), & \text{on } \Sigma_0, \\
\partial_{t_*}\phi(0, x) = \phi_1(x), & \text{on } \Sigma_0,
\end{cases}
\end{equation}
where $\Box_g$ is the wave operator associated with the metric $g$, $p \geq 3$ is an integer, and $\Sigma_0 = \{t_* = 0\}$ is the initial Cauchy hypersurface. The coefficient $a(\ts,x)$ of the nonlinearity is a function depending smoothly on $x\in \RR^3$ which is also smooth in $|x|^{-1}$ if $|x|>1$, and we assume that $a(\ts,x)$ is a symbol of order $0$ in $\ts$, i.e. satisfies the estimate \eqref{a(t*,x) estimate}. The initial data $(\phi_0, \phi_1)$ are assumed to be smooth and compactly supported outside the event horizon. 
\end{problem}

Definition \ref{def:metric} is formulated on an exterior region and therefore does not include the full Kerr or Schwarzschild manifolds with event horizons. The arguments in this paper use only the stationary asymptotically flat structure and the spectral assumptions in the exterior. With the standard modifications near the horizons (compare \cite[Section 4]{Hin22}), the same conclusions are expected to hold for subextremal Kerr exteriors, in the same manner explained by the author of \cite{Hin22}.

Our main results can be summarized as follows: The main asymptotic regimes are different for $p=3$ and for $p\ge 4$.

\begin{enumerate}
\item (Cubic case.) For $\Box_g\phi=a(t_*,x)\phi^3$, we prove that $\phi$ admits a full asymptotic description on the resolved spacetime and, in particular, that on any fixed compact set $K\Subset\RR^3$ one has
\[
\phi(t_*,x)=2c_0\,t_*^{-2}+\mathcal{O}(t_*^{-3+}),
\]
with $c_0$ given explicitly in terms of the initial data, the geometry, and the first radiation field. We prove complete asymptotics in the entire spacetime. 

\item (Higher powers.) For $\Box_g\phi=b(t_*,x)\phi^p$ with $p\ge4$, we show that on any compact set $K\Subset\RR^3$,
\[
\phi(t_*,x)=d\,t_*^{-3}+\mathcal{O}(t_*^{-4+}),
\]
where $d$ is an explicit constant determined by the background metric, the nonlinearity, and radiation field data. This matches the $t_*^{-3}$ Price tail on compact sets, while accounting for the nonlinear contribution to the coefficient.

Theorem \ref{thm:highorder} contains details about this constant. This result extends the classical linear Price's law to nonlinearities of the above form.
\end{enumerate}

Our method has two main components. First, we derive higher-order radiation field expansions at $\mathscr I^+$ with conormal remainders, using available pointwise decay estimates for semilinear waves. Second, we extend low-energy resolvent expansions to inputs with these conormal asymptotics. The low-frequency structure of the resolvent produces the logarithmic terms that control the leading late-time tail and yields formulas for the leading coefficients.

We separate the asymptotic analysis from the existence theory. Throughout, we assume the forward solution exists globally in $t_*$. This holds in standard settings (e.g. defocusing nonlinearities, or small-data focusing problems). For convenience we recall in Appendix~\ref{app:GE} the global existence results and sharp pointwise decay estimates from the literature that motivate the a~priori bounds used in the main arguments.

Our main theorems are as follows.

\begin{theorem}[Asymptotics in the near zone/Spatially compact asymptotics]\label{thm:compactmain intro}
Suppose that a global solution $\phi$ to the initial value problem \eqref{IVP intro} with $p=3$ exists. Assume either that the initial data are smooth and compactly supported, or that they are smooth and satisfy the conditions \eqref{initial data u0, u1} with solution satisfying the conormal bounds \eqref{phi conormal est}. Then for any compact set $K \Subset X^\circ \cong \R^3$, we have the asymptotic expansion
\begin{equation}
\label{phi asymp spatial compact intro}
|\pa_{\ts}^j \pa_x^\beta (\phi - 2c_0 t_*^{-2})|\leq C_{j\beta K} t_*^{-3-j+},\quad \forall\, j\in\mathbb{N},\ \beta\in\mathbb{N}^3,
\end{equation}
where $c_0$ is a computable constant given by \eqref{eq:leadcoefficient}, which depends on the initial data, the nonlinearity, and the first-order radiation field of the solution.
\end{theorem}

To formulate the global statement, we work on a resolved spacetime $M_+$ obtained by compactifying the causal future of $\{t_*=0\}$ and blowing up the corner corresponding to $t_*\to\infty$ and $r\to\infty$. The resulting manifold has four distinguished boundary hypersurfaces (Figure \ref{fig:M+}):
\begin{enumerate}
    \item The initial Cauchy surface $\Sigma$ at $t_* = 0$
    \item Future null infinity $\mathscr{I}^+$, representing the endpoints of outgoing light rays
    \item Minkowski future timelike infinity $I^+$
    \item Spatially compact future infinity $\mathcal K^+$
\end{enumerate}

This framework keeps the different asymptotic regimes separate and makes it possible to state a single expansion that is compatible at the corners where the faces meet. 

\begin{figure}[ht]
\begin{tikzpicture}[scale=2]
  \draw (0,0) -- node[below] {$\Sigma$} ++(3,0);
  
  \draw (0,0) -- node[left] {$\mathscr{I}^+$} ++(0,1.5);
  
  \draw (0,1.5) -- node[above,sloped] {$I^+$} ++(0.5,0.5);
  
  \draw (0.5,2) -- node[above] {$\mathcal K^+$} ++(2,0);
  
  \draw (2.5,2) -- node[above,sloped] {$I^+$} ++(0.5,-0.5);
  
  \draw (3,1.5) -- node[right] {$\mathscr{I}^+$} ++(0,-1.5);
\end{tikzpicture}
\caption{The resolved spacetime manifold $M_+$, with $+$ denoting positive times. The faces other than $\Sigma$ denote asymptotic regimes of the spacetime, i.e., large-time or large-radii regimes. We sometimes call the intersection points of different line segments in the figure ``corners''.}
\label{fig:M+}
\end{figure}
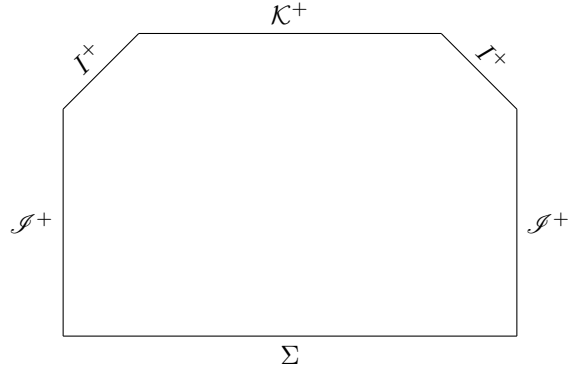

We obtain leading-order asymptotics at each boundary hypersurface of $M_+$, with a fully conormal/polyhomogeneous formulation given later in Theorem~\ref{thm:global main_body text}.

We prove global asymptotics across $M_+$ as follows. A more precise version of Theorem~\ref{thm:global main} is presented in Theorem~\ref{thm:global main_body text} using the formalism of conormal spaces.
\begin{theorem}[Global asymptotics]\label{thm:global main}
Under the same assumptions as in Theorem \ref{thm:compactmain intro}, the solution $\phi$ has the following leading order terms at the various boundary hypersurfaces of $M_+$ as shown in Figure~\ref{fig:M+}:
\begin{enumerate}
\item On $\calK^+$: $2c_0 t_*^{-2}$, where $c_0$ is given by \eqref{eq:leadcoefficient};

\item On $I^+$: $2c_0 t_*^{-2} \frac{\rho t_*}{\rho t_* + 2}$, where $\rho:=|x|^{-1}$. We remark that $v:=\rho\ts$ serves as a coordinate on $I^+$, in fact we have $I^+\cong (0,\infty)_v \times\partial X$;

\item On $\mathscr{I}^+$: $\rho\phi_{\rm rad,1}(\ts,\omega)$, where $\phi_{\rm rad,1}\in \CI([0,\infty)_{t_*};\CI(\S^2))$ is the first radiation field of $\phi$ given in Lemma \ref{lem:rad field small time} and satisfies $\phi_{\rm rad,1}(\ts,\omega) - c_0 \ts^{-1} \in \mathcal{O}(\ts^{-2+})$.
\end{enumerate}
\end{theorem}

For higher-order power nonlinearities $\phi^p$ with $p \geq 4$, we prove that the solution admits an asymptotic expansion consistent with a sharpened version of Price's law, but with a modified constant:
\begin{theorem}[Asymptotics for higher-order nonlinearities]\label{thm:highorder}
Assume that a global solution $\phi$ to the initial value problem \eqref{IVP intro} with $p\geq 4$ exists, with a nonlinearity $b(\ts,x) \phi^p$, and with smooth and compactly supported initial data. Then on any compact set $K \Subset X^\circ \cong \R^3$, we have the asymptotic expansion 
\begin{equation}
\label{phi asymp spatial compact highorder}
|\pa_{\ts}^j \pa_x^\beta (\phi - 2d_X(\wh{f}_0) t_*^{-3})|\leq C_{j\beta K} t_*^{-4-j+},\quad \forall\, j\in\mathbb{N},\ \beta\in\mathbb{N}^3, \ x \in K
\end{equation}
where $d_X(\wh{f}_0)$ is a constant that can be computed by the formula \eqref{eq:d(f)} with $\wh{f}_0$ defined by \eqref{eqn:fhat(0) defn Sec 6} and \eqref{the forcing f Sec 6}; therefore $d_X(\wh{f}_0)$ depends on the background metric, the integrals of the radiation fields over the sphere, and the integral of the solution over the full space.
\end{theorem}

\begin{remark}
We note that the leading order term $2d_X(\wh{f}_0) t_*^{-3}$ vanishes trivially when restricted to the Minkowskian spacetime, $\Box\phi = b\phi^p$, for $p\geq 5$. This follows from \eqref{eq:d(f)}, noting that $m=0$, and from the fact that $d(\omega)$ defined in \eqref{eq:d(omega) defn Sec 6} has vanishing average over $\mathbb{S}^2$ due to the absence of all correction terms in the metric (see Lemma \ref{lem:rewritingBoxg}) in the Minkowskian case. Nevertheless, our approach can be extended to identify the genuine leading-order term, which does not vanish trivially, in the Minkowski setting. For instance, the late-time asymptotic behavior of solutions to $\Box\phi = b\phi^5$ has a leading order term $C t_*^{-4}$. This analysis requires higher-order expansions of both the radiation fields (see Lemma~\ref{lem:rad field small time}) and the resolvent (see Lemma~\ref{thm:resolvent expansion Sec 6}). We refer the reader to Remark~\ref{remark:Minkowski techs} for further technical details.
\end{remark}

Theorem \ref{thm:highorder} complements the other two theorems by providing the asymptotic behavior of solutions for higher-order power nonlinearities. The leading order term $t_*^{-3}$ in the asymptotic expansion is consistent with a sharp version of Price's law, though with the constant $2d_X(\wh{f}_0)$, and decays faster than the $t_*^{-2}$ leading term obtained for the cubic nonlinearity. 

Although we assume in Theorems \ref{thm:compactmain intro}--\ref{thm:highorder} that the initial data are smooth and compactly supported, our approach applies to more general initial data which admits a partial expansion in the formalism of conormal spaces; for more details, please see Remark \ref{rmk:general intial data p=3} and Remark \ref{rmk:p=4 or larger}. 

The proofs rely on a careful analysis of the radiation fields associated with the solutions and the Fourier transform to study the low-frequency behavior of the resolvent operator. The key ingredients include the derivation of high-order radiation field expansions, which allow us to identify the leading-order terms in the asymptotic behavior of the solutions, and the resolvent expansion with more general inputs than in \cite[Theorem 3.1]{Hin22}, see Theorem \ref{thm:resolvent expansion Sec 6}. We remark that these results sharpen previous decay estimates by \cite{Looi22,Looi22.2,Toh2022} into precise asymptotics with explicit leading-order terms. 

The present work builds on the linear analysis in \cite{Hin22} and on decay estimates for semilinear problems in \cite{Looi22,Looi22.2,Toh2022} (though \cite{Looi22.2} also studied quasilinear problems). The additional ingredients needed for the nonlinear asymptotics are the following:
\begin{enumerate}
\item We show that the nonlinear forcing term $a(t_*,x)\phi^p$ admits a partial expansion near $\rho=0$ with a conormal remainder. This uses higher-order radiation field expansions derived from pointwise decay bounds.
\item We extend the low-energy resolvent expansion of \cite{Hin22} to such conormal inputs. This yields the singular low-frequency terms responsible for the leading tail and provides computable formulas for the leading coefficients in the expansions.
\item We combine the resolvent expansion with the radiation field analysis to obtain complete leading-order asymptotics on $M_+$ and, on compact sets, the $t_*^{-2}$ tail for $p=3$ and the $t_*^{-3}$ tail for $p\ge 4$.
\end{enumerate}

  These novelties %
give insight into the analysis of more general nonlinear wave equations. While this paper focuses on power-type nonlinearities as a canonical example, the core of our methodology is robust: the systematic expansion of the low-energy resolvent and the derivation of high-order radiation fields using b-calculus. We expect that these techniques can be adapted to treat a broader class of nonlinear problems, such as those with derivative nonlinearities or more complex algebraic structures.

\medskip

\textit{Potential future extensions:}
\begin{enumerate}
  \item \textbf{Genericity of non-vanishing leading coefficient.} For the cubic nonlinearity, the explicit formula \eqref{eq:leadcoefficient} for the coefficient $c_0$ of the leading $t_*^{-2}$ term along $\mathcal{K}^+$ and $I^+$ expresses $c_0$ as a nonlinear functional of the radiation field $\phi_{\mathrm{rad},1}$ and of the angular coefficients $c_2,d_1$ in the spatial asymptotics. Since $\phi_{\mathrm{rad},1}$ is itself determined by the Cauchy data on $\Sigma$, a natural open problem is to understand genericity of the condition $c_0\neq 0$ at the level of initial data, for instance by showing that in the class of small asymptotically flat data considered here the subset of data whose associated solution has $c_0\neq 0$ is open and dense (or at least residual) in the corresponding topology.

  \item \textbf{Relation to strong cosmic censorship (SCC).} We work on stationary, asymptotically flat exteriors, impose Cauchy data supported away from the event horizon, and derive complete late-time asymptotics in this exterior region. We do not study horizon crossing or interior evolution, and we do not address the inextendibility statements required by strong cosmic censorship, which are potential extensions of our results. Nevertheless, exterior tails and quantitative decay along the event horizon are standard inputs in modern SCC analyses: they provide lower and upper bounds for horizon fluxes which, once propagated across the horizon, drive blow-up/instability at the Cauchy horizon in symmetry-reduced models; see the works \cite{Luk2019a,Luk2019b,LukOh,DafermosLuk2017}. In that sense, the present paper provides a refined exterior ingredient for a model problem in nonlinear wave equations that is relevant to SCC, as the explicit tail coefficients (such as $c_0$ and $d_X(\widehat f_0)$) and the detailed asymptotics at $\mathcal K^+$, $I^+$, and $\mathscr I^+$ give more than just decay rates -- they give a precise profile of the exterior field. 

  \item \textbf{More general nonlinearities than the power-type.} In this paper we restrict to power-type nonlinearities $a(t_*,x)\phi^p$ with $p\ge 3$. The arguments only use that the nonlinearity is a fixed polynomial in $\phi$ with $p\ge 3$ and that the coefficients are sufficiently regular and slowly varying in $t_*$. It is natural to ask whether analogous asymptotics hold for semilinear equations
  \[
    \Box_g\phi = \mathcal N(t_*,x,\phi),
  \]
  where $\mathcal N$ is smooth in $(t_*,x)$, at least $C^\infty$ (or analytic) in $\phi$ near $\phi=0$, and satisfies $\mathcal N(t_*,x,0)=\partial_\phi\mathcal N(t_*,x,0)=0$. The structure of the same proofs in this paper suggests that, in such a setting, the leading late-time behaviour on compact sets should again be governed by the first non-vanishing term in the Taylor expansion of $\mathcal N$ in $\phi$ (cubic versus $p\ge 4$), but we do not pursue this extension here. A further direction, also outside the scope of the present work, would be to allow dependence on derivatives of $\phi$, or even quasilinear problems in which the metric depends on the solution. This would require a further detailed calculation of the radiation fields and its interplay with the nonlinear structure. We also remark that our approach still applies while assuming weaker pointwise decay estimates of the solution, which is therefore expected to treat more general nonlinear wave equations. 
\end{enumerate}

\subsection{Related work}
The work \cite{Hin22} by Hintz proved a sharp version of Price's law, giving precise asymptotics for solutions of the linear wave equation $\Box_g u = 0$ on a class of stationary asymptotically flat spacetimes, including subextremal Kerr black holes. Hintz's approach is based on a detailed analysis of the low energy resolvent, using geometric microlocal methods in Vasy's perspective \cite{VasyLAPLag,VasyLowEnergyLag}. The central innovation in \cite{Hin22} is the algorithmic method for obtaining an expansion of the resolvent at zero energy, which allows for the computation of the explicit leading order term in the asymptotic expansion of the solution.

The works \cite{VasyLAPLag, VasyLowEnergyLag} provide the foundational framework for the analysis in \cite{Hin22}. These papers develop a new approach to the limiting absorption principle and resolvent estimates near zero energy on asymptotically conic spaces, using Lagrangian methods. This geometric perspective allows for a unified treatment of various asymptotically flat-like settings and provides insight into the behavior of waves at low frequencies.

\vspace{.2cm}

The study of nonlinear wave equations on asymptotically flat spacetimes builds upon extensive research in partial differential equations and mathematical relativity. Foundational work on nonlinear hyperbolic equations includes \cite{K1,K2,K3,Ch,John}, establishing global existence results for small initial data or blowup results. 

Regarding global existence of solutions, we provide several remarks specifically on the energy-critical problem; refer to Appendix~\ref{app:GE} for further details. For the energy-critical problem with large and smooth initial data, \cite{ShatahStruwe1993} established a global existence result. During the same period, global existence results were established for solutions with different notions of regularity, including weak solutions, which are less regular than smooth solutions. \cite{IbrahimMajdoub} proved global existence for wave operators with time-independent and variable coefficients, while \cite{LaiZhou} proved similar results for time-dependent coefficients.

The study of decay estimates for wave equations dates back to \cite{Lax1963} and extends to more recent times \cite{K1,K2,K3,KP,Ch}. Even more recently, contributions have been made by \cite{Looi21,Looi22,Looi22.2,Looi22.3,Looi2023,LooiTohaneanu,MiaoPeiYu,Mor24,MW23,MTT,OliSte,Pec,Schlag2021,Schlue2013,Sha,SW,Tataru2013}. Also included in this list are some works that study wave equations on black hole spacetimes. Earlier contributions for power-type wave equations include those by \cite{GeorgievLindbladSogge1997} and, for the quintic nonlinearity, \cite{Gri}. For a spectral perspective on decay estimates, see \cite{BoucletBurq2021}.
In \cite{LooiTohaneanu}, the authors prove global existence and generically sharp pointwise decay for small‐data nonlinear wave equations with a null‐condition nonlinearity, allowing time‐dependent variable coefficients in both the nonlinear and linear terms, and requiring only weighted Sobolev space initial data rather than compact support.

Wave equations on black hole spacetimes, motivated by the black hole stability problem, have been studied extensively. Key contributions include \cite{AB,DR09,DR10,DR11,DR13,DRS,KS,KS1,LindbladTohaneanu2018,LindbladTohaneanu2020,Luk2013,MTT,Tataru2013,TT,Toh2022,Hin22,Angelopoulos2018,Angelopoulos2021a,Angelopoulos2021b,Angelopoulos:2018uwb}. These works have addressed both linear and nonlinear aspects of wave propagation on various black hole backgrounds.

Price's law \cite{Price} has been a central focus in understanding late-time asymptotics. Subsequent contributions from physicists, including \cite{Gundlach1994a,Gundlach1994b}, provided important numerical results and insight. Late time tails in this context have been explored by \cite{Bizon2007b,Bizon2009a,Bizon2009b,Bizon2010,Szpak2008,Szpak2009}. Rigorous mathematical investigations of Price's law and wave tails on black hole spacetimes include \cite{Angelopoulos2018,Angelopoulos2021a,Angelopoulos2021b,DafermosRodnianski2005,MTT,Tataru2013,Hin22,MaZhang2023TAMS,LooiTohaneanu}. \cite{Donninger2011,Donninger2012,FinsterKamranSmollerYau2006} use separation of variables to control the spectral measure for low frequencies. Works studying generalized versions of Price's law include \cite{Looi21,MW23}. The methods by Angelopoulos-Aretakis-Gajic \cite{Angelopoulos2018,Angelopoulos2021a,Angelopoulos2021b} and Hintz \cite{Hin22} obtain sharp asymptotic expansions and extend to higher spin equations on subextremal Kerr black holes. Recent advancements in the study of the Teukolsky equation have been made by \cite{MaZhang2023CMP,Millet,Shlapentokh-Rothman2020a,Shlapentokh-Rothman2023}.

Sharp asymptotics and precise behavior have been approached from both spectral and physical space perspectives. Spectral and microlocal methods have been developed in \cite{Baskin2015,Baskin2018,Dyatlov2015,Hintz2023,Hintz2023b,Millet,VasyMicroKerrdS,Hintz2023a,Haefner2021}, with low energy resolvent estimates studied in \cite{BH,GuillarmouHassell2008,GuillarmouHassell2009,GuillarmouHassellSikora2013,VasyLowEnergyRiemannian}. The low frequency behavior of resolvents, important for long-time asymptotics, has been investigated in \cite{BH,VasyLowEnergyRiemannian,VasyLowEnergyLag,Sussman} and the present work. Purely physical space approaches, including \cite{LukOh,Luk2015,LukOh2}, are providing complementary insights into precise asymptotics.

Related models and equations have also been studied extensively. A non-exhaustive list of these include the Klein-Gordon equation \cite{Gajic2024,Pasqualotto2023,Sussman2023,Shlapentokh-Rothman2024}, wave equations with inverse-square potentials \cite{Gajic2023,Baskin2022,Hintz2023}, spin fields \cite{Ma2022b}, massless Dirac fields on a Schwarzschild background \cite{Ma2022c}, and the spherically symmetric Maxwell-charged-scalar-field equations on a Reissner-Nordström exterior \cite{Van2022}, and other related nonlinear models \cite{Dafermos2022}. Work on Schrödinger operators \cite{Sussman,LooiSussman} and the Skyrmion model \cite{Bizon2007a} has further broadened our understanding of wave-like equations on curved backgrounds.

The strong cosmic censorship conjecture and the structure of black hole interiors have been subjects of intensive research in general relativity. The work \cite{DafermosLuk2017} made significant progress on the $C^0$ stability of the Kerr Cauchy horizon. The works \cite{Luk2019a,Luk2019b} proved results on strong cosmic censorship in spherical symmetry for two-ended asymptotically flat initial data.  \cite{Sbierski2023} showed instability of the Kerr Cauchy horizon under linearized gravitational perturbations. \cite{Giorgi2022a,KS1} focus on the nonlinear stability of slowly rotating Kerr black holes. 
Strichartz estimates and related techniques have been developed in \cite{MarzuolaMetcalfeTataruTohaneanu2010,TohaneanuKerrStrichartz,SmSo}. Integrated local energy decay estimates, also known as Morawetz estimates, have been studied in many works such as \cite{MST,MT}, with a spectral perspective provided in \cite{VasyWunschMorawetz}.  These estimates are not the focus of the present work and we do not provide an exhaustive list. A robust complement to integrated local energy decay estimates using vector field methods was developed in \cite{DR rp}, with subsequent improvements in \cite{Mos16}. High energy resolvent estimates have been explored in \cite{VasyZworskiSmcl,Vasy2011,VasyMicroKerrdS}, with \cite{Vasy2011} being a special case of \cite{VasyMicroKerrdS} that focuses on asymptotically hyperbolic spaces, while the latter also addresses Lorentzian problems. The paper \cite{Looi2025} explores how resolvent estimates relate to Morawetz estimates in the context of wave equations.

\subsection{Organization} The paper is organized as follows. In Section 2, we introduce the necessary geometric and analytic preliminaries. Section 3 is devoted to the study of the mapping properties of the resolvent operator. In Section 4, we derive the asymptotic expansions for the cubic nonlinearity in compact spatial sets before deriving global asymptotic expansions in Section 5. In Section 6, we prove long time asymptotic expansions in compact spatial sets for equations with the higher-order power nonlinearities. Appendix A presents pointwise decay estimates for solutions of power-nonlinearity wave equations on curved spacetimes. These results establish the hypotheses required for the theorems in the main text, at least in various special cases.  %

\subsection{Acknowledgements}The first author wishes to express his gratitude to Jared Wunsch for his valuable input and guidance on various aspects of this project. We are very appreciative of Mihai Tohaneanu for offering insightful comments and constructive feedback.

\section{Preliminaries}\label{sec:prelim}

In this section, we introduce the necessary geometric and analytic preliminaries, including stationary and asymptotically flat metrics, the b-calculus, and conormal spaces. While our exposition closely follows \cite{Hin22}, we also present new and detailed arguments, such as Propositions~\ref{prop:bODE0} and \ref{prop:bODEconormal}, which may help the reader develop a deeper understanding of b-differential operators and conormal spaces. We begin with the radial compactification of $\RR^3$, which plays a fundamental role in the analysis of asymptotic behavior near spatial infinity.
\begin{definition}[Spatial compactification]
    \label{DefACompact}
      The compactified spatial manifold 
      \[
        X := \ol{\R^3} = (\RR^3\sqcup([0,\infty)_\rho\times\Sph^2))/\sim
      \]
      is the radial compactification of $\R^3$, where the equivalence relation $\sim$ identifies points $r\omega\in\R^3$ for $r>0$, $\omega\in\Sph^2$ with $(\rho,\omega)$, $\rho=r^{-1}$.
    \end{definition}

\noindent  
It follows that the smooth functions on $X$ are those functions in $\CI(\RR^3)$ which in the region $\{r>1\}$ are smooth in $\rho=r^{-1}$ and the spherical variables. Near the boundary $\pa X\cong\S^2$, we shall work in the collar neighborhood $[0,\epsilon)_\rho\times\S^2$, $\epsilon>0$. Throughout this paper, we use $\mathbb{S}^2$ to denote the boundary of the spatial manifold $X$. Thus $\CI(\partial X)$ denotes the space of smooth functions on $\mathbb{S}^2$. We primarily use $\mathbb{S}^2$ notation, though $\partial X$ can be substituted where $\mathbb{S}^2$ appears.

We recall the definition of a scattering tangent bundle in order to introduce our assumptions on the metric.

\begin{definition}[Scattering tangent bundle]\label{scattering tangent bundle}
The scattering tangent bundle ${}^\textrm{sc} T X$ is the unique vector bundle over $X$ characterized as follows:

\begin{enumerate}
    \item Its smooth sections over $X^\circ$ are smooth vector fields $V$ of the form
   \[ V = a\partial_r + r^{-1}\sum_{j=1}^3 b_j \Omega_j \]
   in $(1,\infty)_r\times\mathbb{S}^2$, where $a, b_j\in\CI(X)$.
Here, $\{\Omega_1, \Omega_2, \Omega_3\} \subset \mathcal{V}(\mathbb{S}^2)$ are the rotation vector fields that span $T_p \mathbb{S}^2$ at each $p\in\mathbb{S}^2$.

\item In local coordinates $(\theta,\varphi)$ on $\mathbb{S}^2$, a section $V$ takes the form
   \[ V = a\partial_r + r^{-1} \tilde{b}_1 \partial_\theta + r^{-1} \tilde{b}_2 \partial_\varphi \]
   where $\tilde{b}_1, \tilde{b}_2 \in \CI(X)$.
\end{enumerate}
\end{definition}

We make the same assumptions on the metric as in \cite{Hin22}:
\begin{definition}\label{def:metric}
A smooth Lorentzian metric $g$ on $\R_{\ts} \times X^\circ$ is called \textit{stationary and asymptotically flat} (with mass $m\in\RR$) if $\pa_{\ts}$ is a Killing vector field, and if
    \begin{enumerate}
    \item
    $d\ts$ is everywhere future timelike, namely, $g^{00}<0$;
    \item
    the coefficients of the dual metric
    $$g\inv = g^{00} \pa_{\ts}^2 + 2\pa_{\ts} \otimes_s g^{0X} + g^{XX}$$
    satisfy the following conditions, with $\rho=r^{-1}$,
    \begin{align*}
    g^{00} &\in \rho^2 \calC^\iy(X),\\
    g^{0X} &\in -\pa_r + \rho^2 \calC^\iy(X; {}^\textrm{sc}TX),\\
    g^{XX} &\in (1-2m\rho)\pa_r^2 + \rho^2 g_{\S^2}\inv + \rho^2 \calC^\iy(X;S^2\, {}^\textrm{sc} T X).
    \end{align*}
    Here $g_{\S^2} = d\theta^2 + \sin^2\theta d\varphi^2$ is the standard metric on $\S^2$.
    \end{enumerate}
\end{definition}

The concept of b-vector fields or b-differential operators is essential in our analysis, and we recall the following definition from \cite{Hin22}:
\begin{definition}\label{b vectors}
The space $\cV_{\rm{b}}(X)$ of \textit{b-vector fields} on $X$ consists of all smooth vector fields $\xG_\text{b}$ on $X$ that are tangent to $\pa X\cong\S^2$. This means that $\Gamma_{\text{b}} = a\rho\pa_\rho + \sum_{j=1}^3 b_j \Omega_j$ in $(1,\infty)_r\times\S^2$, with $a, b_j\in\CI(X)$. 

For $m\in\N$ let $\Diffb^m(X)$ denote the space of b-differential operators of order at most $m$. Explicitly,
\begin{equation}
\Diffb^m(X) = \left\{ \sum_{|\alpha| \leq m} a_\alpha(x) V_1^{\alpha_1} \cdots V_k^{\alpha_k} : a_\alpha \in \CI(X), V_i \in \mathcal{V}_b(X) \right\},
\end{equation}
where $\mathcal{V}_b(X)$ is the space of b-vector fields on $X$, and $\alpha = (\alpha_1, \ldots, \alpha_k)$ is a multi-index with $|\alpha| = \alpha_1 + \cdots + \alpha_k \leq m$.
\end{definition}

\noindent
For a metric $g$ given in Definition \ref{def:metric}, we can rewrite its corresponding wave operator: 
\[
\Box_g = -|g|^{-1/2}\pa_\mu (|g|^{1/2} g^{\mu\nu} \pa_\nu)
\] 
using b-differential operators as follows:
\begin{lemma}
\label{lem:rewritingBoxg}
Let $g$ be a stationary and asymptotically flat metric (see Definition \ref{def:metric}). Let $\Box$ denote $\Box_g$. Then
$\Box = \wh\Box(0) - g^{00}\pa_{\ts}^2 - 2\rho \pa_{\ts} Q$ with $\wh\Box(0) \in\rho^2\Diffb^2(X)$ and $Q \in \Diffb^1(X)$.
Near $\partial X =\{\rho=0\}$,
\begin{equation}
\label{eqn:Q=Q0+Q'}
    Q = Q_0 + \rho^2\widetilde{Q},\quad Q_0=\rho \pa_\rho - 1, \quad \widetilde{Q} \in \Diffb^1(X);
\end{equation}
\begin{equation}
\label{eqn:Boxhat(0)}
    \rho^{-2} \wh\Box(0) =  L_0 + \rho L_1 + \rho^2 L_2, \quad L_2 \in  \Diffb^2(X) ;
\end{equation}
\begin{equation}
\label{eqn:operator L0}
    L_0 = -(\rho\pa_\rho)^2 + \rho\pa_\rho + \Delta_\omega,
\end{equation}
\begin{equation}
\label{eqn:operator L1}
	L_1 = 2m(\rho\pa_\rho)^2,
\end{equation}
where $\Delta_\omega = -(\sin\theta)\inv \pa_\theta \sin \theta \pa_\theta - (\sin\theta)^{-2} \pa_\varphi^2$ is the positive and spherical Laplacian. 
\end{lemma}

\noindent
We omit the proof here and refer the reader to \cite[Lemma 2.7]{Hin22}.

We furnish the compactified space $X$ with a volume density $|dg_X|$, which is related to the spacetime volume density $|dg|$ by the equation $|dg| = |dg_X| |dt_*|$. Here, $|dg|$ is the natural volume density associated with the metric $g$. We define the corresponding $L^2$ space on $X$ as $L^2(X) := L^2(X; |dg_X|)$. This setting allows us to introduce the b-Sobolev spaces.

\begin{definition}
\label{def:bSobolev}
  On the spatial manifold $X$, we introduce the following function spaces:
  \begin{enumerate}
  \item The \textit{b-Sobolev spaces} $\Hb^s(X)$ for $s \in \mathbb{R}$ are defined as follows:
    \begin{itemize}
      \item For $s = 0$, set $\Hb^0(X) = L^2(X)$.
      \item For positive integers $s$, define $\Hb^s(X)$ as the set of all $u \in L^2(X)$ such that $Au \in L^2(X)$ for every $A \in \Diffb^s(X)$.

      We equip $\Hb^s(X)$ with the Hilbert space norm
    \[
      \|u\|_{\Hb^s(X)}^2 = \sum_{j=0}^{s} \sum_k \|A_{jk}u\|_{L^2(X)}^2,
    \]
    where $\{A_{jk}\}$ is a finite set of operators spanning $\Diffb^j(X)$ over $\mathcal{C}^\infty(X)$. 
      \item For non-integer $s > 0$, define $\Hb^s(X)$ by interpolation between the integer cases.
      \item For $s < 0$, define $\Hb^s(X)$ as the dual space of $\Hb^{-s}(X)$.
    \end{itemize}
  
  \item For a small parameter $h > 0$, we define the \textit{semiclassical b-Sobolev spaces} $\Hbh^s(X)$. These share the same underlying vector space as $\Hb^s(X)$ but are equipped with the modified norm
    \[   \|u\|_{\Hbh^s(X)}^2 = \sum_{j=0}^s \sum_k \|h^j A_{jk}u\|_{L^2(X)}^2. \]
    This norm scales each b-derivative by a factor of $h$.
  
  \item The \textit{weighted b-Sobolev spaces} $\Hb^{s,\ell}(X)$ and $\Hbh^{s,\ell}(X)$ for $s, \ell \in \mathbb{R}$ are defined as
    \[
      \Hb^{s,\ell}(X) = \rho^\ell \Hb^s(X), \quad \Hbh^{s,\ell}(X) = \rho^\ell \Hbh^s(X),
    \]
    with respective norms
    \[
      \|u\|_{\Hb^{s,\ell}(X)} = \|\rho^{-\ell}u\|_{\Hb^s(X)}, \quad \|u\|_{\Hbh^{s,\ell}(X)} = \|\rho^{-\ell}u\|_{\Hbh^s(X)}.
    \]
  \end{enumerate}
\end{definition}
\begin{remark}
The operators $A_{jk}$ in the definition of the $\Hb^s(X)$ norm can be specifically chosen as all up to $j$-fold compositions of $\partial_{x^i}$ and $x^i\partial_{x^j}$ for $i,j=1,2,3$, since these vector fields span $\mathcal{V}_b(X)$ over $\mathcal{C}^\infty(X)$.
\end{remark}

\noindent
We next introduce the notion of conormal spaces, which is closely related to the b-Sobolev spaces and is essential in our future analysis.

\begin{definition}\label{def:conormalspace} %

The space of conormal functions of order $\alpha \in \mathbb{R}$ is:
    $$\mathcal{A}^\alpha(X) := \{ u \in \CI(X) : |(r\partial_r)^j \Omega^\beta u| \leq C_{j,\beta} r^{-\alpha}, \quad \forall j \in \N \text{ and multi-index }\beta \},$$
\end{definition}

\noindent
The following proposition establishes inclusions between b-Sobolev spaces and conormal spaces. While this connection is discussed in \cite{Hin22}, we provide a more detailed proof for the reader's convenience.

\begin{proposition}
\label{prop:Sobolev embedding}
We define
\[
  \Hb^{s,\ell-}(X) := \bigcap_{\eps>0} \Hb^{s,\ell-\eps}(X),\quad
  \cA^{\alpha-}(X) := \bigcap_{\eps>0} \cA^{\alpha-\eps}(X).
\]
Then we have
\begin{equation}
\label{EqASobolevEmb}
  \Hb^{\infty,\ell}(X) \subset \cA^{\ell+3/2}(X) \subset \Hb^{\infty,\ell-}(X);
\end{equation}
\end{proposition}

\begin{proof}
Noticing that $|dg_X|$ is comparable to $r^2 dr |d g_{\S^2}|$ and that $\rho\pa_\rho = -\pa_s$ if $\rho = e^{-s}$, we observe that $u\in \Hb^{\infty,\ell}(X)$ implies $ \pa_s^k \pa_{\theta,\varphi}^\beta (\rho^{-\ell} u) \in L^2(X; r^2 dr |d g_{\S^2}|)$, $\forall k,\beta$. Since $r^2 \,dr = r^3 \,ds$, it follows that in the variable $s$ we have
\[
    \rho^{-\ell - \frac{3}{2}} u \in H^\infty (\RR_s\times\S^2) \implies \pa_s^k \pa_{\theta,\varphi}^\beta (\rho^{-\ell - \frac{3}{2}} u) \in L^\infty(\RR_s\times\S^2),\quad\forall k, \beta .
\]
Returning to the variable $r$ or $\rho$, we obtain $u\in \A^{\ell+3/2}(X)$ provided that $u\in\Hb^{\infty,\ell}(X)$. Similarly, one can justify the other inclusion $\A^{\ell+3/2}(X) \subset \Hb^{\infty,\ell-}(X)$ in \eqref{EqASobolevEmb}.
\end{proof}

For future reference, we prove here two results to illustrate the connection between b-differential equations and conormal spaces in a simple setting:
\begin{proposition}[Gain-of-decay]
\label{prop:bODE0}
Let $1<\beta<\alpha$ and $u\in\A^\beta(X)$. If $(\rho\pa_\rho - 1) u = f \in\A^\alpha(X)$, then $u\in\A^\alpha(X)$. I.e., if $u$ solves $(\rho\pa_\rho - 1) u \in\A^\alpha(X)$, then $u$ gains decay.
\end{proposition}

\begin{proof}
Passing to the variable $s=-\log\rho$ as in the proof of Proposition \ref{prop:Sobolev embedding}, we shall work in the space $(0,\infty)_s\times\S_\omega^2$ replacing the collar region $\{0<\rho<1\}$. Using $\rho\pa_\rho = -\pa_s$ and suppressing the spherical variable $\omega\in\S^2$ for simplicity, we get
\begin{equation}
\label{Eq.1 prop:bODE0}
    (\pa_s + 1) u = -f,\quad 0<s<\infty
\end{equation}
We note that $f\in \A^\alpha(X)$ implies $g:=e^{\alpha s}f\in\CI_{\rm b}((0,\infty)_s\times\S^2)$, the space of smooth functions with all derivatives bounded; similarly, $e^{\beta s}u\in \CI_{\rm b}((0,\infty)_s\times\S^2)$ thus $u=\mathcal{O}(e^{-\beta s})$. We can therefore solve \eqref{Eq.1 prop:bODE0} and obtain
\[
    u(s) = e^{-s} \int_s^\infty e^t f(t)\,dt,\quad 0<s<\infty.
\]
To prove $u\in\A^\alpha(X)$, it suffices to show $e^{\alpha s}u \in \CI_{\rm b}((0,\infty)_s\times\S^2)$. We first estimate
\[
\begin{split}
    |e^{\alpha s}u(s)| = e^{(\alpha-1)s} \left| \int_s^\infty e^t f(t)\,dt \right| &= e^{(\alpha-1)s} \left| \int_s^\infty e^{-(\alpha-1)t} g(t)\,dt \right| \\
    &\leq \|g\|_{C^0} \int_s^\infty e^{(\alpha-1)(s-t)} dt = \frac{1}{\alpha-1}\|g\|_{C^0}.
\end{split}
\]
To bound the derivatives, we observe by \eqref{Eq.1 prop:bODE0} that
\[
    \pa_s(e^{\alpha s}u) = e^{\alpha s} (\pa_s +\alpha) u = e^{\alpha s} ((\alpha-1)u - f) = (\alpha-1)e^{\alpha s}u - g,
\]
we can then iterate and conclude that $e^{\alpha s}u \in \CI_{\rm b}((0,\infty)_s\times\S^2)$, thus $u\in\A^\alpha(X)$.
\end{proof}

The next proposition shows that in the case $\beta<1<\alpha$, we will obtain a $\rho^1$ leading order term when we intend to get better decay of $u$ than $\A^\beta$, using the equation $(\rho\pa_\rho - 1)u\in\A^\alpha(X)$. This is due to $\rho^1$ being an indicial solution of the regular singular operator $\rho\pa_\rho - 1$.  

\begin{proposition}
\label{prop:bODEconormal}
Let $0<\beta<1<\alpha$ and $u\in\A^\beta(X)$. Suppose $(\rho\pa_\rho - 1) u\in \A^{\alpha}(X)$. Then in the collar neighborhood $[0,1)_\rho\times\S^2$ of $\partial X$, there exists $g\in\CI(\S^2)$ such that
\[
    u = \rho g(\omega) + \widetilde{u},\quad \omega\in\S^2,\quad\text{with}\ \widetilde{u}\in \A^{\alpha}(X).
\]
\end{proposition}

\begin{proof}
Let $\chi=\chi(\rho)\in \CI([0,\infty))$ be identically $1$ for $0\leq \rho\leq 1/2$ and identically $0$ for $\rho\geq 1$, then we have
\begin{equation}
\label{prop:bODEconormal eq.1}
    (\rho\pa_\rho - 1)(\chi u) = \rho\chi' u + \chi (\rho\pa_\rho - 1)u =: h \in \A^{\alpha}(X).
\end{equation}
We use a b-operator argument similar to \cite[Lemma 2.23]{Hin22}. To this end, we recall the Mellin transform in $\rho$, defined by
\[
    \calM(f)(\xi):= \int_0^\infty \rho^{-i\xi} f(\rho) \frac{d\rho}{\rho},
\]
where we suppressed the spherical variables. The equation \eqref{prop:bODEconormal eq.1} implies
\begin{equation}
\label{prop:bODEconormal eq.2}
    (i\xi - 1)\calM(\chi u)(\xi) = \calM(h)(\xi).
\end{equation}
Here we used $\calM(\rho\pa_\rho f)(\xi) = i\xi \calM(f)(\xi)$. We note that $h \in\A^{\alpha}(X)$ implies that $\calM(h)(\xi)$ is holomorphic in $\Im\xi > -\alpha$ and satisfies the estimates
\begin{equation}
\label{prop:bODEconormal eq.3}
    \|\calM(h)(\xi)\|_{C^k(\S^2)} \leq C_{kN\delta}\langle\xi\rangle^{-N},\quad -\alpha+\delta<\Im\xi<0,
\end{equation}
for any $k, N\in\NN$ and $\delta>0$. Since $\chi u\in\A^{\beta}(X)$, the estimate \eqref{prop:bODEconormal eq.3} holds for $\calM(\chi u)(\xi)$ but with $\alpha$ replaced by $\beta$. Using the inverse Mellin transform and \eqref{prop:bODEconormal eq.2} we get
\[
    (\chi u) (\rho) = \frac{1}{2\pi}\int_{\Im\xi = -\beta+\epsilon} \rho^{i\xi} \frac{\calM(h)(\xi)}{i\xi - 1}\,d\xi.
\]
We now shift the contour of integration through the simple pole $\xi=-i$ to $\Im\xi=-\alpha+\epsilon$, with estimate \eqref{prop:bODEconormal eq.3} justifying the infinite ends, the residue theorem gives 
\[
    \chi u = \rho \calM(h)(-i) + \frac{1}{2\pi}\int_{\Im\xi = -\alpha+\epsilon} \rho^{i\xi} \frac{\calM(h)(\xi)}{i\xi - 1}\,d\xi .
\]
We can therefore infer from the estimate \eqref{prop:bODEconormal eq.3} that
\[
    \widetilde{u} := \frac{1}{2\pi}\int_{\Im\xi = -\alpha+\epsilon} \rho^{i\xi} \frac{\calM(h)(\xi)}{i\xi - 1}\,d\xi \in \A^{\alpha-\epsilon}(X),\quad\forall\epsilon>0 ,
\]
namely $\widetilde{u}\in\A^{\alpha-}(X)$. Our desired expansion for $u$ follows by setting $g:=\calM(h)(-i) = \int_0^\infty \rho^{-2}f\,d\rho \in\CI(\S^2)$ and noting that $(\rho\pa_\rho - 1)\widetilde{u} = (\rho\pa_\rho - 1)u\in \A^\alpha(X)$ implies $\widetilde{u}\in \A^\alpha(X)$ by Proposition \ref{prop:bODE0}.
\end{proof}

\section{Mapping properties of the spectral wave operator}

In this section we review the mapping properties of the specral family associated to the wave operator $\Box_g$, following closely \cite[Section 2]{Hin22}. 

We note that in this paper we shall use the Fourier transform defined by 
\[
\calF \phi(\sigma) := \hat \phi(\sigma) := \int_{\R} e^{i\sigma\ts}\phi(\ts)d\ts,
\]
which is not the usual convention. As a consequence, the inverse Fourier transform is 
\[
(\calF^{-1} f)(\ts) := \check f(\ts) = \frac{1}{2\pi}\int_{\R} e^{-i\sigma\ts} f(\sigma)d\sigma.
\]
The spectral family of $\Box\equiv\Box_g$ is the following conjugated wave operator:
\begin{equation}
\label{eqn:Boxhatsigma}
    \hat\Box(\xs):= e^{i\xs\ts}\Box e^{-i\xs\ts} = 2i\xs\rho Q + \hat\Box(0) + \xs^2g^{00} \in \rho \rm{Diff}^2_{\rm{b}}.
\end{equation}
We see that $\wh\Box(\sigma)\colon\cA^{\alpha}(X)\to\cA^{1+\alpha}(X)$ for all $\sigma\in\C$, and that at $\xs = 0$, 
\[\wh\Box(0)\colon\cA^{\alpha}(X)\to\cA^{2+\alpha}(X).\]

In the present paper, we consider only the following class of ``spectrally admissible'' metrics defined below. 

\begin{definition}
\label{def:specadm}
We call the metric $g$ \textit{spectrally admissible} if the following hold:
\begin{enumerate}
\item (Mode stability)
The nullspace of the operator $\wh\Box(\xs)$ defined on $\calA^1(X)$ is trivial for all complex numbers $\xs$ with $\Im\xs\ge0$.
\item
There exists a real number $\xd$ such that for $s\in\R, \ell < -1/2, s + \ell>-1/2$, there exists a positive constant $C$ such that the estimate
\begin{equation}
\label{boxhat resolv semiclassical bounds}
	\|u\|_{H^{s,\ell}_{\textrm{b}, |\xs|\inv}(X)} \le C |\xs|^{-1+\xd}\|\wh\Box(\xs)u\|_{H^{s,\ell+1}_{\textrm{b}, |\xs|\inv}(X)}
\end{equation}
holds for $0\le \Im\xs<1, |\Re\xs|\ge C,$ and for all $u$ for which the norms on both sides are finite.
\end{enumerate}
\end{definition}

We now turn our attention to the behavior of the resolvent $\wh\Box(\sigma)^{-1}$, both for $\sigma$ farther away from and close to 0. Our analysis builds on results of \cite{Hin22} and \cite{VasyLAPLag,VasyLowEnergyLag}.

\subsection{Resolvent regularity at nonzero frequencies} We first consider the case where $|\sigma| > R_0 > 0$ for some $R_0$:

\begin{proposition}[Resolvent estimates for non-zero frequencies, Theorem 1.1 from \cite{VasyLAPLag}]
\label{prop:nonzero_freq}
Let $0 < R_0 < R_1$ be fixed. For $s, \ell \in \mathbb{R}$ satisfying $\ell < -\frac{1}{2}$ and $s + \ell > -\frac{1}{2}$, the Mode stability assumption implies the following estimate:
\begin{equation}
\label{eq:medium_energy}
\|u\|_{{H}_b^{s,\ell}(X)} \leq C\|\wh\Box(\sigma)u\|_{{H}_b^{s,\ell+1}(X)},
\end{equation}
where $\Im\sigma \geq 0$, $R_0 < |\sigma| < R_1$, and $C$ depends on $R_0, R_1, s, \ell$.
\end{proposition}

Building on this result, we can establish regularity properties of the resolvent:

\begin{lemma}[Regularity of the resolvent]
\label{lem:resolvent_regularity}
Let $m \in \mathbb{N}_0$. Then:

1. The operator $\partial_\sigma^m\wh\Box(\sigma)^{-1}: H_b^{s,\ell+1}(X) \to H_b^{s-m,\ell}(X)$ is bounded for $\sigma, s, \ell$ as in Proposition \ref{prop:nonzero_freq}.

2. For $s \in \mathbb{R}$, $\ell < -\frac{1}{2}$, $s + \ell > -\frac{1}{2}$, $\Im \sigma \in [0,1)$, and $|\Re\sigma| \geq C > 0$, the operator
\[
\partial_\sigma^m\wh\Box(\sigma)^{-1}: H_{b,|\sigma|^{-1}}^{s,\ell+1}(X) \to |\sigma|^{-1+(m+1)\delta} H_{b,|\sigma|^{-1}}^{s-m,\ell}(X)
\]
is uniformly bounded.
\end{lemma}

\begin{proof}
The full proof can be found in \cite[Lemma 2.10]{Hin22}. The key idea is to use the factorization $\partial_\sigma\wh\Box(\sigma)^{-1} = -\wh\Box(\sigma)^{-1} \circ \partial_\sigma\wh\Box(\sigma) \circ \wh\Box(\sigma)^{-1},$
along with the fact that $\partial_\sigma\wh\Box(\sigma) \in \rho\text{Diff}_b^1 + \sigma\rho^2\mathcal{C}^\infty$.
\end{proof}

Next, we consider the low energy behavior of $\wh\Box(\sigma)^{-1}$, which is critical for our asymptotic analysis:

\begin{theorem}[Low energy resolvent estimate, Theorem 1.1 from \cite{VasyLowEnergyLag}]
\label{thm:low_energy}
Assume Mode stability for $\sigma = 0$. Let $s, \ell, \nu \in \mathbb{R}$ satisfy $\ell < -\frac{1}{2}$, $s + \ell > -\frac{1}{2}$, and $\ell - \nu \in (-\frac{3}{2}, -\frac{1}{2})$. Then for $\Im\sigma \geq 0$ and $|\sigma| \leq \sigma_0 \ll 1$, we have
\[
\|(\rho + |\sigma|)^\nu u\|_{H_b^{s,\ell}(X)} \leq C\|(\rho + |\sigma|)^{\nu-1}\wh\Box(\sigma)u\|_{H_b^{s,\ell+1}(X)}.
\]
\end{theorem}

As a consequence of the embedding estimate mentioned in \eqref{EqASobolevEmb}, we have
\begin{equation}
\label{EqAZeroMapping}
  \wh\Box(0)^{-1} \colon \cA^{2+\alpha}(X) \to \cA^{\alpha-}(X),\quad \alpha\in(0,1).
\end{equation}
In order to capture the output of the resolvent $\wh\Box(\sigma)^{-1}$ precisely near $\rho=\sigma=0$, we recall the concept of the resolved space, which is important for our analysis of low energy behavior. This construction from \cite{Hin22} allows for a precise description of the resolvent near $\rho = \sigma = 0$. For the original formulation, we refer the reader to Section 2 of \cite{Hin22}.

\begin{definition}[Resolved space]
\label{def:resolved_space}
For positive frequencies, we define the \emph{resolved space} $X^+_{\text{res}}$ as the blow-up
\[
X^+_{\text{res}} := [X \times [0,1)_\sigma; \partial X \times \{0\}].
\]
Here, $\beta: X^+_{\text{res}} \to [0,1) \times X$ denotes the blow-down map. The space $X^+_{\text{res}}$ has three boundary hypersurfaces:

\begin{itemize}
\item The \textit{transition face} $\text{tf}$: This is the new front face created by the blow-up.
\item The \textit{b-face} $\text{bf}$: This is the lift of $[0,1) \times \partial X$, given by $\overline{\beta^{-1}((0,1) \times \partial X)}$.
\item The \textit{zero face} $\text{zf}$: This is the lift of $\{0\} \times X$, given by $\overline{\beta^{-1}(\{0\} \times X^\circ)}$.
\end{itemize}
\end{definition}

This resolved space helps us analyze the resolvent's behavior near $\rho = \sigma = 0$ by delineating the transition between different asymptotic regimes.

In $\rho<1$, the functions
\[
  \rho_\bface := \frac{\rho}{\rho+\sigma},\quad
  \rho_\tface := \rho+\sigma,\quad
  \rho_\zface := \frac{\sigma}{\rho+\sigma}
\]
are smooth defining functions of the respective boundary hypersurfaces. 

Near $\bface$, we use $(\hat\rho, \sigma)$ where $\hat\rho = (\sigma r)^{-1}$. Near $\zface$, we use $(\rho, \hat r)$ where $\hat r = \sigma r$. The transition face $\tface$ represents the transition between regimes $\sigma r \lesssim 1$ and $\sigma r \gtrsim 1$.

For functions on the resolved space $X^+_\text{res}$ (which includes both radial and $\sigma$ asymptotics), we use spaces of the form $\calA^{\alpha,\beta,\gamma}(X^+_\text{res})$, where $\alpha$, $\beta$, $\gamma$ describe the decay rates at the b-face, transition face, and zero face respectively. More specifically
\[
  \cA^{\alpha,\beta,\gamma}(X^+_{\rm res})=\rho_\bface^\alpha\rho_\tface^\beta\rho_\zface^\gamma\cA^{0,0,0}(X^+_{\rm res}).
\]
When $\alpha=\beta=\gamma=0$ this denotes the set of locally bounded functions that remain locally bounded after applying any finite number of vector fields tangent to all boundary hypersurfaces of $X^+_{\rm res}$.

\begin{definition}[Partial expansions]
\label{defn:partial expansion}
An \textit{index set} is a subset $\cE\subset\C\times\N_0$ such that the number of $(z,k)\in\cE$ with $\Re z<C$ is finite for any fixed $C\in\R$, and so that $(z,k)\in\cE$ implies $(z+1,k)\in\cE$ and, if $k\geq 1$, $(z,k-1)\in\cE$. For example, a typical index set is
\[
  (z_0,k_0) := \{ (z,k)\in\C\times\N_0 \colon k\leq k_0,\ z-z_0\in\N_0 \}.
\]
Let $\cE$ be an index set and $\alpha\in\RR$. We define the space $\calA^{(\cE,\alpha)}(X)$ as the set of all functions $u$ smooth in the interior of $X$ which, near the boundary $\partial X \cong \S^2$, admit a partial expansion of the form:
\[
u -\sum_{\substack{(z,k)\in\cE\\\mathrm{Re}\,z < \alpha}} u_{z,k}(\omega) \rho^z (\log\rho)^k \in \A^\alpha (X)
\]
with $u_{z,k}(\omega)\in \CI(\S^2)$. For instance, we have
$\A^{((1,0),2)}(X) = + \rho\CI(\S^2) + \A^{2}(X)$ and $\A^{((0,1),1-)}(X) = \CI(\S^2) + (\log\rho)\CI(\S^2) + \A^{1-}(X)$. 
\end{definition}

Partially polyhomogeneous spaces like $\cA^{\alpha_\bface,\alpha_\tface,(\cE_\zface,\alpha_\zface)}(X^+_{\rm res})$ have partial expansions only at the boundary hypersurfaces at which an index set is given. 

\subsection{Conormal regularity of low-energy resolvent}
The conormal regularity of the low energy resolvent is then as follows. We prove a more general version in \cref*{lem:aux}. Both are taken from \cite[Proposition 2.14]{Hin22}.

\begin{proposition}
\label{prop:1}
  Let $\alpha\in(0,1)$ and $f\in\cA^{2+\alpha}(X)$. Then
  \begin{equation}
   \wh\Box(\sigma)^{-1}f \in \cA^{\alpha-,\alpha-,((0,0),\alpha-)}(X^+_{\rm res}).
  \end{equation}
  \end{proposition}

\begin{proof}
Let $u := \wh\Box(\sigma)^{-1}f$, $f \in \mathcal{A}^{2+\alpha}(X)$. Theorem \ref{thm:low_energy} and \eqref{EqASobolevEmb} yield $u \in \mathcal{A}^0([0,1)_\sigma; \mathcal{A}^{\alpha-}(X))$. The identity $\sigma\partial_\sigma\wh\Box(\sigma)^{-1} = -\wh\Box(\sigma)^{-1} \circ (\sigma\partial_\sigma\wh\Box(\sigma)) \circ \wh\Box(\sigma)^{-1}$, combined with Theorem \ref{thm:low_energy} (using $\ell = -3/2 + \alpha - \delta$, $\nu = 0$, $0 < \delta \leq \alpha$), shows $\mathcal{A}^{2+\alpha}(X) \to \mathcal{A}^{\alpha-}(X)$. Iterating for higher derivatives gives $u \in \mathcal{A}^{\alpha-,\alpha-,0}(X^+_\text{res})$. Refining with $(\rho + |\sigma|)^{-1} \leq |\sigma|^{-1+\alpha-\delta}\rho^{-\alpha+\delta}$, $0 < \delta < \alpha$, we get $\sigma\partial_\sigma u \in \mathcal{A}^{\delta-,\alpha-,\alpha-\delta}(X^+_\text{res})$. Thus, $u \in \mathcal{A}^{\alpha-,\alpha-,((0,0),\alpha-)}(X^+_\text{res})$.

Moreover by taking $\ell< -1/2$ close to $-1/2$ while applying Theorem~\ref{thm:low_energy} in this proof we obtain 
\[\wh\Box(\sigma)^{-1}f \in \cA^{1-,\alpha-,((0,0),\alpha-)}(X^+_{\rm res}).\]
\end{proof}

\begin{lemma}\label{lem:aux}
Let $\alpha\in(0,1)$ and $f\in \cA^{\beta}([0,1);\cA^{2+\x}(X))$ with $\beta\in\R$. We have
$$\hat\Box(\xs)\inv f (\sigma)\in \cA^{\x-,\x+\xb-,\xb}(X^{+}_{{\rm res}}).$$
\end{lemma}
\begin{proof}
Recall the resolvent identity:
$$\sigma\partial_\sigma \wh\Box(\sigma)^{-1} = - \wh\Box(\sigma)^{-1} \circ (\sigma\partial_\sigma \wh\Box(\sigma)) \circ \wh\Box(\sigma)^{-1}.$$
Given $\wh\Box(\sigma)^{-1}f \in \mathcal{A}^{\beta}([0,1); \mathcal{A}^{\alpha-}(X))$, this identity implies $\sigma\partial_\sigma \wh\Box(\sigma)^{-1}f$ has the same regularity. Applying Theorem~\ref{thm:low_energy} with $\ell = -3/2 + \alpha - \delta$, $\nu = 0$, $0 < \delta \leq \alpha$, with the estimate 
\[ \|u\|_{\Hb^{s,-\f32 + \x - \xd}} \leq C \|(\rho+|\sigma|)^{-1}\wh\Box(\sigma)u\|_{\Hb^{s,-\f12 + \x - \xd}} \]
and using $(\rho + |\sigma|)^{-1} \leq |\sigma|^{-1}$ yields the desired result.
\end{proof}

The following result appeared as Lemma 2.16 in \cite{Hin22}.
\begin{lemma}
\label{lem:lem2.16}
For $f \in \mathcal{A}^{2+\beta, 2+\beta, \beta}(X_{\text{res}}^+)$, which are functions on the resolved space with $\beta \in (0,1)$, applying the resolvent maps these functions into the space $\mathcal{A}^{\beta-, \beta-, \beta-}(X_{\text{res}}^+)$.
\end{lemma}

\begin{proof}
The proof we give is similar to the proof of Lemma 2.16 in \cite{Hin22}. Given that $f \in \mathcal{A}^0([0, 1);\mathcal{A}^{2+\alpha}(X)) \cap \mathcal{A}^{\alpha-\delta}([0, 1);\mathcal{A}^{2+\delta}(X))$ for all $0 < \delta < \alpha$, it follows that
$$\hat\Box(\sigma)^{-1} f \in \mathcal{A}^0([0, 1);\mathcal{A}^{\alpha-}(X)) \cap \mathcal{A}^{\alpha-\delta}([0, 1);\mathcal{A}^{\delta-}(X)).$$
This space is contained within $\mathcal{A}^{\alpha-, \alpha-, 0}(X_{\text{res}}^+) \cap \mathcal{A}^{\delta-, \alpha-, \alpha-\delta}(X_{\text{res}}^+)$.
\end{proof}

\paragraph{\bf The model problem for the linear wave equation $\Box_g w = 0$}

We recall the model problem $\widetilde \Box w = \hat \rho^2$ from \cite[Lemma 2.23]{Hin22} where $$\hat \rho := \rho / \sigma.$$ In particular, the notation $\widetilde{\Box}$ is used to denote the d'Alembertian $\Box$, but in the $\hat \rho$ variable (as opposed to the d'Alembertian in the variable $\frac{1}{r}$ where $r = |x|$ is the usual variable).
\begin{equation}
\label{Tilde Box}
	\widetilde{\Box} := 2i\hat{\rho}(\hat{\rho}\pa_{\hat{\rho}}-1) + \hat{\rho}^2 (- (\hat{\rho}\pa_{\hat{\rho}})^2 + \hat{\rho}\pa_{\hat{\rho}} + \Delta_\omega)
\end{equation}
This is regarded as an operator on the intermediate face tf. We refer to \cite[Definition 2.20]{Hin22}, where the author  defines the operator $\widetilde{\Box}(1)$, which we denote in this article by $\widetilde{\Box}$.

\begin{lemma}[Lemma 2.23 from \cite{Hin22}]
\label{lem:2.23Hintz}
Let $\ti f \in \CI(\mathbb{S}^2)$. We call the unique solution $v$ of $\widetilde\Box v = \hat \rho^2 \tilde f$ the model solution. The model solution lies in the space $\A^{1-,((0,1),1-)}(\rm{tf})$. The leading term of this model solution at the zeroface {\emph{zf}} is $\big((4\pi)^{-1}\int_{\S^2} \tilde{f}\,d\omega\big)\log \hat \rho$.
\end{lemma}

\begin{proposition}
\label{Hintz Lemma 2.24}
Let $\ti f \in \CI(\S^2)$ and let $\ti u_{\rm mod} := \widetilde{\Box}\inv (\hat \rho^2 \ti f)$ be the model solution, as in Lemma \ref{lem:2.23Hintz}. Then
$$\wh\Box(\xs)\inv (\rho^2 \ti f) \in \cA^{1-, ((0,0), 1-), ( (0,1), 1-)}(X_{\rm res}^+).$$
The leading order term at \text{tf} is equal to $\ti u_{\rm mod}$, and the leading order term at \text{zf} is equal to $- \big((4\pi)^{-1}\int_{\S^2} \tilde{f}\,d\omega\big)\log(\sigma/\rho)$. In fact, we have a more precise expansion:
\begin{equation}
\label{eq:better}
\wh\Box(\xs)\inv (\rho^2 \ti f) \in \tilde u_{\rm mod}(\sigma/\rho) + \tilde{u}_0(x) + \calA^{1-,1-,1-}(X_{\rm res}^+),
\end{equation}
with $\tilde{u}_0 \in \A^1(X;\R)$ being real-valued. In fact, we have $\tilde{u}_0\in\A^{((1,0),2-)}(X;\RR)$.
\end{proposition}

\begin{proof}
In view of Lemma \ref{lem:rewritingBoxg}, \eqref{eqn:Boxhatsigma} and \eqref{Tilde Box}, we have
\[
    \wh\Box(\sigma) - \sigma^2 \widetilde\Box = \rho^3 L_1 + \rho^4 L_2 + 2i\sigma\rho^3 \widetilde{Q} + \sigma^2 g^{00}.
\] where we recall from our notation \eqref{eqn:Q=Q0+Q'} that $\widetilde Q\in\Diffb^1(X).$ %
We shall now compute this operator applied to $\tilde u_{\rm mod}$. We first note that $L_1=2m(\rho\pa_\rho)^2$ eliminates the singular term $\log(\sigma/\rho)$ and the $(\sigma/\rho)^0$ term in the expansion of $\tilde u_{\rm mod}\in \cA^{1-, 0, ( (0,1), 1-)}(X_{\rm res}^+)$ near zf. It follows that $\rho^3 L_1 \tilde u_{\rm mod} \in \rho^3 \cA^{1-, 0, 1-}(X_{\rm res}^+) \subset \cA^{4-, 3, 1-}(X_{\rm res}^+)$.

\noindent
In view of our definition of $L_2\in \Diffb^2(X)$, which in turn is a consequence of the definition of the metrics we consider, \cref*{def:metric}, we see that in local coordinates $(\theta,\varphi)$ on $\S^2$,
\[
\begin{split}
	L_2 = &a_1(\rho\pa_{\rho})^2 + a_2\rho\pa_{\rho} \pa_\theta + a_3 \rho\pa_{\rho} \pa_\varphi + b_1\pa_\theta^2 + b_2\pa_\theta \pa_\varphi + b_3\pa_\varphi^2 \\
	&+ c_1 \rho\pa_{\rho} + c_2 \pa_\theta + c_3\pa_\varphi,\quad a_j, b_j, c_j \in\CI(X;\R).
\end{split}
\]
We remark that $L_2$ does not contain zeroth-order terms, thus $L_2$ eliminates the $(\sigma/\rho)^0$ term in the expansion of $\tilde u_{\rm mod}$ near zf, while $L_2 (-\log(\sigma/\rho)) = c_1\in\CI(X;\R)$ is real-valued.  Therefore,
\[
	\rho^4 L_2 \tilde u_{\rm mod} \in \rho^4 c_1(x) + \rho^4 \A^{1-,0,1-}(X^+_{\rm res}) \subset \rho^4 c_1(x) + \cA^{5-, 4, 1-}(X_{\rm res}^+).
\]
Note that $\sigma\rho^3 = \rho_\bface^3 \rho_\tface^4\rho_\zface$ and $\sigma^2 g^{00}\in \sigma^2\rho^2 \CI(X) = \rho_\bface^2 \rho_\tface^4\rho_\zface^2 \CI(X)$. Thus after viewing $\tilde u_{\rm mod}$ as an element of $\A^{1-, 0, 0-}(X^+_{\rm res})$, we have
\[
	2i\sigma\rho^3\widetilde{Q} \tilde u_{\rm mod} \in \cA^{4-,4,1-}(X_{\rm res}^+),\quad \sigma^2 g^{00}\tilde u_{\rm mod} \in \cA^{3-,4,2-}(X_{\rm res}^+).
\]
We therefore conclude from above that the following ``error term'' lies in the space
\[
    \ti e := (\wh\Box(\sigma) - \sigma^2 \widetilde\Box)\ti u_{\rm mod} \in \ti e_0(x) + \A^{3-,3,1-}(X_{\rm{res}}^+),
\]
where $\ti e_0 := \rho^4 c_1(x) \in \rho^4 \CI(X;\R)$ is real-valued. Since $\rho^2 \ti f= \sigma^2 \widetilde\Box\ti u_{\rm mod}$, we have
\[
    \wh\Box(\sigma)^{-1}(\rho^2 \ti f) = \ti u_{\rm mod} - \wh\Box(\sigma)^{-1} \ti e.
\]
We recall from \cref*{lem:lem2.16} that $\wh\Box(\sigma)^{-1}: \A^{3-,3,1-}(X_{\rm{res}}^+) \to \A^{1-,1-,1-}(X_{\rm{res}}^+)$, where the target space is the remainder term we desire in \eqref{eq:better}. Thus it remains to consider the term $\wh\Box(\sigma)^{-1} \ti e_0$. Using the resolvent identity we write
\[
	\wh\Box(\sigma)^{-1} \ti e_0 = \wh\Box(0)^{-1} \ti e_0 -\sigma\wh\Box(\sigma)^{-1} [\sigma^{-1}(\wh\Box(\sigma)-\wh\Box(0))\wh\Box(0)^{-1} \ti e_0 ].
\]
We note that $\ti u_0 := -\wh\Box(0)^{-1} \ti e_0$ is real-valued since $\ti e_0$ is real-valued. More precisely, arguing as in \cite[Lemma 3.2]{Hin22} or following the proof of Theorem \ref{thm:resolvent expansion Sec 6} (with $c(\omega)=0$ there), we obtain 
\begin{equation}
\label{eq:ti u_0}
    \ti u_0 = - \wh\Box(0)^{-1} \ti e_0 \in \ti c \rho + \A^{2-}(X) \subset \A^{((1,0),2-)}(X) .
\end{equation}
Recalling from \eqref{eqn:Boxhatsigma} and Lemma \ref{lem:rewritingBoxg} that $\sigma^{-1}(\wh\Box(\sigma)-\wh\Box(0)) = 2i\rho (Q_0+\rho^2 \widetilde{Q}) + g^{00}\sigma$ and noticing that $Q_0\rho = 0$, we get
\[
    \sigma^{-1}(\wh\Box(\sigma)-\wh\Box(0))\tilde{u}_0 = 2i\rho Q_0\tilde{u}_0 + 2i\rho^3\widetilde{Q}\tilde{u}_0 + \sigma g^{00}\tilde{u}_0 \in \A^{3-}(X) + \sigma\A^3(X) .
\]
In view of Proposition \ref{prop:1} and Lemma \ref{lem:aux} we have
\[
	\wh\Box(\sigma)^{-1}: \A^{3-}(X) \to \A^{1-,1-,((0,0),1-)}(X_{\rm{res}}^+);\quad
	\wh\Box(\sigma)^{-1}: \sigma\A^{3}(X) \to \A^{1-,2-,1}(X_{\rm{res}}^+),
\]
we therefore conclude that
\[
    \wh\Box(\sigma)^{-1} \ti e_0 \in -\ti u_0 + \A^{1-,2-,1}(X_{\rm{res}}^+),
\]
which completes the proof of \eqref{eq:better} since $\wh\Box(\sigma)^{-1}(\rho^2 \ti f) = \ti u_{\rm mod} - \wh\Box(\sigma)^{-1} \ti e$.
\end{proof}

\begin{remark}
Our expansion \eqref{eq:better} of the resolvent applied to $\rho^2\CI(\partial X)$ inputs generalizes the result for $\rho^2$ inputs in Proposition 2.24 of \cite{Hin22}. 
    Moreover, our expansion is more explicit even for radial $\rho^2$ inputs compared to the aforementioned work; we have shown that the expansion near zf corresponding to the index set $(0,1)$ comes purely from $\ti u_{\rm mod}$ and an element $e'$ in $\A^1(X;\RR)$. Indeed, we even specify the first-order expansion for $e'\in\A^{((1,0),2-)}(X)$.
\end{remark}

\section{Long-time asymptotics in compact spatial sets}

In this section, we study the long time asymptotics in compact spatial sets for the solution of the following initial value problem. Let $u_0, u_1 \in \mathcal{A}^1(X)$ satisfy
\begin{equation}
\label{initial data u0, u1}    \begin{gathered}
	u_0(x) = \rho c_1(\omega) + \rho^2 c_2(\omega) + v_0(x), \quad v_0 \in \mathcal{A}^3(X),\\
	u_1(x) = \rho d_1(\omega) + v_1(x), \quad v_1 \in \mathcal{A}^2(X),
\end{gathered}
\end{equation}
where $c_1, c_2, d_1 \in \CI(\mathbb{S}^2)$, $\rho = |x|^{-1}$, and $x = |x|\omega$ with $\omega \in \mathbb{S}^2$. We consider the initial value problem
\begin{equation}
\label{I.V. problem}
	\Box_g \phi = a(\ts,x)\phi^3,\quad \phi(0,x) = u_0(x),\quad \partial_{t_*}\phi(0,x) = u_1(x). 
\end{equation}
Here $a\in\CI(\RR_{\ts};\CI(X))$ satisfies the following estimate:
\begin{equation}
\label{a(t*,x) estimate}
    |\pa_{\ts}^k \mathcal{V}_X^I a(\ts,x)|\leq C_{k, I} \la \ts \ra^{-k},\quad\forall k,I,
\end{equation}
where $\mathcal{V}_X\in \mathrm{span}\{\pa_{r}, \pa_\rho, \Omega_1,\Omega_2,\Omega_3\}$, $k\in\NN$, and $I$ is a multi-index of any order.

\medskip
\noindent
It follows from Theorem \ref{thm:power_nonlinearity} that the solution to \eqref{I.V. problem} with compactly supported initial data satisfies 
\begin{equation}
\label{phi conormal est}
	|\pa_{\ts}^k \Gamma_b^I \phi(\ts,x)| \leq C_{k, I} \la\ts + 2|x|\ra^{-1} \la \ts \ra^{-1-k},
\end{equation}
where $\Gamma_b\in\mathrm{span}\{r\pa_{r},\Omega_1,\Omega_2,\Omega_3\}$, $I$ is a multi-index of any order.

Before stating our main theorem, we begin with two lemmas describing the asymptotics of the solution $\phi$ near null infinity $\mathscr I^{+} = \{\rho = 0\}$. The first lemma establishes the existence of the first radiation field $\phi_{\rm{rad},1}$ and second radiation field $\phi_{\rm{rad},2}$ of $\phi$.

\begin{lemma}[Existence of second-order radiation field expansion]
	\label{lem:rad field small time}
	Let $\phi$ be the solution of the initial value problem \eqref{I.V. problem}. Then there exist $\phi_{\rm{rad},1}, \phi_{\rm{rad},2} \in \CI([0,\infty)_{t_*};\CI(\S^2))$ such that
	\begin{equation}
		\label{radiation field 2nd order}
		\phi(\ts,x) \in \rho\phi_{\rm{rad},1}(\ts,\omega) + \rho^2\phi_{\rm{rad},2}(\ts,\omega) + \CI([0,\infty)_{t_*}; \A^{3}(X))
	\end{equation}
Furthermore, $\phi_{\rm{rad},2}$ is given explicitly in terms of: $\phi_{\rm{rad},1}$, the angular factors $c_2$ and $d_1$ of the initial data given in \eqref{initial data u0, u1}, and the leading angular factors $\tilde g$, $a_0$ of $g^{00}$, $a$, given in \eqref{eqn:tilde g, g00}, \eqref{eqn:a a0} respectively.
 \begin{equation}
 \label{eq:form of 2nd radi field}
    \begin{split}
		\phi_{\rm{rad},2}(\ts,\omega) &= c_2(\omega) + \frac{1}{2}\ti g(\omega) \big[ d_1(\omega) - \pa_{\ts}\phi_{\rm{rad},1}(\ts,\omega) \big] \\ 
			&\quad\, + \frac{1}{2}\int_0^{\ts} (\Delta_\omega\phi_{\rm{rad},1}(s,\omega)-a_0(s,\omega)\phi_{\rm{rad},1}^3(s,\omega)) ds.
		\end{split}
	\end{equation} 
\end{lemma} 

We make some comments on the second-order radiation field expansion:

\noindent
(i) The second-order radiation field expansion \eqref{radiation field 2nd order} plays a crucial role in obtaining an $\A^{((3,0),4))}(X)$ partial expansion of the forcing term associated to the initial value problem defined in \eqref{forcing problem}, more precisely, its average in $\ts$; see the expansion \eqref{eqn:fhat} with \eqref{eqn:c(omega)} and Definition \ref{defn:partial expansion}. In contrast, the first-order radiation field expansion would only yield that the $\ts$-average of the forcing term, $\wh{f}_0$, lies in $\A^3(X)$, which would not produce an explicit leading singular term in the resolvent expansion $\wh\Box(\sigma)^{-1}\wh{f}_0$; see \eqref{decompose u sing and reg}--\eqref{u regular}.

\noindent
(ii) We use the formula \eqref{eq:form of 2nd radi field} for $\phi_{{\rm rad},2}$ to derive the more explicit formula \eqref{eq:leadcoefficient} for the leading coefficent $c_0$ from the formula \eqref{eq:c0 again}.

\begin{remark}[Polyhomogeneous expansion at $\mathscr I^+$]
On the resolved spacetime $M_+$ we work in conormal/polyhomogeneous classes rather
than assuming a formal $\rho$-series. Using the identity \eqref{eqn:dt*Q0 phi}
and the a~priori bounds, one obtains \eqref{eqn:Q0 phi A2}.
By Proposition~\ref{prop:bODEconormal} (indicial root $1$), this forces a leading term
$\rho\,\phi_{\mathrm{rad},1}(t^*,\omega)$ with conormal remainder \eqref{rad field 1st}.
Expanding the coefficients at $\rho=0$ and refining the identity gives \eqref{eqn:Q0 phi 2nd}, and
Proposition~\ref{prop:bODE0} then inverts $(\rho\partial_\rho-1)$ at weight $\rho^2$ to produce
$\rho^2\,\phi_{\mathrm{rad},2}(t^*,\omega)$ with the explicit formula \eqref{eq:form of 2nd radi field}. Thus $\phi_{\mathrm{rad},2}$ is determined by a (time) integral relation built from $\phi_{\mathrm{rad},1}$, not by assuming a series.
For $p\ge4$, the same mechanism yields a third term $\rho^3\,\phi_{\mathrm{rad},3}$, see
\eqref{radiation field 3rd order}-\eqref{eq:form of 3rd radi field}. The construction rests on the structural assumptions of \cref{sec:prelim} together with
spectral admissibility (\cref{def:specadm}) and the initial conormal bounds mentioned in \cref{app:GE}.
\end{remark}

\begin{proof}
We recall the expression from \cref*{lem:rewritingBoxg}
\[
	\Box_g = -2\rho\pa_{\ts}(\rho\pa_{\rho}-1 + \rho^2\widetilde{Q}) + \wh\Box(0) - g^{00}\pa_{t_*}^2.
\]
We rewrite the equation \eqref{I.V. problem} as follows
\begin{equation}
\label{eqn:dt*Q0 phi}
    \pa_{\ts}(\rho\pa_{\rho}-1)\phi = (2\rho)^{-1}(\wh\Box(0)\phi - a\phi^3 - g^{00}\pa_{t_*}^2 \phi) - \pa_{t_*}\rho^2\widetilde{Q} \phi.
\end{equation}
Integrating \eqref{eqn:dt*Q0 phi} and using the initial values in \eqref{I.V. problem}, we obtain
	\begin{equation}
		\label{eqn:Q0 phi full}
		\begin{split}
			(\rho\pa_{\rho}-1)\phi(t_*,x) = &(\rho\pa_{\rho}-1)u_0 + \frac{g^{00}}{2\rho} u_1 + \rho^2\widetilde{Q} u_0 - \frac{g^{00}}{2\rho}\pa_{\ts}\phi(\ts,x) - \rho^2\widetilde{Q}\phi(\ts,x) \\
			&+\frac{1}{2\rho}\int_0^{\ts} (\wh\Box(0)\phi(s,x)-a(x)\phi^3(s,x)) ds.
		\end{split}
	\end{equation}
It follows from \eqref{phi conormal est} that $\phi(\ts,x)\in \CI([0,\infty)_{\ts};\A^1(X))$, noting that $g^{00}\in \rho^2\CI(X)$, and in view of \eqref{eqn:Boxhat(0)} and \eqref{initial data u0, u1}, \eqref{a(t*,x) estimate}, we see that
	\begin{equation}
	\label{eqn:Q0 phi A2}
		(\rho\pa_{\rho}-1)\phi(t_*,x) \in \CI([0,\infty)_{\ts};\A^2(X)).
	\end{equation}
Proposition \ref{prop:bODEconormal} and the equation \eqref{eqn:Q0 phi A2} 
give the $\rho^1$ leading term of $\phi$ at $\mathscr{I}^+$, i.e. there exists $\phi_{\rm{rad},1}\in \CI([0,\infty)_{t_*};\CI(\S^2))$ such that
	\begin{equation}
		\label{rad field 1st}
		\phi(t_*,x) - \rho \phi_{\rm{rad},1}(\ts,\omega) \in \CI([0,\infty)_{\ts};\A^{2}(X)).
	\end{equation}
To iterate this argument, we expand $g^{00}\in \rho^2\CI(X)$ at $\mathbb{S}^2$ to write
\begin{equation}
\label{eqn:tilde g, g00}
    g^{00}(x) - \rho^2 \ti g(\omega) \in \A^3(X),\quad x=\rho^{-1}\omega,\ \;\omega\in\mathbb{S}^2 .
\end{equation}
Similarly, for $a(\ts,x)\in\CI(\RR_{\ts};\CI(X))$, we write the Taylor expansion at $\rho=0$,
\begin{equation}
\label{eqn:a a0 atilde}
    a(\ts,x) = a_0(\ts,\omega) + \rho\tilde{a}(\ts,x),\quad \tilde{a}(\ts,x) = \int_0^1 (\pa_\rho a) (\ts,(s\rho)^{-1}\omega)\,ds.
\end{equation}
It follows that $\tilde{a}$ satisfies the same estimate \eqref{a(t*,x) estimate} as $a$ does. In particular, we have
\begin{equation}
\label{eqn:a a0}
    a(\ts,x) - a_0(\ts,\omega) \in \CI([0,\infty)_{\ts};\A^{1}(X)).
\end{equation}
In view of \eqref{rad field 1st}, using \eqref{initial data u0, u1}, \eqref{eqn:tilde g, g00}, \eqref{eqn:a a0} and Lemma \ref{lem:rewritingBoxg}, we obtain a more precise version of \eqref{eqn:Q0 phi full} as follows, noticing that $\rho \pa_\rho -1$ annihilates $\rho$,
    \begin{equation}
		\label{eqn:Q0 phi 2nd}
		\begin{split}
			(\rho\pa_{\rho}-1)\phi(t_*,x) &\in
   \frac{\rho^2}{2}\int_0^{\ts} (\Delta_\omega\phi_{\rm{rad},1}(s,\omega)-a_0(s,\omega)\phi_{\rm{rad},1}^3(s,\omega)) \,ds + 
   \rho^2 c_2(\omega) \\
			&+ \frac{\rho^2 \ti g(\omega)}{2}(d_1(\omega) - \pa_{\ts}\phi_{\rm{rad},1}(\ts,\omega)) +   \CI([0,\infty)_{\ts};\A^{3}(X)).
		\end{split}
	\end{equation}
Noting that $(\rho \pa_{\rho} - 1)\rho^2 = \rho^2$, we can invert \eqref{eqn:Q0 phi 2nd} using Proposition \ref{prop:bODE0} to conclude
	\begin{equation*}
	    \phi(\ts,x) \in \rho\phi_{\rm{rad},1}(\ts,\omega) + \rho^2\phi_{\rm{rad},2}(\ts,\omega) + \CI([0,\infty)_{t_*}; \A^{3}(X)).
	\end{equation*}
In sum, if one writes $\phi = \rho \phi_{\mathrm{rad},1}+r_2$, then Prop. \ref{prop:bODE0} is a gain-of-decay lemma for the remainder term once we have (1) used Prop. \ref{prop:bODEconormal} to extract the kernel mode, and (2) built a particular $\rho^2$-term to match the $\rho^2$ forcing in $Q_0r_2$. Given (1) and (2), Prop.~\ref{prop:bODE0} upgrades the residual from $\mathcal A^2$ to $\mathcal A^3$.

One sees that the second radiation field $\phi_{\rm{rad},2}\in \CI([0,\infty)_{t_*};\CI(\S^2))$ is given explicitly by
	\begin{equation*}
		\begin{split}
			\phi_{\rm{rad},2}(\ts,\omega) &= c_2(\omega) + \frac{1}{2}\ti g(\omega) ( d_1(\omega) - \pa_{\ts}\phi_{\rm{rad},1}(\ts,\omega) ) \\ 
			&\quad\, + \frac{1}{2}\int_0^{\ts} (\Delta_\omega\phi_{\rm{rad},1}(s,\omega)-a_0(s,\omega)\phi_{\rm{rad},1}^3(s,\omega)) ds.
		\end{split}
	\end{equation*} 
\end{proof}

To study the behavior of $\phi$ near the corner $\{\rho=\ts^{-1}=0\}$ (see Figure \ref{fig:M+}), especially the decay of the first radiation field $\phi_{\rm{rad},1}(\ts,\omega)$ in $\ts$, we will prove that $\phi_{{\rm rad},1}$ allows a partial expansion $\phi_{\rm{rad},1} - f(\omega) \tau \in \A^{2-}([0,1)_\tau)$, $\tau = t_*^{-1}$, for some $f(\omega)\in\CI(\S^2)$, in the following lemma: 

\begin{lemma}
	\label{lem:rad field large time}
 Denote by $\p$ the solution of \eqref{I.V. problem}. Then near $\mathscr I^{+}\cap\{\ts>1\}$, we have
	\begin{equation}
		\label{radiation field 1}
		\begin{gathered}
			\phi_{\rm{rad},1}(\ts,\omega)\in \CI(\S^2;\A^1([0,1)_\tau),\quad \tau := \ts^{-1}, \\
			\phi - \rho \phi_{\rm{rad},1}(\ts,\omega) \in \rho^2 \A^0 ([0,1)_\rho \times [0,1)_\tau \times \mathbb{S}^2),
		\end{gathered}
	\end{equation}
where the $\A^0$ conormality is with respect to b-vector fields $\rho\pa_\rho$, $\tau\pa_\tau$, and $\Omega_j$. 
\end{lemma}

\begin{proof}
We %
follow the proof of \cite[Theorem 3.9]{Hin22}. Let
	\[
	v = \rho_{\mathscr I^{+},1} = \rho \ts,\quad \tau = \rho_{I^+,1} = \ts^{-1},
	\]
	near $\mathscr I^{+}\cap\{\ts>1\}\subset M_+$, so that
	\[
	\partial_{\ts} = \tau (-\tau\partial_\tau + v\partial_v),\quad \rho\partial_\rho = v\partial_v.
	\]
    First, we work in a triangular subset of $M_+$, 
    \[
        M_0' := [0,1)_v \times [0,1)_\tau \times \partial X.
    \]
    Then by Lemma \ref{lem:rewritingBoxg} we have
	\[
	\Box_g = 2v\tau^2 (A+R),\quad A = (\tau\partial_\tau - v\partial_v)(v\partial_v - 1),\quad R\in v\Diffb^2(M_0')
	\]
	Noting that $\ts\partial_{\ts} = -\tau\partial_\tau + v\partial_v$ and $\rho\partial_\rho = v\partial_v$, it then follows from \eqref{phi conormal est} that 
	\[
	\p\in \rho\tau \A^{0}(M_0') = v\tau^2 \A^{0}(M_0') = \A^{1,2}(M_0').
	\] 
	Therefore,
	\[
	A\p = (\tau\partial_\tau - v\partial_v)(v\partial_v - 1)\p = (2v\tau^2)^{-1} a \p^3 - R\p \in \A^{2,2}(M_0').
	\]
	Integrating $\tau\partial_\tau - v\partial_v$ along its characteristics gives
	\[
	[(v\partial_v - 1)\p](v,\tau) = [(v\partial_v - 1)\p](1/2,2v\tau) + \A^{2,2}(M_0').
	\]
	Note that $\p\in \A^{1,2}(M_0')$ implies that
	\[
	[(v\partial_v - 1)\p](1/2,2v\tau)\in (v\tau)^2 \A^{0}(M_0') = \A^{2,2}(M_0'),
	\]
	thus we have $(v\partial_v - 1)\p \in \A^{2,2}(M_0')$. Proposition \ref{prop:bODEconormal} then implies $\p \in \A^{((1,0),2),2}(M_0')$. More precisely, we have
	\begin{equation}
		\label{eqn:phi with v leading}
		\phi \in v g_1(\tau,\omega) + \A^{2,2}(M_0'),\quad g_1\in\CI(\S^2;\A^2([0,1)_\tau)).
	\end{equation} 
We can reinterpret \eqref{eqn:phi with v leading} using the relations: $v=\rho\ts$, $\tau=\ts^{-1}$, and $M_0' = ([0,1)_\rho \times [0,1)_\tau \times \S^2) \cap \{\rho<\tau\}$,
	\begin{equation}
		\label{eqn:phi with v leading new}
		\textrm{For }\rho<\tau,\quad \phi - \rho \ts g_1(\tau,\omega) \in \rho^2 \A^0([0,1)_\rho \times [0,1)_\tau \times \S^2).
	\end{equation}
	By comparing the newly derived expression in \eqref{eqn:phi with v leading new} with the previously obtained second-order expansion of the solution in \eqref{radiation field 2nd order}, we can draw two conclusions. First, we have
	\begin{equation}
		\label{1=t*g1}
		\phi_{\rm{rad},1} = \ts g_1
        \text{  where } \ts g_1 \in \CI(\S^2;\A^1([0,1)_\tau)),
        \end{equation}
	and second, 
	\begin{equation}
		\label{On triangle 1}
		\phi - \rho	\phi_{\rm{rad},1} \in \rho^2 \A^0 \left( \Big( [0,1)_\rho \times [0,1)_\tau \times \mathbb{S}^2 \Big) \cap \{\rho<\tau\} \right).
	\end{equation}
However, on the region $\rho\geq \tau$, $\rho\tau\A^0_{\rho,\tau,\omega} \subset \rho^2 \A^0_{\rho,\tau,\omega}$. Recalling \eqref{phi conormal est} that $\phi \in \rho\tau \A^0 ([0,1)_\rho \times [0,1)_\tau \times \mathbb{S}^2)$, we therefore have
	\begin{equation}
		\label{On traingle 2}
		\phi - \rho	\phi_{\rm{rad},1} \in \rho\tau \A^0 \subset \rho^2 \A^0 \left( \Big( [0,1)_\rho \times [0,1)_\tau \times \mathbb{S}^2 \Big) \cap \{\rho\geq\tau\} \right).
	\end{equation}
    Eq. \eqref{On triangle 1} and eq. \eqref{On traingle 2} together imply that the subleading term of $\phi$ possesses $\rho^2 \A^0$ conormality on the entire space.
	Combining \eqref{1=t*g1}, \eqref{On triangle 1} and \eqref{On traingle 2}, we obtain \eqref{radiation field 1}.
\end{proof}

We are ready to state our main theorem for the long-time asymptotics of $\phi$ in spatially compact sets. 

\begin{theorem}\label{thm:compactmain}
Let $\phi$ be the unique forward solution to the initial value problem \eqref{I.V. problem}, with initial data $(u_0,u_1)$ satisfying \eqref{initial data u0, u1} and with solution $\phi$ satisfying conormal bounds \eqref{phi conormal est}. Then $\phi$ satisfies
\begin{equation} 
	|\p(\ts,x) - 2c_0 \ts^{-2}| \leq C_K \ts^{-3+}, \quad x \in K
\end{equation}
for a fixed compact set $K \subset X^\circ$ (the set $X^\circ$ is equal to $\R^3$), where
\begin{equation}
\label{eq:leadcoefficient}
    c_0 = \frac{1}{4\pi}\int_{\S^2} \left(\int_0^\infty a_0(\ts,\omega)\phi_{\rm{rad},1}^3(\ts,\omega) d\ts - 2c_2(\omega)-\ti{g}(\omega)d_1(\omega)\right) d\omega ,
\end{equation}
with $\phi_{\rm rad,1}$ given in \eqref{radiation field 2nd order}, $c_2, d_1$ given in \eqref{initial data u0, u1}, $\tilde{g}$ given in \eqref{eqn:tilde g, g00}, and $a_0$ given in \eqref{eqn:a a0}. 
Moreover, the derivatives of $\phi$ decay as follows for $x\in K$,
\begin{equation}
\label{phi asymp spatial compact}
	|\pa_{\ts}^j \pa_x^\beta (\phi(\ts,x)-2c_0\ts^{-2})| \leq C_{j\beta,K} \ts^{-3-j+},\quad\forall\,j\in\NN,\ \beta\in\NN^3.
\end{equation}
\end{theorem}

\begin{remark}
\label{rmk:general intial data p=3}
In Theorem \ref{thm:compactmain} we treat more general initial data than the compactly supported smooth initial data assumed in Theorem \ref{thm:compactmain intro}. More specifically, we assume that the initial data $u_0$ and $u_1$ lie in certain conormal spaces that allow partial expansions, namely $\A^{((1,0),3)}(X)$ and $\A^{((1,0),2)}(X)$, respectively, see \eqref{initial data u0, u1} and Definition \ref{defn:partial expansion}. The compactly supported smooth $u_0$ and $u_1$ satisfy this assumption trivially, by noticing that $u_0,\,u_1\in\CcI(\RR^3)$ can be expanded as in \eqref{initial data u0, u1} with $c_1=c_2=d_1=0$.

For the Schwarzschild metric, one simply has $\tilde{g}(\omega)=-2m^2$, where $m$ is the mass, in the formula \eqref{eq:leadcoefficient}. For an explicit computation, we refer to \cite[Section 4.1]{Hin22}.
\end{remark}

We convert the initial value problem to a forcing problem by introducing a cutoff function. The forcing problem allows us to decompose its solution into low frequency, high frequency, and intermediate frequency parts. Each part is then analyzed separately using Fourier analysis and resolvent expansion/estimates. By combining the asymptotic expansion for each part, we derive the full asymptotic expansion. 

\subsection{Preliminaries for the proof of Theorem \ref{thm:compactmain}}
\label{subsection:Preliminaries for spatially compact asymp}
In order to convert the initial value problem \eqref{I.V. problem} to a forcing problem, we fix a cutoff $\chi(\ts)\in \CI(\RR_{\ts};\RR)$ which is identically equal to $0$ when $\ts\leq 1/2$ and identically equal to
$1$ when $\ts\geq 1$. Then $\phi_\chi := \chi\phi$ vanishes for $\ts\leq 1/2$ and solves the following forcing problem 
\begin{equation}
\label{forcing problem}
	\Box_g \phi_\chi = [\Box_g , \chi(\ts)]\phi + \chi(\ts)a(\ts,x)\phi^3 =: f,\quad \textrm{on }\RR_{\ts}\times X^\circ.
\end{equation}
Here $f$ can be computed, in view of Lemma \ref{lem:rewritingBoxg},
\begin{equation}
\label{the forcing f}
    f = \chi a\phi^3 - 2\chi'(\rho Q + g^{00}\pats)\phi - \chi'' g^{00} \phi,
\end{equation}
all terms on the right-hand side other than the first are compactly supported in $\ts$. Combining \eqref{a(t*,x) estimate}, \eqref{phi conormal est}, \eqref{eqn:Q0 phi A2} and the assumptions on $\chi(\ts)$, and recalling that $Q=(\rho\pa_{\rho}-1)+\rho^2 \widetilde{Q}$, $g^{00}\in \rho^2\CI(X)$, we conclude that for any $k,m\in\mathbb{N}$,
\begin{equation}
\label{eqn:forcing f decay}
	\pa_{t_*}^k(\ts\pa_{t_*})^m \Gamma_b^I f(\ts,x) = \mathcal{O}(\la t_*\ra^{-3-k} \la t_*+2r\ra^{-3}).
\end{equation}
It follows that $\wh{f}(\sigma,x) = \int_\RR e^{i\sigma \ts} f(\ts,x) d\ts$ is well-defined. For future reference, we note that for any $k,m\in\mathbb{N}$,
\[
	\sigma^k (\sigma\pa_\sigma)^m \wh{f}(\sigma,x) = i^k (-1)^m \int_\RR e^{i\sigma\ts} \pa_{\ts}^k (\pa_{\ts} \ts)^m f(\ts,x) d\ts.
\]
Therefore, using \eqref{eqn:forcing f decay}, we can deduce that
\begin{equation}
\label{hat f rapid decay}
	\forall k,m\in\mathbb{N},\quad \sigma^k(\sigma\pa_\sigma)^m \widehat{f}(\sigma,x) \in L^{\infty}(\R_\sigma; \A^3(X)).
\end{equation}
Here $\A^3(X)$ is equipped with all the seminorms $\|\cdot\|_{j,3}$, $j\in\NN$, defined as follows
\begin{equation}
\label{seminorms on A ell}
	\|u\|_{j,\ell} := \sum_{q} \|\rho^{-\ell} A_{jq}u\|_{L^\infty(X)},\quad u\in \A^\ell(X),
\end{equation}
where the $A_{jq}$'s form a finite basis of $\Diffb^j(X)$ over $\CI(X)$. Thus \eqref{hat f rapid decay} means that every seminorm $\|\sigma^k(\sigma\pa_\sigma)^m \widehat{f}(\sigma,x)\|_{j,3}$ is uniformly bounded in $\sigma\in\RR$.

Similarly, we can define $\wh{\phi_\chi} (\sigma,x) := \int_\RR e^{i\sigma\ts} \chi(\ts)\phi(\ts,x) d\ts$, where we used the fact that $\phi_\chi \in L^1(\RR_{\ts};\A^{1-}(X))$ due to \eqref{phi conormal est}. Moreover, we note that for any $N\in\NN$,
\[
	\sigma^N \wh{\phi_\chi} (\sigma,x) = i^N \int_\RR e^{i\sigma\ts} \pa_{\ts}^N (\chi(\ts)\phi(\ts,x)) d\ts.
\]
Thus, we deduce from \eqref{phi conormal est} that
\[
	\la\sigma\ra^N \wh{\phi_\chi} (\sigma,x) \in L^\infty(\RR_{\sigma};\A^{1-}(X)),
\]
which shows that $\wh{\phi_\chi} \in L^1(\RR_{\sigma};\A^{1-}(X))$. Therefore, by the Fourier inversion theorem, 
\begin{equation}
\label{FT*IFT=1}
	\phi_\chi(\ts,x) = \frac{1}{2\pi} \int_\RR e^{-i\sigma\ts} \wh{\phi_\chi}(\sigma,x) d\sigma.
\end{equation}

We recall the spectral family of $\Box_g$ given in \eqref{eqn:Boxhatsigma}:
\[
	\wh\Box (\sigma) = e^{i\sigma \ts}\Box_g e^{-i\sigma \ts} = 2i\xs\rho Q + \hat\Box(0) + \xs^2g^{00}.
\]
It then follows from \eqref{forcing problem} that 
\[
	\wh\Box (\sigma) \wh{\phi_\chi} (\sigma,x) = \wh{f}(\sigma,x).
\]  
Recalling from Lemma \ref{lem:resolvent_regularity} and Theorem \ref{thm:low_energy} the mapping properties of the resolvent $\wh\Box (\sigma)^{-1}$, $\Im\sigma\geq 0$, we get
\[
	\wh{\phi_\chi} (\sigma,x) = \wh\Box (\sigma)^{-1} \wh{f}(\sigma,x).
\]
Applying this to \eqref{FT*IFT=1}, we conclude that
\begin{equation}
\label{phichi by IFT}
	\phi_\chi (t_*,x) = \frac{1}{2\pi}\int_{\R} e^{-i\sigma t_*} \wh\Box(\sigma)^{-1} \widehat{f}(\sigma,x) d\sigma.
\end{equation}
We  now follow \cite[Theorem 3.4]{Hin22} to split $\phi_\chi=\phi_0 + \phi_1 + \phi_2$,
\begin{equation}
	\label{eqn:phi012}
	\begin{split}
		\widehat{\phi_0}(\sigma) &:= \psi(\sigma)\wh\Box(\sigma)^{-1} \wh{f}(0), \\
		\widehat{\phi_1}(\sigma) &:= \sigma\psi(\sigma)\wh\Box(\sigma)^{-1}(\sigma^{-1}(\wh{f}(\sigma)-\wh{f}(0))), \\
		\widehat{\phi_2}(\sigma) &:= (1-\psi(\sigma))\wh\Box(\sigma)^{-1}\wh{f}(\sigma),
	\end{split}
\end{equation}
where $\psi\in \CcI((-1,1)_\sigma)$ is identically equal to $1$ for $\sigma$ near $0$, and we require $\psi$ to be even.

\subsection{Proof of Theorem \ref{thm:compactmain}} 
Using the decomposition \eqref{eqn:phi012}, we shall first show that the high frequency part $\phi_2$ has rapid decay; we then consider the low frequency part $\phi_0$, where by using a low-energy expansion we show that $\phi_0$ contributes to the leading term $2c_0\ts^{-2}$; we finally illustrate that the piece $\phi_1$ has a higher order of $\ts$ decay than $\phi_0$ and contributes to the remainder $\mathcal{O}(\ts^{-3+})$.

\subsubsection{High frequency contribution} 
\label{subsubsection:high frequency}
For the high frequency part $\phi_2$, we work explicitly with its Fourier transform $\hat\phi_2(\sigma)$.
We assert that
\begin{equation}
\label{phi2 rapid decay}
    \phi_2(t_*,x) \in \mathscr{S}(\mathbb{R}_{t_*};\A^{1-}(X)).
\end{equation}
We remark that in contrast to the case of forcing terms $f$ which are $\CI_c$ in $t_*$, we cannot simply claim that $\hat\phi_2$ is Schwartz in $\sigma$. Thus our argument for the high frequency piece $\phi_2$ will be significantly more involved than in \cite{Hin22}.
We will now compute, for $k,m\in\mathbb{N}$,
\[
	\pa_\sigma^m ((1-\psi(\sigma))\wh\Box(\sigma)^{-1}(\sigma^k\wh{f}(\sigma))).
\]
We see that it splits into terms of the form
\begin{equation}
\label{split integrand}
   (\partial_\sigma^{m_1}(1-\psi))(\partial_\sigma^{m_2} \wh\Box(\sigma)^{-1}) \pa_\sigma^{m_3} (\sigma^k\wh{f}),\quad m_1+m_2+m_3 = m.
\end{equation}
We recall from \eqref{boxhat resolv semiclassical bounds} that for $s,\ell,\sigma$ satisfying
\begin{equation}
	\label{eqn:s,ell,sigma}
	s\in\RR,\quad \ell<-\frac{1}{2},\quad s+\ell>-\frac{1}{2},\quad |\sigma|\geq R,
\end{equation} 
there exists a uniform constant $C>0$ such that
\begin{equation}
\label{deriv resolvent bounds}
    \|\partial_\sigma^m \wh\Box(\sigma)^{-1} f\|_{\Hb^{s-m,\ell}(X)} \leq C|\sigma|^{s-(m+1)(1-\delta)} \|f\|_{\Hb^{s,\ell+1}(X)}.
\end{equation}
We used $h^k \|u\|_{H^k_{\rm b}} \le \|u\|_{H^k_{{\rm b}, h}}$.
For every $k,m,s\in\NN$, we choose $N$ sufficiently large and rewrite terms in \eqref{split integrand}, using the fact that $\sigma$ commutes with the differentiated resolvent $\partial_\sigma^m \wh\Box(\sigma)^{-1}$, 
\begin{equation}
\label{split integrand 2}
    (\partial_\sigma^{m_1}(1-\psi))\sigma^{-N}(\partial_\sigma^{m_2} \wh\Box(\sigma)^{-1}) \sigma^{N}\pa_\sigma^{m_3} (\sigma^k \wh{f}).
\end{equation}
By a direct computation, we see that for $N$ sufficiently large,
\[
	\sigma^{N}\pa_\sigma^{m_3} (\sigma^k \wh{f}) = \sigma^{N+k-m_3} \sum_{q\leq k,m_3} c_q \sigma^{m_3-q} \pa_\sigma^{m_3-q} \wh{f},
\]
noting that for any $m\in\NN$,
\[
    \sigma^m \pa_\sigma^m = (\sigma\pa_\sigma)^{m} + a_1 (\sigma\pa_\sigma)^{m-1} + \cdots + a_{m-1}\sigma\pa_\sigma,\quad a_1,\ldots, a_{m-1}\in\mathbb{Z},
\]
Applying the above formulas with \eqref{hat f rapid decay} and \eqref{deriv resolvent bounds}, recalling from \eqref{EqASobolevEmb} that $\A^3(X)\subset \A^2(X) \subset \Hb^{s,\ell+1}(X)$ for any $s,\ell$ satisfying \eqref{eqn:s,ell,sigma}, we get
\[
    \|\sigma^{-N}(\partial_\sigma^{m_2} \wh\Box(\sigma)^{-1}) \sigma^{N}\pa_\sigma^{m_3} (\sigma^k \wh{f})\|_{\Hb^{s-m_2,\ell}(X)} \leq C|\sigma|^{s-N}.
\]
Here we note that $|\xs| \ge R > 0$. Combining this with \eqref{split integrand 2}, and noting that $N$ can be chosen sufficiently large, we obtain
\[
	\|\pa_\sigma^m ((1-\psi(\sigma))\wh\Box(\sigma)^{-1}(\sigma^k\wh{f}(\sigma)))\|_{\Hb^{s-m,\ell}(X)} \leq C|\sigma|^{s-N}.
\]
This justifies the following integration by parts. Using the fact that $\sigma$ commutes with the differentiated resolvents $\partial_\sigma^m \wh\Box(\sigma)^{-1}$, $m\in\NN$, we compute directly:
\begin{equation}
\label{t*4 phi2 expression}
(it_*)^m (i\partial_{\ts})^k \phi_2 = \frac{1}{2\pi}\int_{\R} e^{-i\sigma t_*} \pa_\sigma^m ((1-\psi(\sigma))\wh\Box(\sigma)^{-1}(\sigma^k\wh{f})) d\sigma.
\end{equation}
Therefore, we conclude that for every $k,m\in\mathbb{N}$,
\[
    \|t_*^m \pa_{t_*}^k \phi_2\|_{\Hb^{s-m,\ell}(X)} \leq C,\quad\textrm{$\forall s,\ell$ satisfying \eqref{eqn:s,ell,sigma}}.
\]
This, with \eqref{EqASobolevEmb} proves \eqref{phi2 rapid decay}, since $s$ can be arbitrarily large in the sense that for any $\epsilon>0$, $t_*^m \pa_{t_*}^k \phi_2\in\A^{1-\epsilon}(X)$ has all bounded seminorms $\|t_*^m \pa_{t_*}^k \phi_2\|_{j,1-\epsilon}$, $j\in\NN$; see the definition in \eqref{seminorms on A ell}. The $\eps$ equals the gap between $\ell$ and $-\frac12$, $\eps = |\ell -(-\frac12)|$. 

\subsubsection{Low frequency part $\phi_0$}
\label{subsubsection:low frequency part}
We begin with a more precise description of 
\[
	\wh{f}_0 (x) := \wh{f}(0,x) = \int_\RR f(\ts,x) d\ts,
\]
than $f_0\in \A^3(X)$ which follows from \eqref{hat f rapid decay}.
Lemma \ref{lem:rad field large time} shows that $\phi_{{\rm rad},1} (\ts,\omega) = \mathcal{O}(\ts^{-1})$ and that 
\begin{equation}
\label{tilde phi decay}
	\phi-\rho\phi_{{\rm rad},1} \in \rho^2 \A^0 ([0,1)_{\rho,\tau}^2 \times \mathbb{S}^2) \cap \rho\tau \A^0 ([0,1)_{\rho,\tau}^2 \times \mathbb{S}^2),
\end{equation}
with the $\rho\tau\A^0$ estimate coming from \eqref{phi conormal est} and the first estimate for $\phi_{{\rm rad},1}$ in \eqref{radiation field 1}.  Then
\begin{equation}
	\label{phicube explicit}
	\begin{split}
		\phi^3 &= \rho^3\phi_{\rm rad,1}^3 + 3\rho^2\phi_{\rm rad,1}^2(\phi-\rho\phi_{{\rm rad},1}) + 3\rho\phi_{\rm rad,1}(\phi-\rho\phi_{{\rm rad},1})^2 + (\phi-\rho\phi_{{\rm rad},1})^3.
	\end{split}
\end{equation}
It follows from \eqref{tilde phi decay} and \eqref{eqn:a a0 atilde}, \eqref{eqn:a a0} that
\begin{equation}
\label{phicube expansion}
	a(\ts,x) \phi^3 - a_0(\ts,\omega)\rho^3 \phi_{\rm rad,1}^3(\ts,\omega) \in  \A^2([0,1)_{\tau} ; \A^4(X)),\quad \tau=t_*^{-1}.
\end{equation}
Next, we consider the remaining terms in \eqref{the forcing f}, which are all supported in $0<\ts<1$. Using the second-order expansion \eqref{radiation field 2nd order}, together with \eqref{eqn:tilde g, g00}, 
\begin{equation}
	\begin{split}
		f - \chi a\phi^3 &= -2\chi'\rho Q\phi - 2\chi' g^{00}\pa_{\ts}\phi - \chi''g^{00}\phi \\
		&\in -\rho^3(2\chi' \phi_{\rm{rad},2} + 2\chi' \ti{g} \pa_{\ts}\phi_{\rm{rad},1} + \chi'' \ti{g}\phi_{\rm{rad},1}) + \CcI((0,1)_{\ts};\A^4(X)).
	\end{split}
 \label{eq:f expansion}
\end{equation}
Even though we only proved $\phi_{{\rm rad},2} \in \A^0_\tau$, the $-2\rho^3\chi'\phi_{{\rm rad},2}$ term is compactly supported in $t_*$, thus integrable on $\RR_{\ts}$. We have in view of \eqref{phicube expansion}, 
 \begin{equation*}
        \int f(\ts,x) \,d\ts \in  \rho^3 \int (\chi a_0 \phi^3_{ {\rm rad},1} -2\chi' \phi_{{\rm rad},2}
        - 2\chi' \ti g \pats \phi_{{\rm rad},1} - \chi''\ti g \phi_{{\rm rad},1}) \,d\ts + \A^4(X) .
\end{equation*} 
We conclude that
\begin{equation}
\label{eqn:fhat}
    \wh{f}_0(x) = \int_\RR f(\ts,x) d\ts \in c(\omega)\rho^3 + \A^4(X),
\end{equation}
where, in view of $\chi' \pa_{\ts}\phi_{\rm{rad},1} + \chi'' \phi_{\rm{rad},1} = \pa_{\ts} (\chi' \phi_{\rm{rad},1})$, we have 
\begin{equation}
\label{eqn:c(omega)}
\begin{split}
	c(\omega) &:= \int_\RR (\chi a_0 \phi_{\rm rad,1}^3 - 2\chi' \phi_{\rm{rad},2} - 2\chi' \ti{g} \pa_{\ts}\phi_{\rm{rad},1} - \chi'' \ti{g}\phi_{\rm{rad},1}) d\ts \\
	&=  \int_\RR (\chi a_0 \phi_{\rm rad,1}^3 - 2\chi' \phi_{\rm{rad},2} - \chi' \ti{g} \pa_{\ts}\phi_{\rm{rad},1}) d\ts
\end{split}
\end{equation}since by the fundamental theorem of calculus and the compact support of $\chi'(\ts)$, the integral $$\int \pats(\chi'(\ts) \phi_{{\rm rad},1}(\ts,\xo)) \,d\ts$$ vanishes. We remark that the d$t_*$ integral of an element of $\cA^2([0,1)_\tau)$ is bounded, since $\A^2([0,1)_\tau)$ corresponds to $\mathcal{O}(\ts^{-2})$-decay. Thus we deduce the $\A^4(X)$ term in \eqref{eqn:fhat} from \eqref{phicube expansion}. 

To derive a low energy expansion for $\wh\Box(\sigma)^{-1} \wh{f}_0$,
we use the following resolvent identity:
\begin{equation}
\label{resolvent identity: cubic}
    \wh\Box(\sigma)^{-1} \wh{f}_0 = \wh\Box(0)^{-1} \wh{f}_0 - \sigma \wh\Box(\sigma)^{-1} [\sigma^{-1} (\wh\Box(\sigma) - \wh\Box(0)) \wh\Box(0)^{-1} \wh{f}_0].
\end{equation}
This identity can also be equivalently written without the $1=\xs \xs^{-1}$ factor, and the above formulation views the difference quotient as a sort of derivative especially for small $\sigma$. We shall use the following notation for the first term in the expansion of the resolvent:
\[
    u:= \wh\Box(0)^{-1}\wh{f}_0,\quad \wh\Box(0) u = \wh{f}_0 \in c(\omega)\rho^3 + \A^4(X).
\]
Recalling the decomposition of $\Box_g$ in Lemma \ref{lem:rewritingBoxg}, we have
\begin{equation}
\label{eqn:L0u}
    L_0 u \in c(\omega) \rho + \A^{2-}(X), \quad L_0 = -(\rho\pa_\rho)^2 + \rho\pa_\rho + \Delta_\omega
\end{equation}
for which we used the fact that $u=\wh\Box(0)^{-1}\wh{f}_0\in \A^{1-}(X)$ due to \eqref{EqAZeroMapping}, as $\wh{f}_0 \in \A^3(X)$ (indeed, $\wh{f}_0 \in \A^{3-}(X)$ suffices). Following the same computation as in \cite[Lemma 3.3]{Hin22}, we can conclude from \eqref{eqn:L0u} that
\begin{equation}
\label{eqn:expansion of u}
    u = -c_0 \rho\log\rho + \ti u,\quad \ti u\in \A^{((1,0),2-)}(X),
\end{equation}
where the leading term of $\ti u$ is $\rho Y(\xo)$ for some $Y\in\CI(\mathbb{S}^2)$, and
\begin{equation}
\label{eqn:coefficient c0}
    c_0 = \frac{1}{4\pi} \int_{\S^2} c(\omega)\, d\omega .
\end{equation} 
(In fact, to achieve \eqref{eqn:expansion of u}, one only needs  $\wh\Box(0)u - c(\omega)\rho^3 \in \A^{4-}(X)$, as in \cite[Lemma 3.3]{Hin22}.) In view of \eqref{eqn:expansion of u} and \eqref{eqn:Boxhatsigma}, with the latter being an explicit decomposition of $\wh\Box(\xs)$ into a polynomial in $\xs$ with coefficients being operators, we can compute the next term in \cref*{resolvent identity: cubic},
\begin{equation}
\label{eqn:f_1(sigma)}
    \begin{gathered}
        f_1(\sigma) := -\sigma^{-1}(\wh\Box(\sigma) - \wh\Box(0)) u = (-2i\rho(Q_0 + \rho^2\widetilde{Q})-g^{00}\sigma)u \\
        \Rightarrow  f_1(\sigma) = 2ic_0 \rho^2 + i\ti{f_1}(x) - \sigma g^{00} u,\quad  \ti{f_1}\in \A^{3-}(X;\R).
    \end{gathered}
\end{equation} 
Here $\ti{f_1}$ is real-valued because $\ti u$ is real-valued. We have also used the fact that $Q_0 \rho = 0$ and thus $Q_0=\rho\pa_\rho - 1$ eliminates the leading term in $\ti u$. So we rewrite \eqref{resolvent identity: cubic},
\begin{equation}
\label{resolvent full expansion}
    \wh\Box(\sigma)^{-1} \wh{f}_0 = u + 2ic_0\sigma\wh\Box(\sigma)^{-1}(\rho^2) + i\sigma\wh\Box(\sigma)^{-1}\ti{f_1} + \sigma^2 \wh\Box(\sigma)^{-1}(g^{00}u).
\end{equation}
Recalling Proposition \ref{Hintz Lemma 2.24}, we have
\[
    2ic_0 \sigma\wh\Box(\sigma)^{-1}(\rho^2) \in \sigma\A^{1-,((0,0),1-),((0,1),1-)}(X^+_{\rm res}),
\]
whereas the remaining terms are more regular in $\sigma$. In view of Lemma \ref{lem:aux}, which computes the target space of $\wh\Box(\xs)^{-1}$ on $\A^\beta([0,1)_\xs;\A^{2+\alpha}(X))$, $\beta \in \RR$, $\alpha\in (0,1)$ inputs, and noting that $\ti{f_1}, g^{00}u\in \A^{3-}(X)$, we have
\[
    \begin{gathered}
        i \sigma\wh\Box(\sigma)^{-1} \ti{f_1} \in \sigma \A^{1-, 1-, ((0,0),1-)}(X^+_{\rm res}); \\
        \sigma^2\wh\Box(\sigma)^{-1} (g^{00}u)\in \sigma^2 \A^{1-, 1-, ((0,0),1-)}(X^+_{\rm res})\subset \A^{1-,2-,2-}(X^+_{\rm res}).
    \end{gathered}
\]
Applying the resolvent identity \eqref{resolvent identity: cubic} to $\ti{f_1}(x)$, with the help of \cite[Lemma 2.17]{Hin22} , we have more explicitly
\[
    \sigma\wh\Box(\sigma)^{-1} \ti{f_1} \in \sigma \wh\Box(0)^{-1} \ti{f_1} + \A^{1-,2-,2-}(X^+_{\rm res})
\]
where $\wh\Box(0)^{-1} \ti{f_1} \in \A^{1-,1-,1-}(X^+_{\rm res};\R)$ is real-valued. To see this, noting that $\wh\Box(0)^{-1} : \A^{3-}(X) \to \A^{1-}(X)$, we can write
\[
\begin{split}
	\sigma\wh\Box(\sigma)^{-1} \ti{f_1} - \sigma \wh\Box(0)^{-1} \ti{f_1}
	& = \sigma^2 \wh\Box(\sigma)^{-1}[-\sigma^{-1}(\wh\Box(\sigma) - \wh\Box(0))\wh\Box(0)^{-1} \ti{f_1}] \\
	&\in \sigma^2 \wh\Box(\sigma)^{-1} (\A^{2-}(X)+\sigma\A^{3-}(X)).
\end{split}
\]
Recalling the mapping property $\wh\Box(\xs)^{-1}:\A^{2-}\to \A^{1-,0-,0-}$ from \cite[Lemma 2.17]{Hin22}, we conclude that we are done since $\xs^2\A^{1-,0-,0-}= \A^{1-,2-,2-}$.

To conclude, we have obtained an expansion for frequencies $\sigma>0$,
\begin{equation}
\label{decompose u sing and reg}
    \wh\Box(\sigma)^{-1} \wh{f}_0 = u_{\rm{sing}}(\sigma) + u_{\rm{reg}}(\sigma),
\end{equation}
\begin{equation}
\label{u singular}
    u_{\rm{sing}}(\sigma) \in \sigma\A^{1-,((0,0),1-),((0,1),1-)}(X^+_{\rm res});
\end{equation}
\begin{equation}
\label{u regular}
    u_{\rm{reg}}(\sigma) \in \wh\Box(0)^{-1} \wh{f}_0 + i\sigma \wh\Box(0)^{-1}\ti{f_1} + \A^{1-,2-,2-}(X^+_{\rm res}).
\end{equation}
The leading term of $\sigma^{-1} u_{\rm{sing}}(\sigma)$ at zf is $-2ic_0 \log(\sigma/\rho)$, and at tf it is $2ic_0 \ti u_{\rm mod}$, with $\ti u_{\rm mod} = \widetilde\Box^{-1}(\sigma^2/\rho^2)$ defined in Lemma \ref{lem:2.23Hintz}. We also remark that $\wh\Box(0)^{-1} \wh{f}_0$, $\wh\Box(0)^{-1} \ti{f_1}$ are both real-valued; this can be seen by \eqref{resolvent negative sigma 0}, which implies that for real inputs $g$, the zero-energy inverse is real-valued.

To prove Theorem \ref{thm:compactmain}, we restrict our attention to spatially compact sets $K\Subset X^\circ$ i.e. where $\rho>c>0$ for some constant $c$. This will allow us to take all terms in the resolvent expansion \eqref{decompose u sing and reg}--\eqref{u regular} and replace the conormal space $\A^{1-}(X)$ with $\CI(X^\circ)$. We note that for real $\sigma$,
\begin{equation}
\label{resolvent negative sigma 0}
    \wh\Box(\sigma)^{-1} g = \overline{\wh\Box(-\sigma)^{-1}\overline{g}}.
\end{equation}
It follows that the smooth terms of $u_{\rm{reg}}$ in \eqref{u regular} for $\sigma >0$ and $\sigma<0$ can be glued together as
\[
    \wh\Box(0)^{-1} \wh{f}_0 + i\sigma \wh\Box(0)^{-1}\ti{f_1} \in\CI((-1,1)_\sigma;\CI(X^\circ)),
\]
since 
\[
	\overline{\wh\Box(0)^{-1} \wh{f}_0 + i(-\sigma)\wh\Box(0)^{-1}\ti{f_1}} = \wh\Box(0)^{-1} \wh{f}_0 + i\sigma \wh\Box(0)^{-1}\ti{f_1},
\]
for which we recalled that $\wh\Box(0)^{-1} \wh{f}_0$ and $\wh\Box(0)^{-1}\ti{f_1}$ are both real-valued. Arguing as in \cite[Theorem 3.4]{Hin22} to glue the pieces $u_{\rm{sing}}$ for $\sigma>0$ and $\sigma<0$ together, we conclude that
\begin{equation}
\label{eq:phi0}
\begin{split}
     \widehat{\phi_0}(\sigma)&=\psi(\sigma) \wh\Box(\sigma)^{-1} \wh{f}_0 \in -2ic_0\psi(\sigma)\sigma\log(\sigma+i0)\\ &+\CcI((-1,1)_\sigma;\CI(X^\circ)) +\sum_{\pminus} \A^{2-}(\pminus[0,1)_\sigma;\CI(X^\circ)).
\end{split}
\end{equation}
We now compute $\mathcal{F}^{-1} (\psi(\sigma)\sigma\log(\sigma+i0))(\ts)$. Writing $\log(\sigma+i0) = \log|\sigma| + i\pi H(-\sigma),$
where $H$ is the Heaviside function, we get, through a straightforward computation, 
\[
\begin{split}
    \mathcal{F}^{-1} (\log(\sigma+i0)) (\ts)
    &= -\frac{1}{2} \left( \mathrm{p.v.} \frac{1}{\ts} + \mathrm{f.p.} \frac{1}{|\ts|} \right) + \left( \frac{i\pi}{2} - \gamma \right)\delta(\ts) \\
    &= - \mathrm{f.p.} \frac{1}{(\ts)_+} + \left( \frac{i\pi}{2} - \gamma \right)\delta(\ts).
\end{split}
\]
Here $\mathrm{p.v.}$ denotes the principal value, $\mathrm{f.p.}$ denotes the finite part distribution; $(\ts)_+ = \max(\ts,0)$; $\gamma$ is the Euler's constant. Then
\begin{equation}
\label{IFTsigmalog(sigma+i0)}
\begin{split}
    F(t_*) := \mathcal{F}^{-1} (\sigma\log(\sigma+i0))(\ts) 
    &= i\partial_{\ts} \left(\mathcal{F}^{-1} (\log(\sigma+i0))(\ts)\right) \\
    &= -i \left(\mathrm{f.p.} \frac{1}{(\ts)_+}\right)' - \left( \frac{\pi}{2} + i\gamma \right)\delta'(\ts) .
\end{split} 
\end{equation}
Noting that $F(\ts) \in \mathscr{S}'(\RR_{\ts})$, we can therefore note that $\text{sing supp}\,F = \{0\}$. Moreover, letting $\chi(\ts)\in \CI(\RR_{\ts})$ be identically equal to $0$ when $\ts\leq 1/2$ and identically equal to
$1$ when $\ts\geq 1$, we have, in view of \eqref{IFTsigmalog(sigma+i0)},
\[
    (\chi F)(\ts) = i\ts^{-2} \chi(\ts);\quad \supp (1-\chi)F \subset [0,1]_{\ts}.
\]
Using the convolution theorem, $\mathcal{F}(u\ast v) = \mathcal{F}u \mathcal{F}v$, we can write
\[
\begin{split}
    { }&\quad\ \mathcal{F}^{-1} (\psi(\sigma)\sigma\log(\sigma+i0)) = F \ast \mathcal{F}^{-1}\psi = \chi F \ast \mathcal{F}^{-1}\psi + (1-\chi)F \ast \mathcal{F}^{-1}\psi.
\end{split}
\]
Note that $\psi\in\CcI(\RR_\sigma)$ implies $\mathcal{F}^{-1}\psi\in\mathscr{S}(\RR_{\ts})$, and that $(1-\chi)F$ is a distribution of compact support. It follows that
\[
    (1-\chi)F \ast \mathcal{F}^{-1}\psi \in \mathscr{S}(\RR_{\ts}).
\]
We assert that 
\[
    \chi F - \chi F \ast \mathcal{F}^{-1}\psi \in \mathscr{S}(\RR_{\ts}).
\]
This is equivalent to claim that $(1-\psi)\widehat{\chi F} \in\mathscr{S}(\RR_\sigma)$, i.e. for any $k, m\in\NN$, there exists $C_{km}$ such that
\begin{equation}
\label{bound by Ckm}
	\left|\sigma^m \partial_\sigma^k \big((1-\psi(\sigma))\widehat{\chi F}(\sigma)\big)\right| \leq C_{km} .
\end{equation}
Recalling that $(\chi F)(t_* ) = \chi(t_*) it_*^{-2}$, through integration by parts, we have
\[
	\sigma \widehat{\chi F}(\sigma) = \int_{\RR} \partial_{t_*} (e^{i\sigma t_*}) \chi(t_*) t_*^{-2} d t_* = - \int_{\RR} e^{i\sigma t_*} \partial_{t_*} (\chi(t_*) t_*^{-2}) d t_* .
\]
It follows that for any $k\in\NN$,
\[
	\sigma^k \widehat{\chi F}(\sigma) = i^{k+1} \int_{\RR} e^{i\sigma t_*} \partial_{t_*}^k (\chi(t_*) t_*^{-2}) d t_* \in L^\infty(\RR_\sigma);
\]
here we used the fact that $\partial_{t_*}^k (\chi(t_*) t_*^{-2}) = \mathcal{O}(|t_*|^{-2-k})$. Since 
$$t_*^l \partial_{t_*}^k (\chi(t_*) t_*^{-2}) \in L^1(\RR_{t_*}), \quad l\leq k,$$ 
we can therefore write
\[
	\partial_\sigma^l \big(\sigma^k \widehat{\chi F}(\sigma)\big) = i^{k+l+1} \int_{\RR} e^{i\sigma t_*} t_*^l \partial_{t_*}^k (\chi(t_*) t_*^{-2}) d t_* \in L^\infty(\RR_\sigma).
\]
We recall that $\psi\in\CcI(\RR_\sigma)$ is identically equal to $1$ near $\sigma=0$. Thus, we can define $\sigma^{-k} (1-\psi(\sigma)) \in \CI_b(\RR_\sigma)$ (with all derivatives bounded). The product rule gives
\begin{equation}
\label{eqn:Leibniz schwartz}
	\partial_\sigma^k \big((1-\psi(\sigma))\widehat{\chi F}(\sigma)\big) = \sum_{l=0}^k \binom{k}{l} \partial_\sigma^{k-l}\left(\frac{1-\psi(\sigma)}{\sigma^k}\right) \partial_\sigma^{l} \big(\sigma^k \widehat{\chi F}(\sigma)\big).
\end{equation}
Using integration by parts as above, we can derive that for $m\in \NN$, $l\leq k$,
\begin{equation}
\label{IBP schwartz}
	\sigma^m \partial_\sigma^l \big(\sigma^k \widehat{\chi F}(\sigma)\big) = i^{k+l+m+1} \int_{\RR} e^{i\sigma t_*} \partial_{t_*}^m \big(t_*^l \partial_{t_*}^k (\chi(t_*) t_*^{-2})\big) d t_* \in L^\infty(\RR_\sigma).
\end{equation}
Combining \eqref{eqn:Leibniz schwartz} and \eqref{IBP schwartz}, we obtain \eqref{bound by Ckm}. Thus $(1-\psi)\widehat{\chi F} \in\mathscr{S}(\RR_\sigma)$. We conclude from the above arguments that
\[
    \mathcal{F}^{-1} (\psi(\sigma)\sigma\log(\sigma+i0)) - i \chi(\ts)t_*^{-2} \in \mathscr{S}(\RR_{\ts});
\]
thus, the contribution to $\phi_0$ from the most singular term is $2c_0 t_*^{-2}$.
More generally,  we have
\[
    \calF_{t_*\to\sigma}\inv ( \sigma^{k}\log(\xs+i0)) = d_k t_*^{-k-1}, \quad k \in \N_0, \quad d_k = e^{-i\pi k/2}(-1)^k k!.
\]
We also note that the inverse Fourier transform of the smooth term (namely the  $\CcI((-1,1)_\sigma;\CI(X^\circ))$ part) lies in $\mathscr{S}(\R_{t_*})$. In view of \cite[Lemma 3.6]{Hin22}, either of the conormal terms, which belongs to $\A^{2-}([0,1)_\sigma;\CI(X^\circ))$ or $\A^{2-}((-1,0]_\sigma;\CI(X^\circ))$, has inverse Fourier transform that is bounded by $t_*^{-3+}$ together with all their $t_*\pa_{\ts}$ and $\pa_x$ derivatives. Hence, we conclude that for $x\in K\Subset X^\circ$,
\begin{equation}
\label{phi0 asymp}
    |\pa_{\ts}^j \pa_x^\beta (\phi_0 - 2c_0 t_*^{-2})|\leq C_{j\beta K} t_*^{-3-j+},\quad \forall\, j\in\mathbb{N},\ \beta\in\mathbb{N}^3.
\end{equation}

\subsubsection{The piece $\phi_1$}
We consider $\phi_1$ given in \eqref{eqn:phi012}. Using the Taylor expansion of $\wh{f}(\sigma)$ (in $\sigma$), we obtain
\begin{equation}
\label{phi1 Taylor expansion}
    \frac{\widehat{f}(\sigma)-\widehat{f}(0)}{\sigma} = i\hat{g} - \sigma\int_0^1 (1-s)\int_{\RR_{\ts}} e^{i s\sigma\ts} \ts^2 f(\ts,x)d\ts ds,
\end{equation}
where
\begin{equation}
\label{eqn:hat g}
    \hat{g}:= \int_{\RR_{\ts}} \ts f(\ts,x)d\ts.
\end{equation}
The decay estimate \eqref{tilde phi decay} implies that for each $\epsilon>0$,
\[
	\phi-\rho\phi_{\rm{rad},1} \in \rho^{2-\epsilon}\tau^\epsilon \A^0 ([0,1)_{\rho,\tau}^2 \times \mathbb{S}^2).
\]
We deduce the following from \eqref{phicube explicit} and \eqref{phicube expansion}:
\[
    \ts\phi^3 \in \rho^3 \ts \phi_{\rm{rad},1}^3(\ts,\omega) + \A^{1+\epsilon}([0,1)_\tau;\A^{4-\epsilon}(X)),\quad\tau=\ts^{-1}.
\]
Integrating this in $\ts$, and noting that $\A^{1+\epsilon}([0,1)_\tau)$ corresponds to $\mathcal{O}(\ts^{-1-\epsilon})$ which is in $L^1(\RR_{\ts})$, and that the other terms of $f(\ts,x)$ given in \eqref{the forcing f} are all supported in $0<\ts<1$, we conclude that
\begin{equation}
\label{hat g leading term}
    \hat{g}\in \rho^3 \widetilde{c}(\omega) + \A^{4-}(X),
\end{equation}
where $\widetilde{c}(\omega)$ can be computed explicitly, like \eqref{eqn:c(omega)}. 

We now estimate the remainder of the integral form in \eqref{phi1 Taylor expansion} in an iterated conormal space. Recalling \eqref{phi conormal est}, we obtain that for each $\epsilon>0$,
\[
    \ts^2 \phi^3 \in \A^{1+\epsilon}([0,1)_\tau;\A^{3-\epsilon}(X)),\quad\tau=\ts^{-1},
\]
since $\ts^2 {\la t_*+2r \ra}^{-3} \jts^{-3}$ is bounded by $\jts^{-1-\epsilon} {\la r \ra}^{-3+\epsilon}$. Therefore,
\begin{equation}
\label{integral remainder}
    \int_0^1 (1-s)\int_{\RR_{\ts}} e^{is\sigma\ts} \ts^2 f(\ts,x)d\ts ds \in \A^0(\pminus[0,1)_\sigma;\A^{3-}(X)).
\end{equation}
Repeating the steps to obtain \eqref{eq:phi0}, it follows from \eqref{hat g leading term} that
\begin{equation}
\label{Boxhat of hat g}
\begin{split}
     \psi(\sigma)\wh\Box(\sigma)^{-1} \hat{g} \in & -2i\ti{c}_0\psi(\sigma)\sigma\log(\sigma+i0) + \CcI((-1,1)_\sigma;\CI(X^\circ)) \\
     & + \sum_{\pminus} \A^{2-}(\pminus[0,1)_\sigma;\CI(X^\circ)),
\end{split}
\end{equation}
where $\ti{c}_0 := (4\pi)^{-1}\int_{\S^2} \widetilde{c}(\omega) d\omega.$ 
Recalling from Lemma \ref{lem:aux} that 
\[
\wh\Box(\sigma)^{-1} : \A^0([0,1)_\sigma;\A^{3-}(X)) \to \A^{1-,1-,0}(X_{\rm{res}}^+),
\]
we then conclude from \eqref{phi1 Taylor expansion} and \eqref{integral remainder} that
\begin{equation}
\label{phi1 hat split}
\begin{split}
    \widehat{\phi_1}(\sigma) &= \sigma\psi(\sigma)\wh\Box(\sigma)^{-1}(\sigma^{-1}(\widehat{f}(\sigma)-\widehat{f}(0))) \\
    &\in i\sigma\psi(\sigma)\wh\Box(\sigma)^{-1} \hat{g} + \sigma^2 \A^{1-,1-,0}(X_{\rm{res}}^\pminus). 
\end{split}
\end{equation}
Using \eqref{Boxhat of hat g} we proceed to obtain 
\begin{equation}
\label{phi1 hat compact}
    \widehat{\phi_1}(\sigma)\in \CcI((-1,1)_\sigma;\CI(X^\circ)) + \sum_{\pminus} \A^{2-}(\pminus[0,1)_\sigma;\CI(X^\circ)),
\end{equation}
for which we simply estimated
$	2\ti{c}_0 \psi(\sigma)\sigma^2 \log(\sigma+i0) \in \sum_{\pminus} \A^{2-}(\pminus[0,1)_\sigma;\CI(X^\circ)),$
with the help of the extra factor $\sigma$ comparing to the case of $\phi_0$; we also used the inclusion $\sigma^2 \A^{1-,1-,0}(X_{\rm{res}}^\pminus)\subset \A^{2}(\pminus[0,1)_\sigma;\A^{1-}(X))$.
Arguing as in the steps to obtain \eqref{phi0 asymp}, we can conclude from \eqref{phi1 hat compact} that for $\phi_1$ on spatially compact sets $K\Subset X^\circ$,
\begin{equation}
\label{phi1 asymp}
    |\pa_{\ts}^j \pa_x^\beta \phi_1 (\ts,x) |\leq C_{j\beta K} t_*^{-3-j+},\quad \forall\, j\in\mathbb{N},\ \beta\in\mathbb{N}^3.
\end{equation}

\subsubsection{Finalizing the proof of Theorem \ref{thm:compactmain}}
Combining \eqref{phi0 asymp}, \eqref{phi1 asymp}, and \eqref{phi2 rapid decay}, we obtain the following long time asymptotics for $\phi$ in spatially compact subsets of $K\Subset X^\circ \cong \R^3$,
\begin{equation*}
    |\pa_{\ts}^j \pa_x^\beta (\phi - 2c_0 t_*^{-2})|\leq C_{j\beta K} t_*^{-3-j+},\quad \forall\, j\in\mathbb{N},\ \beta\in\mathbb{N}^3,
\end{equation*}
which completes the proof of \eqref{phi asymp spatial compact}. 

It remains to show that the leading coefficient $c_0$ is independent of the choice of the cutoff function $\chi(\ts)$ in \eqref{forcing problem}. Specifically, we prove the formula \eqref{eq:leadcoefficient}. 
Indeed, we can justify this with a straightforward calculation.

Recall the expression from \eqref{eqn:c(omega)} and \eqref{eqn:coefficient c0}: 
\begin{equation}
c_0 = \frac{1}{4\pi} \int_{\S^2} \int_\RR (\chi(t_*) a_0(\ts,\omega) \phi_{\rm rad,1}^3 - 2\chi' \phi_{\rm{rad},2} - \chi' \ti{g} \pa_{\ts}\phi_{\rm{rad},1}) d\ts d\omega.
\label{eq:c0 again}
\end{equation}

Using the formula \eqref{eq:form of 2nd radi field} for $\phi_{\rm rad,2}$, we compute
\[
\begin{split}
    { }&\quad\chi a_0 \phi_{\rm rad,1}^3 - 2\chi' \phi_{\rm{rad},2} - \chi' \ti{g} \pa_{\ts}\phi_{\rm{rad},1} \\
    &= \chi a_0 \phi_{\rm{rad},1}^3 - (2c_2+\ti{g}d_1)\chi'(\ts)  + \chi'(\ts) \int_0^{\ts} (a_0(s,\omega)\phi_{\rm{rad},1}^3(s,\omega) -\Delta_\omega\phi_{\rm{rad},1}(s,\omega) )\,ds \\
    &= \pa_{\ts} \left( \chi(\ts) \int_0^{\ts} a_0(s,\omega) \phi_{\rm{rad},1}^3(s,\omega) ds \right)  -(2c_2+\ti{g}d_1)\chi'(\ts)  - \chi'(\ts)\int_0^{\ts} \Delta_\omega\phi_{\rm{rad},1}(s,\omega) ds.
\end{split}
\]
We compute directly to see that 
\[
	\int_{\S^2} \Delta_\omega h(\omega) d\omega=0,\quad\forall h\in\CI(\S^2),
\]
and that 
\[
	\left(\chi(\ts) \int_0^{\ts} a_0(s,\omega)\phi_{\rm{rad},1}^3(s,\omega) ds\right)\bigg\lvert_{\ts=-\infty}^{\ts=\infty} = \int_0^\infty a_0(\ts,\omega)\phi_{\rm{rad},1}^3(\ts,\omega) d\ts.
\]
Therefore, we conclude that with $\int_{\R} \chi'(\ts) d\ts = 1$,
\begin{equation}
	c_0 = \frac{1}{4\pi} \int_{\S^2}  \left( \int_0^\infty a_0(\ts,\omega)\phi_{\rm{rad},1}^3(\ts,\omega) d\ts -  (2c_2(\omega)+\ti{g}(\omega)d_1(\omega))\right) d\omega,
\label{eq:c0final}
\end{equation}
which thus completes the proof of Theorem \ref{thm:compactmain}. Thus $c_0$ is independent of the choice of the cutoff $\chi$. 

\section{Global asymptotics}
\label{sec:global asymp}
This section is devoted to the proof of Theorem~\ref{thm:global main_body text}, which is the fully detailed version of Theorem~\ref{thm:global main}. Let us first recall the formal definition of the resolved spacetime manifold $M_+$ from \cite[Definition 3.7]{Hin22}:
\begin{definition}
\label{def:M+}
Let $X=\overline{\RR^3}$ be the radial compactification of $\RR^3$ defined in Definition \ref{DefACompact}. Let $[0,\infty]_{\ts} := ([0,\infty)_{\ts}\sqcup [0,\infty)_\tau)/\sim$ where the equivalence relation $\sim$ identifies $\ts>0$ with $\tau=\ts^{-1}$. The resolved spacetime (for positive $\ts$) is the blow-up
\[
    M_+ := [[0,\infty]_{\ts}\times X ; \{\infty\}_{\ts}\times\partial X],
\]
where we blow up the corner $\{\infty\}_{\ts}\times\partial X = \{\rho=\tau=0\}$, as in \cite[Definition 3.7]{Hin22} or \cite[Definition 2.12]{Hin22} and introduce a new variable $v:=\rho/\tau$ and the front face $I^+ := [0,\infty]_v \times\partial X$, which is the Minkowski future timelike infinity. Moreover, we have null infinity $\mathscr{I}^+$ as the lift of $[0,\infty]_{\ts}\times\partial X$ and the spatially compact future infinity $\mathcal{K}^+$ is the lift of $\{\infty\}_{\ts}\times X$. Let $\rho_{\mathscr{I}^+},\,\rho_{I^+},\,\rho_{\mathcal{K}^+} \in \CI(M_+)$ be the defining functions of the boundary hypersurfaces $\mathscr{I}^+,\,I^+,\,\mathcal{K}^+$, respectively. Using the same notation as in \cite[Definition 3.7]{Hin22}, let us set, away from $\mathscr{I}^+$, $\rho_{\mathcal{K}^+} = \rho_{\mathcal{K}^+,0} := \tau/\rho$, $\rho_{I^+} = \rho_{I^+,0} := \rho$; away from $\mathcal{K}^+$, $\rho_{\mathscr{I}^+} = \rho_{\mathscr{I}^+,1} := \rho/\tau$, $\rho_{{I}^+} = \rho_{{I}^+,1} := \tau$.
\end{definition}

\noindent
We recall the definition of the conormal spaces on $M_+$ from \cite[Definition 3.8]{Hin22}:
\begin{definition}
\label{def:conormal on M+}
The conormal space $\A^{0,0,0}(M_+)$ consists of functions that are smooth at $\Sigma=\{\ts=0\}$ and remain bounded after application of any finite number of b-vector fields on $M_+$ (tangent to the boundary hypersurfaces $\mathscr{I}^+,\,I^+,\,\mathcal{K}^+$), for example, $V=\rho_{{I}^+}\partial_{\rho_{{I}^+}}$ near $I^+$. Let $\A^{\alpha, \beta, \gamma}(M_+) = \rho_{\mathscr{I}^+}^\alpha \rho_{I^+}^\beta \rho_{\mathcal{K}^+}^\gamma \A^{0,0,0}(M_+)$. The polyhomogeneous conormal space 
\[
\A^{(\cE_{\mathscr{I}^+},\alpha_{\mathscr{I}^+}),(\cE_{{I}^+},\alpha_{{I}^+}), (\cE_{\mathcal{K}^+},\alpha_{\mathcal{K}^+})} (M_+),
\]
where $\cE_{\mathscr{I}^+},\,\cE_{{I}^+}\, \cE_{\mathcal{K}^+}$ are index sets, consists of all $u$ which have partial expansions at boundary hypersurfaces ${\mathscr{I}^+},\,{{I}^+},\,{\mathcal{K}^+}$. See Definition \ref{defn:partial expansion} and \cite[Definition 2.13]{Hin22} for a more detailed introduction.
\end{definition}

\begin{theorem}
\label{thm:global main_body text}
Let $\phi$ be the unique forward solution to the initial value problem \eqref{I.V. problem}, with initial data satisfying \eqref{initial data u0, u1} and conormal bounds \eqref{phi conormal est}. Then in terms of index sets which characterize the leading term and the decay rate of the remainder term on each boundary hypersurface of the spacetime $M_+$ ((see Figure~\ref{fig:M+} for a picture of $M_+$)), we have
\[
    \phi \in \A^{((1,0),2),((2,0),3-),((2,0),3-)}(M_+).
\]
Here, the first index set describes the partial expansion (see Definition \ref{defn:partial expansion}) at the boundary hypersurface $\mathscr{I^+}$, while the second and the third index sets correspond to boundary hypersurfaces $I^+$ and $\mathcal{K}^+$, respectively. Moreover, the leading terms of $\phi$ on the different boundary hypersurfaces are:
\begin{enumerate}
    \item On $\calK^+$, $\phi \sim 2 c_0 t_*^{-2}$ with $c_0$ given in \eqref{eq:leadcoefficient}.
    
    \item On $I^+$, $\phi \sim 2c_0 \ts^{-2} \dfrac{\rho \ts}{\rho \ts + 2}$, with $v:=\rho\ts$ being a coordinate on $I^+$, as $I^+\cong (0,\infty)_v \times\partial X$.
    
    \item On $\mathscr{I}^+$, 
    $\phi \sim\rho \phi_{{\rm rad},1}$, where $\phi_{{\rm rad},1}$ is the first order radiation field given in Lemma \ref{lem:rad field small time} that satisfies \eqref{eq:form of first radi}, which shows that the leading term of $\phi_{{\rm rad},1}$ is  $c_0 t_*^{-1}$ and therefore matches up with the leading order term at $I^+$. 
\end{enumerate}
\end{theorem}

We shall build on the spatially compact results, i.e. the asymptotic near $\mathcal{K}^+$ to analyze the behavior of $\phi$ near the other boundary hypersurfaces at ``infinity" of $M_+$. The resolvent expansion \eqref{decompose u sing and reg}--\eqref{u regular} is used to get a precise asymptotic expansion near $I^+$. For the behavior near null infinity $\mathscr{I}^+$, we revisit the radiation field expansion and then combine it with the asymptotic we obtained near the front face $I^+$. The method gives a unified picture of asymptotics in all regions.

\bigskip
\noindent
{\textit{Proof of Theorem \ref{thm:global main_body text}}.} We employ the splitting in \eqref{eqn:phi012}:
\[
    \phi_\chi = \phi_0 + \phi_1 + \phi_2.
\]
It follows from \eqref{phi2 rapid decay} that
\begin{equation}
\label{phi2 in conormal M+}
    \phi_2 \in \A^{1-,\infty,\infty}(M_+).
\end{equation}
As $\widehat{\phi_1}(\sigma)$ has an extra factor of $\sigma$ than $\widehat{\phi_0}(\sigma)$, we first consider $\phi_0$. We recall \eqref{decompose u sing and reg} that for $\sigma>0$,
\[
    \wh{\phi_0}(\sigma) = \psi(\sigma)\wh\Box(\sigma)^{-1}\wh{f}_0 = \psi(\sigma)u_{\rm{sing}}(\sigma) + \psi(\sigma)u_{\rm{reg}}(\sigma).
\]
Since $\wh{f}_0=\wh{f}(0,x)$ is real-valued, we have for $\sigma<0$:
\begin{equation}
\label{resolvent negative sigma}
    \wh\Box(\sigma)^{-1}\wh{f}_0 = \overline{\wh\Box(-\sigma)^{-1}\wh{f}_0}.
\end{equation}
It follows that the smooth terms of $u_{\rm{reg}}$ in \eqref{u regular} for $\sigma \gtrless 0$ can be glued together to be smooth on $(-1,1)_\sigma$,
\[
    \wh\Box(0)^{-1} \wh{f}_0 + i\sigma \wh\Box(0)^{-1}\ti{f_1} \in\CI ((-1,1)_\sigma; \A^{1-}(X)),
\]
in view of the following identity,
\[
	\overline{\wh\Box(0)^{-1} \wh{f}_0 + i(-\sigma)\wh\Box(0)^{-1}\ti{f_1}} = \wh\Box(0)^{-1} \wh{f}_0 + i\sigma \wh\Box(0)^{-1}\ti{f_1}.
\]
Therefore, the corresponding contribution to $\phi_0$ satisfies
\begin{equation}
\label{phi0 reg 1}
    \phi_{0,\rm{reg},1} := \mathcal{F}^{-1} (\psi(\sigma)(\wh\Box(0)^{-1} \wh{f}_0 + i\sigma \wh\Box(0)^{-1}\ti{f_1})) \in \A^{1-,\infty,\infty}(M_+).
\end{equation}
The remaining terms of $u_{\rm{reg}}$ in \eqref{u regular} for $\sigma\gtrless 0$ can be written as
\[
    u_{\rm{reg},2} = \sum_{\pminus} u_{\rm{reg},2,\pminus}, \quad u_{\rm{reg},2,\pminus} \in \A^{2-} (\pminus[0,1)_\sigma ; \A^0(X)).
\]
We recall \cite[Lemma 3.6]{Hin22} to obtain 
\begin{equation}
\label{IFT conormal}
    \mathcal{F}^{-1} : \A^{2-} (\pminus[0,1)_\sigma ; \A^0(X)) \to \A^{3-}([0,1)_\tau;\A^0(X)), \ \ \tau = t_*^{-1}.
\end{equation}
That lemma computes the inverse Fourier transform of a function $\hat g$ with $\supp \hat g \subset [0,1)$ enjoying conormal properties in $\sigma$, and shows that $\mathcal{F}^{-1} \hat g$ has decay in time. It follows that
\begin{equation}
\label{phi0 reg 2}
    \phi_{0,\rm{reg},2} := \mathcal{F}^{-1}(\psi u_{\rm{reg},2}) \in \A^{3-}([0,1)_\tau;\A^0(X)) \subset \A^{0,3-,3-}(M_+).
\end{equation}

We now consider the main contribution to $\phi_0$, which comes from the singular part $u_{\rm{sing}}(\sigma)$. We first recall from Proposition \ref{Hintz Lemma 2.24} the computation of $\hat\Box(\sigma)^{-1}\rho^2$. It follows from \eqref{eq:better} that for $\sigma>0$,
\[
    u_{\rm{sing}}(\sigma) = 2ic_0 \sigma \wh\Box(\sigma)^{-1}\rho^2 \in 2ic_0\sigma \ti u_0(x) + 2ic_0\sigma \ti u_{\rm mod}(\xs/\rho) + \A^{1-,2-,2-}(X_{\rm{res}}^+).
\]
The equation \eqref{resolvent negative sigma} then implies that for $\sigma<0$,
\[
    u_{\rm{sing}}(\sigma) \in 2ic_0\sigma \ti u_0 + 2ic_0\sigma \overline{ \ti u_{\rm mod}(-\sigma/\rho)} + \A^{1-,2-,2-}(X_{\rm{res}}^-).
\]
With the help of the Heaviside function $H(x)$, we combine the pieces for $\sigma\gtrless 0$ as follows:
\[
    u_{\rm{sing}} = u_{\rm{sing},0} + u_{\rm{sing},1} + u_{\rm{sing},2},
\]
\begin{equation}
\label{u sing 0}
	u_{\rm{sing},0} = 2ic_0 \sigma \ti u_0 (x) \in \CI(\RR_\sigma;\A^{1-}(X;\RR)),
\end{equation}
\begin{equation}
\label{u sing 1}
    u_{\rm{sing},1} = 2ic_0 \sigma ( \ti u_{\rm mod}(\sigma/\rho) H(\sigma/\rho) + \overline{ \ti u_{\rm mod}(-\sigma/\rho)} H(-\sigma/\rho) ),
\end{equation}
\begin{equation}
\label{u sing 2}
    u_{\rm{sing},2} = \sum_{\pminus} u_{\rm{sing},2,\pminus}, \quad u_{\rm{sing},2,\pminus} \in \A^{2-} (\pminus[0,1)_\sigma ; \A^0(X)).
\end{equation}
In analogy with \eqref{phi0 reg 1}, by \eqref{u sing 0} we have
\begin{equation}
\label{phi0 sing 0}
	\phi_{0,\rm{sing},0} := \mathcal{F}^{-1}(\psi u_{\rm{sing},0}) \in \A^{1-,\infty,\infty}(M_+).
\end{equation}
Arguing as in \eqref{IFT conormal} and \eqref{phi0 reg 2}, we obtain from \eqref{u sing 2} that
\begin{equation}
\label{phi0 sing 2}
    \phi_{0,\rm{sing},2} := \mathcal{F}^{-1}(\psi u_{\rm{sing},2}) \in \A^{0,3-,3-}(M_+).
\end{equation}
We now compute the leading term of $\phi$ at $I^+$:
\begin{equation}
\label{phi0 sing 1}
        \phi_{0,\rm{sing},1} := \frac{1}{2\pi}\int_{\RR}e^{-i\sigma t_*}\psi(\sigma)u_{\rm{sing},1}(\sigma)\,d\sigma.
\end{equation}
Writing $\sigma = \rho\Hat{r}$, using \eqref{u sing 1} we have
\[
    \phi_{0,\rm{sing},1} = \frac{ic_0\rho^2}{\pi}\int_{\RR} e^{-i\Hat{r}\rho t_*}\Hat{r}\psi(\rho\hat{r})(\ti u_{\rm mod}(\hat{r}) H(\hat{r}) + \overline{ \ti u_{\rm mod}(-\hat{r})} H(-\hat{r})) d\hat{r}.
\]
We recall that $\psi\in \CcI((-1,1)_\sigma)$ is an even function. Then,
\[
\begin{split}
    \phi_{0,\rm{sing},1} &= \frac{ic_0\rho^2}{\pi}\int_0^\infty \Hat{r}\psi(\rho\hat{r})(e^{-i\Hat{r}\rho t_*}\ti u_{\rm mod}(\hat{r}) - e^{i\Hat{r}\rho t_*}\overline{ \ti u_{\rm mod}(-\hat{r})}) d\hat{r} \\
    &= -\frac{2c_0\rho^2}{\pi}\Im \int_0^\infty \Hat{r}\psi(\rho\hat{r})e^{-i\Hat{r}\rho t_*}\ti u_{\rm mod}(\hat{r}) d\hat{r}.
\end{split}
\]
Integrating by parts using $e^{-i\Hat{r}\rho t_*}=i(\rho t_*)^{-1}\pa_{\hat{r}}e^{-i\Hat{r}\rho t_*}$, we obtain
\[
\begin{split}
    \phi_{0,\rm{sing},1} &= -\frac{2c_0\rho^2}{\pi}\Im\left( \frac{-i}{\rho t_*} \int_0^\infty e^{-i\Hat{r}\rho t_*} \pa_{\hat{r}} (\psi(\rho\hat{r})\hat{r}\ti u_{\rm mod}(\hat{r})) d\hat{r} \right) \\
    &= \frac{2c_0\rho t_*^{-1}}{\pi} \Re \int_0^\infty e^{-i\Hat{r}\rho t_*} \pa_{\hat{r}} (\psi(\rho\hat{r})\hat{r}\ti u_{\rm mod}(\hat{r})) d\hat{r} \\
    &= \phi_{0,\rm{sing},1}^\text{I} + \phi_{0,\rm{sing},1}^\text{II}
\end{split}
\]
where
\[
    \phi_{0,\rm{sing},1}^\text{I} := \frac{2c_0\rho t_*^{-1}}{\pi} \Re \int_0^\infty e^{-i\Hat{r}\rho t_*} \psi(\rho\hat{r})\pa_{\hat{r}} (\hat{r}\ti u_{\rm mod}(\hat{r})) d\hat{r},
\]
\[
    \phi_{0,\rm{sing},1}^\text{II} := \frac{2c_0\rho t_*^{-1}}{\pi} \Re\int_0^\infty e^{-i\Hat{r}\rho t_*}\psi'(\rho\hat{r})\rho\hat{r}\ti u_{\rm mod}(\hat{r}) d\hat{r}.
\]
We may assume that $\supp\psi' \cap (0,\infty) \subset (\delta,1)$ for some $\delta>0$. Letting $\rho\hat{r} = \rho$, assuming now without loss of generality that $\psi' = \mathbf{1}_{(0,\infty)} \psi'$, we derive that
\[
    \phi_{0,\rm{sing},1}^\text{II} =  \Re\frac{2c_0 t_*^{-1}}{\pi}\int_{\mathbb{R}} e^{-i\sigma t_*}\psi'(\sigma)\sigma\ti u_{\rm mod}(\sigma/\rho) d\sigma.
\]
It follows from \cite[Lemma 2.23]{Hin22} that $\psi'(\sigma)\sigma\ti u_{\rm mod} \in \CcI(\RR_{\sigma}; \A^{1-}(X))$. Therefore,
\begin{equation}
\label{phi0sing1 II}
    \phi_{0,\rm{sing},1}^\text{II} =  \Re\frac{2c_0 t_*^{-1}}{\pi} \mathcal{F}^{-1} (\psi'(\sigma)\sigma\ti u_{\rm mod}) \in \A^{1-,\infty,\infty} (M_+).
\end{equation}
We recall from \cite[Remark 2.25]{Hin22} that
\[
    \pa_{\hat{r}} (\hat{r}\ti u_{\rm mod}(\hat{r})) = \int_0^\infty e^{-2t\hat{r}}(t-i)^{-1}dt.
\]
It follows that
\[
    \phi_{0,\rm{sing},1}^\text{I} = \frac{2c_0\rho t_*^{-1}}{\pi} \Re \int_0^\infty \left(\int_0^\infty e^{-\hat{r}(2t+i\rho t_*)}\psi(\rho\hat{r})d\hat{r}\right) (t-i)^{-1}dt.
\]
Applying integration by parts to the inner integral, we have
\begin{equation}
\label{phi0sing1 main}
    \phi_{0,\rm{sing},1}^\text{I} = \frac{2c_0\rho t_*^{-1}}{\pi} \Re \int_0^\infty \frac{dt}{(2t+i\rho t_*)(t-i)} + \ti{\phi}_{0,\rm{sing},1}^\text{I},
\end{equation}
where
\[
    \ti{\phi}_{0,\rm{sing},1}^\text{I} := \frac{2c_0\rho^2 t_*^{-1}}{\pi} \Re \int_0^\infty e^{-i\hat{r}\rho t_*}\psi'(\rho\hat{r}) \int_0^\infty \frac{e^{-2t\hat{r}} dt}{(2t+i\rho t_*)(t-i)} d\hat{r}.
\]
We set
\[
    g(\sigma,\rho;t_*) := \int_0^\infty \frac{e^{-2t\hat{r}} dt}{(2t+i\rho t_*)(t-i)} = \int_0^\infty \frac{e^{-2t\sigma}dt}{(2t+it_*)(\rho t - i)}.
\]
Then, we have, with $\supp\psi'\subset (\delta,1)_\sigma$,
\[
    \ti{\phi}_{0,\rm{sing},1}^\text{I} = \frac{2c_0 \rho t_*^{-1}}{\pi} \Re \int_{\R} e^{-i\sigma t_*}\psi'(\sigma) g d\sigma = \frac{2c_0\rho t_*^{-1}}{\pi} \Re \mathcal{F}^{-1}(\psi'(\sigma)g).
\]
We note that for $\delta<\sigma<1$, $g=g(\sigma,\rho;t_*)$ extends smoothly to $\{\rho=0\}$, thus $\psi'(\sigma) g \in \CcI(\RR_{\sigma}; \CI(X))$. It follows that
\begin{equation}
\label{ti phi0sing1}
    \ti{\phi}_{0,\rm{sing},1}^\text{I} \in \rho \A^{0,\infty,\infty}(M_+)\subset \A^{1,\infty,\infty}(M_+).
\end{equation}
It remains to compute the integral in \eqref{phi0sing1 main}:
\[
    \int_0^\infty \frac{dt}{(2t+i\rho t_*)(t-i)} = \frac{i}{\rho t_* + 2}(\log 2 - \log(\rho t_*) - i\pi),
\]
thus we conclude from \eqref{phi0sing1 main} and \eqref{ti phi0sing1} that
\begin{equation}
\label{phi0sing1 I}
    \phi_{0,\rm{sing},1}^\text{I} \in 2c_0 t_*^{-2} \frac{\rho t_*}{\rho t_* + 2} + \A^{1,\infty,\infty}(M_+).
\end{equation}
Combining \eqref{phi0sing1 II} and \eqref{phi0sing1 I} we obtain 
\begin{equation}
\label{phi0sing1 final}
    \phi_{0,\rm{sing},1} \in 2c_0 t_*^{-2} \frac{\rho t_*}{\rho t_* + 2} + \A^{1-,\infty,\infty}(M_+).
\end{equation}
In view of this together with \eqref{phi0 reg 1}, \eqref{phi0 reg 2}, \eqref{phi0 sing 0} and \eqref{phi0 sing 2}, we conclude
\begin{equation}
\label{phi0globalasymp}
    \phi_0 \in \A^{0,((2,0),3-),((2,0),3-)}(M_+),
\end{equation}
where the leading order term at $\mathcal{K}^+$ is $2c_0 t_*^{-2}$ by \eqref{phi0 asymp}, the leading order term at $I^+$ is $2c_0 t_*^{-2} \dfrac{\rho t_*}{\rho t_* + 2}.$

It remains to consider $\phi_1$. We recall, from \eqref{phi1 hat split}, that
\[
    \widehat{\phi_1}(\sigma) \in i\sigma\psi(\sigma)\wh\Box(\sigma)^{-1} \hat{g} + \sum_\pminus \A^{2}(\pminus[0,1)_\sigma;\A^{1-}(X))
\]
Using \cite[Lemma 3.6]{Hin22} and following the argument to obtain \eqref{phi0 reg 2}, we see that
\begin{equation}
\label{phi1 asymp prep}
    \phi_1 \in \mathcal{F}^{-1}(i\sigma\psi(\sigma)\wh\Box(\sigma)^{-1} \hat{g}) + \A^{1-,3,3}(M_+).
\end{equation}
In view of \eqref{hat g leading term} and \eqref{eqn:fhat}, we can repeat the analysis of $\phi_0 = \mathcal{F}^{-1}(\psi(\sigma)\wh\Box(\sigma)^{-1} \hat{f})$ to obtain, similarly as in \eqref{phi0globalasymp},
\begin{equation}
\label{IFT psiBoxinvhatg}
    \mathcal{F}^{-1}(\psi(\sigma)\wh\Box(\sigma)^{-1} \hat{g}) \in \A^{0,((2,0),3-),((2,0),3-)}(M_+).
\end{equation}
We compute
\[
    \mathcal{F}^{-1}(i\sigma\psi(\sigma)\wh\Box(\sigma)^{-1} \hat{g}) = -t_*^{-1} (\ts \pa_{\ts})\mathcal{F}^{-1}(\psi(\sigma)\wh\Box(\sigma)^{-1} \hat{g}).
\]
Note that $\ts \pa_{\ts} = \rho_{\mathscr{I}^+,1}\pa_{\rho_{\mathscr{I}^+,1}} - \rho_{I^+,1}\pa_{\rho_{I^+,1}} = - \rho_{\mathcal{K}^+,0}\pa_{\rho_{\mathcal{K}^+,0}}$ is a b-vector field near boundaries $\mathscr{I}^+$, $I^+$, $\mathcal{K}^+$. Also note that $t_*^{-1} = \rho_{I^+,1} = \rho_{I^+,0}\rho_{\mathcal{K}^+,0}$. Thus, we obtain from \eqref{IFT psiBoxinvhatg}
\[
    \mathcal{F}^{-1}(i\sigma\psi(\sigma)\wh\Box(\sigma)^{-1} \hat{g}) \in \A^{0,3,3}(M_+).
\]
This, with \eqref{phi1 asymp prep}, shows that
\begin{equation}
\label{phi1globalasymp}
    \phi_1 \in \A^{0,3,3}(M_+).
\end{equation}
Combining \eqref{phi0globalasymp}, \eqref{phi1globalasymp} and \eqref{phi2 in conormal M+}, we conclude that
\begin{equation}
\label{globalasymp0}
    \phi \in \A^{0,((2,0),3-),((2,0),3-)}(M_+),
\end{equation}
where the leading order term at $\mathcal{K}^+$ is $2c_0 t_*^{-2}$ in view of \eqref{phi0 asymp}, and the leading order term at $I^+$ is $2c_0 t_*^{-2} \dfrac{\rho t_*}{\rho t_* + 2}.$

We revisit the proof of Lemma \ref{lem:rad field large time} to derive a more precise description of the radiation field $\phi_{\rm{rad},1}$ than given in \eqref{radiation field 1}. Using the same notations for $R$, $M_0'$, etc, as in Lemma \ref{lem:rad field large time}, we recall that
\[
	v = \rho_{\mathscr{I}^+,1} = \rho\ts,\quad \tau = \rho_{I^+,1} = \ts^{-1}.
\] 
\eqref{globalasymp0} implies that $\phi\in \A^{0,((2,0),3-)}(M_0')$. Since, by \eqref{eqn:phi with v leading}, $\phi = vg_1(\tau,\xo) + \A^{2,2}(T)$, where $T:= M_0'$ is a triangle, and since $\phi \in \A^{0,((2,0),3-)}(T)$, we must have 
\begin{equation}
\label{final phi in M0'}
	\phi \in \A^{((1,0),2),((2,0),3-)}(T).
\end{equation}
This improves \eqref{eqn:phi with v leading} in the sense that
\[
	\phi - v g_1(\tau,\omega)\in \A^{2,((2,0),3-)}(M_0'),\quad g_1\in\CI(\S^2;\A^{((2,0),3-)}([0,1)_\tau)).
\] This statement is immediately implied by \eqref{final phi in M0'}.
Recalling \eqref{1=t*g1}, we then obtain
\[
	\phi_{\rm{rad},1} \in \CI(\S^2;\A^{((1,0),2-)}([0,1)_\tau)).
\]
In other words, the radiation field $\phi_{\rm{rad},1}$ expands with a leading term near $\tau=0$,
\[
	\phi_{\rm{rad},1} - c_{\rm{rad},1}(\omega) \tau \in \A^{2-}([0,1)_\tau),\quad c_{\rm{rad},1}\in\CI(\S^2).
\]
Furthermore, we note that the leading order terms at $I^+ \cap \mathscr{I}^+= \{v=\tau=0\}$ should match, which are $2c_0 \rho\tau (v + 2)^{-1}$ from $I^+$ and $\rho\tau c_{\rm{rad},1}(\omega)$ from $\mathscr{I}^+$.
It follows that $c_{\rm{rad},1}(\omega)$ is a constant function and equals $c_0$. We conclude that 
\begin{equation}\label{eq:form of first radi}
    \phi_{\rm{rad},1} - c_0\tau \in \A^{2-}([0,1)_\tau),\quad\textrm{with $c_0$ given by \eqref{eq:leadcoefficient}}.
\end{equation}
This completes the proof of Theorem \ref{thm:global main_body text}. 
\qed

\begin{corollary}
The expansion \eqref{eq:form of first radi} for the first radiation field $\phi_{\rm{rad},1}$ implies the identity:\[
    \begin{split}
    \forall\,\omega\in\S^2&,\quad  \lim_{\ts\to\infty} \ts\phi_{\rm{rad},1}(\ts,\omega) = \frac{1}{4\pi} \int_{\S^2} \left(\int_0^\infty a_0(\ts,\tilde{\omega})\phi_{\rm{rad},1}^3(\ts,\tilde{\omega}) d\ts -  2c_2(\tilde{\omega}) -\ti{g}(\tilde{\omega})d_1(\tilde{\omega})\right) \,d\ti\omega,
\end{split}
\]
where both sides of the identity are equal to the leading coefficient $c_0$. 
\end{corollary}

\section{Price's Law for semi-linear wave equations}

In this section we consider semilinear wave equations with nonlinearities of the form $b(t_*,x)\phi^p$, $p\geq 4$, of which power-type nonlinearities $\pminus \phi^p$ are a special case, and prove Theorem \ref{thm:highorder}. 
We adapt the methods from the cubic case to higher-power nonlinearities and use a refined resolvent expansion that allows us to handle the modified asymptotic behavior. In particular, we carefully analyze how the structure of the nonlinearity affects the leading-order terms. There is a transition at $p=4$ between ``low power" and ``high power" behavior. The proof shows how the nonlinearity modifies the classical Price's law and also provides a systematic method for computing the explicit constant in the leading-order term, a level of precision not attainable by all approaches.

Let $\phi$ be the solution to the following initial value problem:
\begin{equation}
\label{initial value problem Sec 6}
\Box_g \phi = b(\ts,x)\phi^p,\quad \phi(0,x) = u_0(x),\quad \partial_{t_*}\phi(0,x) = u_1(x), 
\end{equation}
with smooth and compactly supported initial data:
\begin{equation}
\label{initial data finer assumption}
\begin{gathered}
    u_0, \,u_1 \in \CcI(\RR^3),\quad\text{thus }u_0,\,u_1\in\A^\alpha(X),\ \forall\alpha>0.
\end{gathered}
\end{equation}
In \eqref{initial value problem Sec 6} we assume that $b\in\CI(\RR_{\ts};\CI(X))$ and that the estimate \eqref{a(t*,x) estimate} holds with $b$ replacing $a$ there, i.e.
\begin{equation}
\label{b(t*,x) estimate}
    |\pa_{\ts}^k \mathcal{V}_X^I b(\ts,x)|\leq C_{k, I} \la \ts \ra^{-k},\quad\forall k,I,
\end{equation}
where $\mathcal{V}_X\in \mathrm{span}\{\pa_{r}, \pa_\rho, \Omega_1,\Omega_2,\Omega_3\}$, $k\in\NN$, and $I$ is a multi-index of any order.

\medskip
\noindent
It follows from Theorem \ref{thm:power_nonlinearity} that the solution to \eqref{initial value problem Sec 6} with $p\geq 4$ satisfies the estimate
\begin{equation}
\label{phi conormal est high order}
	|\pa_{\ts}^k \Gamma_b^I \phi(\ts,x)| \leq C_{k I} \la\ts + 2 |x|\ra^{-1} \la \ts \ra^{-2-k}.
\end{equation}
We first extend the radiation field expansion in Lemma \ref{lem:rad field small time} to the third order. We begin with the case $p=4$. We use Taylor expansion to write
\begin{equation}
\label{eqn:g00 Sec 6}
    g^{00} - \rho^2 \ti g(\omega) - \rho^3 g_3(\omega) \in \A^4(X),\quad x=\rho^{-1}\omega,\ \;\omega\in\mathbb{S}^2.
\end{equation}
For the coefficient $b(\ts,x)$ of the nonlinearity, we write as in \eqref{eqn:a a0 atilde},
\begin{equation}
\label{eqn:b b0 btilde}
    b(\ts,x) = b_0(\ts,\omega) + \rho\tilde{b}(\ts,x),\quad \tilde{b}(\ts,x) = \int_0^1 (\pa_\rho b) (\ts,(s\rho)^{-1}\omega)\,ds.
\end{equation}
It follows that $\tilde{b}$ satisfies the same estimate \eqref{b(t*,x) estimate} as $b$ does. In particular, we have
\begin{equation}
\label{eqn:b b0}
    b(\ts,x) - b_0(\ts,\omega) \in \CI(\RR_{\ts};\A^1(X)),\quad x=\rho^{-1}\omega,\ \;\omega\in\mathbb{S}^2 .
\end{equation}
For the initial value problem \eqref{initial value problem Sec 6}, in view of \eqref{phi conormal est high order}, and following the proof of Lemma \ref{lem:rad field small time}, the second order radiation field expansion \eqref{radiation field 2nd order} still holds with $\phi_{\rm{rad},2}$ given by \eqref{eq:form of 2nd radi field}:
\begin{equation}
\label{eq:R2 Sec 6}
    R_2(\ts,\omega) = -\frac{1}{2}\tilde{g}(\omega)\partial_{\ts} R_1(\ts,\omega) + \frac{1}{2}\int_0^{\ts} \Delta_\omega R_1(s,\omega)\,ds.
\end{equation}
This follows since $a_0(\ts,\omega), \,c_2(\omega),\, d_1(\omega)\equiv 0$ per \eqref{initial value problem Sec 6}, \eqref{initial data finer assumption}.
Here we write $R_1 = \phi_{\rm{rad},1}$, $R_2 = \phi_{\rm{rad},2}$ for simplicity. Using \eqref{radiation field 2nd order}, together with \eqref{initial data finer assumption}, \eqref{eqn:b b0} and Lemma \ref{lem:rewritingBoxg}, we obtain a finer version of \eqref{eqn:Q0 phi 2nd}, 
\begin{equation}
\label{eqn:Q0 phi 3rd order}
\begin{split}
    Q_0(\phi - \rho R_1 - \rho^2 R_2) &\in
    \frac{\rho^3}{2}\int_0^{\ts} \bigr((\Delta_\omega-2)R_2(s,\omega) + 2m R_1(s,\omega) - b_0 R_1(s,\omega)^4 \bigr)\,ds \\
    & -\bigl(\frac{\tilde{g}}{2}\pa_{\ts} R_2 + \frac{g_3}{2} \pa_{\ts}R_1 + \widetilde{Q}_1 R_1 \bigr)\rho^3 + \CI([0,\infty)_{\ts};\A^{4-}(X)),
\end{split}
\end{equation}
where we denoted $\rho^{-1}\circ\widetilde{Q}\circ \rho = \widetilde{Q}_1 \in\Diffb^1(X)$. Inverting \eqref{eqn:Q0 phi 3rd order} using $Q_0\rho^3 = 2\rho^3$ and Proposition \ref{prop:bODE0}, we conclude
\begin{equation}
\label{radiation field 3rd order}
    \phi \in \rho\phi_{\rm{rad},1}(\ts,\omega) + \rho^2\phi_{\rm{rad},2}(\ts,\omega) + \rho^3\phi_{\rm{rad},3}(\ts,\omega) + \CI([0,\infty)_{t_*}; \A^{4-}(X)),
\end{equation}
where the third radiation field $\phi_{\rm{rad},3}\in\CI([0,\infty)_{t_*}\times\S_{\omega}^2)$ is given by
\begin{equation}
\label{eq:form of 3rd radi field}
    \begin{split}
         \phi_{\rm{rad},3}&(\ts,\omega) =  -\frac{1}{2}\widetilde{Q}_1 \phi_{\rm{rad},1}(\ts,\omega) - \frac{1}{4}g_3(\omega)\pa_{\ts}\phi_{\rm{rad},1}(\ts,\omega) - \frac{1}{4} \tilde{g}(\omega) \pa_{\ts}\phi_{\rm{rad},2}(\ts,\omega)\\
        &\quad + \frac{1}{4}\int_0^{\ts} \bigr((\Delta_\omega-2)\phi_{\rm{rad},2}(s,\omega) + 2m \phi_{\rm{rad},1}(s,\omega) - b_0(s,\omega) \phi_{\rm{rad},1}^4(s,\omega)\bigr)\,ds.
    \end{split}
\end{equation}
We remark that one can establish the third order radiation field expansion \eqref{radiation field 3rd order}, \eqref{eq:form of 3rd radi field} in the cases $p\geq 5$ without additional work, by noting that it suffices to replace $b_0$ with $0$ in the formula \eqref{eq:form of 3rd radi field} when $p\geq 5$. Moreover, by arguing as in the proof of Lemma \ref{lem:rad field large time} and using the improved \textit{a priori} estimate \eqref{phi conormal est high order} over \eqref{phi conormal est}, we can derive that
\begin{equation}
\label{rad field 1 decay sec 6}
    \begin{gathered}
			\phi_{\rm{rad},1}(\ts,\omega)\in \CI(\S^2;\A^2([0,1)_\tau),\quad \tau := \ts^{-2}, \\
			\phi - \rho \phi_{\rm{rad},1}(\ts,\omega) \in \rho^2 \tau \A^0 ([0,1)_\rho \times [0,1)_\tau \times \mathbb{S}^2)\cap \rho \tau^2 \A^0 ([0,1)_\rho \times [0,1)_\tau \times \mathbb{S}^2).
		\end{gathered}
\end{equation}

We now follow the approach in Section \ref{subsection:Preliminaries for spatially compact asymp} to convert the initial value problem \eqref{initial value problem Sec 6} into a forcing problem. We fix a cutoff $\chi(\ts)\in \CI(\RR_{\ts};\RR)$ which is identically equal to $0$ when $\ts\leq 1/2$ and identically equal to
$1$ when $\ts\geq 1$. Then $\phi_\chi := \chi\phi$, with $\phi$ solving \eqref{initial value problem Sec 6}, vanishes for $\ts\leq 1/2$ and solves the forcing problem:
\begin{equation}
\label{forcing problem Sec 6}
    \Box_g \phi_\chi = [\Box_g , \chi(\ts)]\phi + \chi(\ts) b(\ts,x)\phi^p =: f,\quad \textrm{on }\RR_{\ts}\times X^\circ.
\end{equation}
Here $f$ can be computed, in view of Lemma \ref{lem:rewritingBoxg}:
\begin{equation}
\label{the forcing f Sec 6}
    f = \chi b\phi^p - 2\chi'(\rho Q + g^{00}\pats)\phi - \chi'' g^{00} \phi. 
\end{equation}
Following the argument in Section \ref{subsection:Preliminaries for spatially compact asymp}, we can show that \eqref{phichi by IFT} still holds, namely
\[
   \phi_\chi (t_*,x) = \frac{1}{2\pi}\int_{\R} e^{-i\sigma t_*} \wh\Box(\sigma)^{-1} \widehat{f}(\sigma,x) \,d\sigma.
\]
As in \eqref{eqn:phi012}, we decompose $\phi_\chi = \phi_0 + \phi_1 + \phi_2$ in frequency by defining
\[
    \begin{split}
		\widehat{\phi_0}(\sigma) &:= \psi(\sigma)\wh\Box(\sigma)^{-1} \wh{f}(0), \\
		\widehat{\phi_1}(\sigma) &:= \sigma\psi(\sigma)\wh\Box(\sigma)^{-1}(\sigma^{-1}(\wh{f}(\sigma)-\wh{f}(0))), \quad 
		\widehat{\phi_2}(\sigma) := (1-\psi(\sigma))\wh\Box(\sigma)^{-1}\wh{f}(\sigma),
	\end{split}
\]
with $\psi\in \CcI((-1,1)_\sigma)$ being an even function that is identically $1$ for $\sigma$ near $0$.

\subsection{The low frequency part $\phi_0$}
\label{subsection: low frequency Sec 6}
We derive an asymptotic expansion of
\begin{equation}
\label{eqn:fhat(0) defn Sec 6}
    \hat{f}_0(x) := \hat{f}(0,x) = \int_{\mathbb R} f(t_*,x)\,dt_*,
\end{equation}
using a similar argument as in Section \ref{subsubsection:low frequency part}. To this end, we first write
\begin{equation}
\label{phi p power expansion}
	\phi^p = \rho^p \phi_{{\rm rad},1}^p + \sum_{j=1}^{p}\binom{p}{j} \rho^{p-j}\phi_{{\rm rad},1}^{p-j} (\phi-\rho\phi_{{\rm rad},1})^j.
\end{equation}
It follows from \eqref{rad field 1 decay sec 6} and \eqref{eqn:b b0 btilde}, \eqref{eqn:b b0} that
\begin{equation}
\label{p nonlinearity expansion}
	b(\ts,x)\phi^p - b_0(\ts,\omega)\rho^p \phi_{{\rm rad},1}^p \in \A^{2p-1}([0,1)_\tau ; \A^{p+1}(X)),\quad \tau=\ts^{-1} .
\end{equation}
Then we have
\begin{equation}
\label{chi b phi^p}
    \int_{\RR} \chi(\ts)b(\ts,x)\phi^p d\ts \in \left( 
    \int_{\RR} \chi b_0\phi_{{\rm rad},1}^p\,d\ts\right)\rho^p + \A^{p+1}(X) .
\end{equation}
For the remaining terms in \eqref{the forcing f Sec 6}, we combine \eqref{radiation field 3rd order} with \eqref{eqn:g00 Sec 6}, \eqref{eqn:b b0}, recalling Lemma \ref{lem:rewritingBoxg}, to write:
\[
\begin{gathered}
    - 2\chi'(\rho Q + g^{00}\pats)\phi - \chi'' g^{00} \phi - \bigl(-2\chi'(R_2+\tilde{g}\pa_{\ts}R_1) - \chi''\tilde{g} R_1\bigr)\rho^3 \\
    - \bigl(-2\chi'(2R_3+\widetilde{Q}_1 R_1 + \tilde{g}\pa_{\ts}R_2 + g_3 \pa_{\ts}R_1) - \chi''(\tilde{g}R_2+ g_3 R_1)\bigr)\rho^4 \in \A^5(X),
\end{gathered}
\]
where $\rho^{-1}\circ\widetilde{Q}\circ\rho = \widetilde{Q}_1$. For simplicity, we write $R_j = \phi_{{\rm rad},j}$, $j=1,2,3$. We conclude from the above computations that $\wh{f}_0$ admits the following asymptotic expansion:
\begin{equation}
\label{fhat(0) asymp Sec 6}
    \wh{f}_0(x) \in c(\omega)\rho^3 + d(\omega)\rho^4 + \A^5(X)
\end{equation}
where $c(\omega),\,d(\omega)\in\CI(\S^2)$ are given by
\begin{equation}
\label{eq:c(omega) defn Sec 6}
    c(\omega) = \int_{\RR} (-2\chi'(R_2+\tilde{g}\pa_{\ts}R_1) - \chi''\tilde{g} R_1)\,d\ts = \int_{\RR} (-2\chi'R_2 -\chi'\tilde{g}\pa_{\ts}R_1)\,d\ts;
\end{equation}
\begin{equation}
\label{eq:d(omega) defn Sec 6}
\begin{split}
    d(\omega) &= \int_{\RR} ( \chi F_p -2\chi'(2R_3+\widetilde{Q}_1 R_1 + \tilde{g}\pa_{\ts}R_2 + g_3 \pa_{\ts}R_1) - \chi''(\tilde{g}R_2+ g_3 R_1))\,d\ts \\
    &= \int_{\RR} (\chi F_p -4\chi'R_3 - 2\chi'\widetilde{Q}_1 R_1 - \chi' \tilde{g}\pa_{\ts}R_2 - \chi' g_3 \pa_{\ts}R_1)\,d\ts.
\end{split}
\end{equation}
Here we used the fundamental theorem of calculus as in \eqref{eqn:c(omega)}. $F_p$ is a convenient notation to summarize all cases $p\geq 4$. In fact, in view of \eqref{chi b phi^p}, we have $F_p = b_0 R_1^4$ if $p=4$, and $F_p=0$ for all $p\geq 5$. 

Let us make an important observation about the angular coefficient $c(\omega)$ in \eqref{fhat(0) asymp Sec 6}. Inspired by the analysis in Section \ref{subsubsection:low frequency part}, the significance of $c(\omega)$ is its average, namely $c_0 = \frac{1}{4\pi}\int_{\S^2} c(\omega) d\omega$. We repeat the computation deriving \eqref{eq:c0final}. In view of \eqref{eq:c(omega) defn Sec 6} (same as \eqref{eq:c0final} but $a_0(\ts,\omega)\equiv 0$) and \eqref{eq:R2 Sec 6}, we get
\[
    c(\omega) = -\int_{\RR_{\ts}}\chi'(\ts)\int_0^{\ts} \Delta_\omega R_1(s,\omega)\,ds\,d\ts,
\]
and therefore, using the fact that $\int_{\S^2}\Delta_\omega h\,d\omega = 0$ for all $h\in\CI(\S^2)$,
\begin{equation}
\label{c(omega) vanishing avg Sec 6}
    c_0 = \frac{1}{4\pi}\int_{\S^2} c(\omega) \,d\omega = -\frac{1}{4\pi}\int_{\RR_{\ts}}\chi'(\ts)\int_0^{\ts} \int_{\S_\omega^2} \Delta_\omega R_1(s,\omega)\,d\omega\,ds\,d\ts = 0.
\end{equation}
We conclude that the angular coefficient $c(\omega)$ in \eqref{fhat(0) asymp Sec 6} has vanishing average on $\S_\omega^2$.

The key ingredient in the analysis of low frequency part $\phi_0$ is the resolvent expansion. To this end, we generalize \cite[Theorem 3.1]{Hin22}, which allows us to treat the inputs $\wh{f}_0$ of the form \eqref{fhat(0) asymp Sec 6}, with the property $\int_{\S^2} c(\omega)\,d\omega = 0$. We remark that the following theorem shows that the resolvent, acting on this more general type of input, produces the same classes of singular and regular terms as in \cite[Theorem 3.1]{Hin22}.

\begin{theorem}
\label{thm:resolvent expansion Sec 6}
     Let $0<\alpha<1$. Let $c(\omega),\,d(\omega)\in \CI(\S^2)$ satisfy $\int_{\S^2} c(\omega)\,d\omega = 0$, and let $\ti f(x) \in \A^{4+\alpha}(X)$. For positive frequencies $\sigma>0$, the resolvent $\wh\Box(\sigma)^{-1}$ acting on 
     \[ 
     	f = c(\omega)\rho^3 + d(\omega)\rho^4 + \tilde{f} \in \A^{((3,0),4+\alpha)}(X)
     \]
    is then of the form: $\wh\Box(\sigma)^{-1}f = u_{\rm sing}(\sigma) + u_{\rm reg}(\sigma)$,
      \begin{align*}
         u_{\rm sing}(\sigma) &\in \sigma^2\A^{1-,((0,0),1-),((0,1),1-)}(X^+_{\rm res}), \\
        u_{\rm reg}(\sigma) &\in \CI([0,1)_\sigma;\A^{1-}(X)) + \A^{\alpha-,2+\alpha-,((2,0),2+\alpha-)}(X^+_{\rm res}), 
      \end{align*}
      where the leading term of $\sigma^{-2}u_{\rm sing}(\sigma)$ at ${\rm zf}$ is $-d_X(f)\log(\sigma/\rho)$. The constant $d_X(f)$ is given by the following formula, where $\tilde{c}(\omega)$ is defined in \eqref{Tilde c(omega) Sec 6}:
      \begin{equation}
          d_X(f) = \frac{m}{\pi}\int_X \big(f-\wh\Box(0)(\tilde{c}(\omega)\rho)\big)\, |dg_X| - \frac{1}{2\pi}\int_{\S^2} d(\omega)\,d\omega .
          \label{eq:d(f)}
      \end{equation}
    \end{theorem}

\begin{proof}
We follow the outline of the proof of \cite[Theorem 3.1]{Hin22} but present a more detailed analysis due to the appearance of the terms $c(\omega)\rho^3$ and $d(\omega)\rho^4$ in the expansion of $f$. 

We have \textit{a priori} $u:=\wh\Box(0)^{-1}f \in\A^{1-}(X)$; then, in view of \eqref{eqn:Boxhat(0)}, 
\[
    L_0 u = \rho^{-2} f - (\rho L_1 + \rho^2 L_2)u = c(\omega)\rho + d(\omega)\rho^2 + \rho^{-2}\tilde{f} - (\rho L_1 + \rho^2 L_2)u \in c(\omega)\rho + \mathcal{A}^{2-} .
\]
To solve away the term $c(\omega)\rho$, we expand $c(\omega)$ into spherical harmonics by writing
\begin{equation}
\label{Expand c(omega) Sec 6}
	c(\omega) = \sum_{k=1}^\infty c_k(\omega),\quad c_k(\omega)\in\mathcal{Y}_k,
\end{equation}
where $\mathcal{Y}_k$ is the $k$-th eigenspace, so that $\Delta_\omega|_{\mathcal{Y}_k} = k(k+1)$. We note that there is no $\mathcal{Y}_0$--component in $c(\omega)$ since $\int_{\S^2} c(\omega)\,d\omega=0$, as $\mathcal{Y}_0$ consists of constant functions. Recalling that $L_0 = -(\rho\pa_\rho)^2 + \rho\pa_\rho + \Delta_\omega$, we compute
\begin{equation}
\label{L0 on rhoYk}
	L_0(\rho c_k(\omega)) = k(k+1)\rho c_k(\omega),\quad k\geq 0.
\end{equation}
We can then find $\tilde{c}(\omega)\in\CI(\S^2)$, solving $L_0 (\tilde{c}(\omega)\rho) = c(\omega)\rho$ and defined by
\begin{equation}
\label{Tilde c(omega) Sec 6}
	\tilde{c}(\omega) = \sum_{k=1}^{\infty} \tilde{c}_k(\omega),\quad \tilde{c}_k(\omega):=\frac{c_k(\omega)}{k(k+1)},\quad\text{with $c_k(\omega)$'s given in \eqref{Expand c(omega) Sec 6}}.
\end{equation} 
It follows that $L_0 (u-\tilde{c}(\omega)\rho) \in \A^{2-}(X)$ and $\wh\Box(0)(u-\tilde{c}(\omega)\rho) = f - \wh\Box(0)(\tilde{c}(\omega)\rho) \in \A^4(X)$. Following the proof of \cite[Lemma 3.2]{Hin22}, we therefore obtain 
\[
    u - \tilde{c}(\omega)\rho = c_{(0)}\rho + \tilde{u}_0,\quad \tilde{u}_0\in \A^{2-},\quad c_{(0)} = \frac{1}{4\pi}\int_X \big(f-\wh\Box(0)(\tilde{c}(\omega)\rho)\big) \,|dg_X| .
\]
Next, we expand $\tilde u_0$. Noting that $L_0 \rho = 0$, we get 
\[
\begin{split}
	L_0 \tilde{u}_0 = L_0 (u-\tilde{c}(\omega)) &= d(\omega)\rho^2 + \rho^{-2}\tilde{f} - \rho L_1 ((c_{(0)} + \tilde{c}(\omega))\rho) - \rho L_1 \tilde{u}_0 - \rho^2 L^2 u \\
	&\in (d(\omega) - 2mc_{(0)} -2m\tilde{c}(\omega))\rho^2 + \mathcal{A}^{2+\alpha} .
\end{split}
\]
By direct computation, we have for any $Y_k(\omega)\in\mathcal{Y}_k$:
\[
    L_0(\rho^2 Y_k(\omega)) = (k(k+1) -2)\rho^2 Y_k(\omega),\quad k\geq 0;\qquad {\rm and } \ L_0((\rho^2\log\rho) Y_1(\omega)) = -3\rho^2 Y_1(\omega).
\]
Expanding $d(\omega)=d_0+\sum_{k=1}^{\infty} d_k(\omega)$ into spherical harmonics, with $d_0\in\calY_0$ being a constant and $d_k(\omega)\in\calY_k$, we can obtain from the computations above that
\[
L_0\left(\tilde u_0 +  \frac{d_1(\omega)-2m\tilde{c}_1(\omega)}{3}\rho^2\log\rho - \big( mc_{(0)}-\frac{d_0}{2} + \sum_{k=2}^\infty \frac{d_{k}(\omega)-2m\tilde{c}_k(\omega)}{k(k+1)-2} \big)\rho^2\right) \in \mathcal{A}^{2+\alpha}.
\]
Using the Mellin transform as in the proof of Proposition \ref{prop:bODEconormal}, we get
\[
\tilde u_0 +  \frac{d_1(\omega)-2m\tilde{c}_1(\omega)}{3}\rho^2\log\rho - \big( mc_{(0)}-\frac{d_0}{2} + \sum_{k=2}^\infty \frac{d_{k}(\omega)-2m\tilde{c}_k(\omega)}{k(k+1)-2} \big)\rho^2 = \rho^2   Y_{1} + \tilde u
\]
for some $Y_1\in\mathcal{Y}_1$ and some $\tilde{u} \in \A^{2+\alpha-}(X)$.
We then conclude and write
\[
\begin{gathered}
    u_0 :=\wh\Box(0)^{-1} f = (c_{(0)}+\tilde{c}(\omega))\rho - \frac{d_1(\omega)-2m\tilde{c}_1(\omega)}{3}\rho^2\log\rho + (\tilde{d}_0 + \tilde{d}(\omega))\rho^2 + \tilde{u}, \\
    \text{with}\quad\tilde{d}_0 := mc_{(0)}-\frac12 d_0, \quad \tilde{d}(\omega) := Y_1(\omega) + \sum_{k=2}^\infty \frac{d_{k}(\omega)-2m\tilde{c}_k(\omega)}{k(k+1)-2} \in \bigoplus_{k=1}^\infty \mathcal{Y}_k
\end{gathered}
\]
As in \eqref{eqn:f_1(sigma)}, we next compute
\begin{equation}
\label{f_1(sigma) defn S6}
    f_1(\sigma) := -\sigma^{-1}(\wh\Box(\sigma) - \wh\Box(0)) u_0 = (-2i\rho(Q_0 + \rho^2\widetilde{Q})-g^{00}\sigma)u_0 .
\end{equation}
Noting that $Q_0\rho = 0$, $Q_0(\rho^2\log \rho) = \rho^2 \log\rho + \rho^2$, and $Q_0(\rho^2)=\rho^2$, we obtain
\begin{equation}
\label{f_1(sigma) expansion S6}
    f_1(\sigma) = \frac{2i}{3}(d_1(\omega)-2m\tilde{c}_1(\omega))\rho^3\log\rho - 2i(\tilde{d}_0 -\frac{d_1(\omega)}{3}+\frac{2m\tilde{c}_1(\omega)}{3}+\tilde{d}(\omega))\rho^3 + \tilde{f}_1(\sigma),
\end{equation}
where $\tilde{f}_1(\sigma) := -2i\rho Q_0 \tilde{u} - 2i\rho^3\widetilde{Q}u_0 - \sigma g^{00}u_0 \in \A^{3+\alpha-}(X) + \sigma \A^3(X)$.

\noindent
Let us now iterate the resolvent identity \eqref{resolvent identity: cubic}, with $f_1(\sigma)$ defined in \eqref{f_1(sigma) defn S6}; we have
\begin{equation}
\label{resolvent expansion 1 S6}
\begin{split}
    \wh\Box(\sigma)^{-1} f &= u_0 + \sigma\wh\Box(\sigma)^{-1} f_1(\sigma) \\
    &= u_0 + \sigma\wh\Box(0)^{-1} f_1(\sigma) - \sigma\wh\Box(\sigma)^{-1} (\wh\Box(\sigma)-\wh\Box(0))\wh\Box(0)^{-1} f_1(\sigma).
\end{split}
\end{equation}
We therefore expand $\wh\Box(0)^{-1} f_1(\sigma)$: to this end, we define
\[
    u_1 := \wh\Box(0)^{-1}f_1(0) \in \A^{1-}(X).
\]
In view of \eqref{f_1(sigma) expansion S6}, we have 
\[
\begin{split}
    L_0 u_1 &= \rho^{-2} f_1(0) - (\rho L_1 + \rho^2 L_2)u_1 \\
    &\in \frac{2i}{3}(d_1(\omega)-2m\tilde{c}_1(\omega))\rho\log\rho - 2i\bigl(\tilde{d}_0 -\frac{d_1(\omega)}{3}+\frac{2m\tilde{c}_1(\omega)}{3}+\tilde{d}(\omega)\bigr)\rho + \A^{1+\alpha-}(X).
\end{split}
\]
To solve away the leading terms, we compute directly:
\[
    L_0(Y_1(\omega)(\rho\log\rho + \dfrac{\rho}{2})) = 2Y_1(\omega)\rho\log\rho,\quad L_0(\rho\log\rho) = -\rho,
\]
Using \eqref{L0 on rhoYk}, we conclude:
\[
\begin{gathered}
	L_0 \left(u_1 - \bigl(2i\tilde{d}_0 + \frac{i}{3}(d_1(\omega)-2m\tilde{c}_1(\omega))\bigr)\rho\log\rho - \widetilde{Y}(\omega)\rho \right) \in \A^{1+\alpha-}(X), \\
	\text{with}\quad \widetilde{Y}(\omega) := \frac{i}{2}(d_1(\omega)-2m\tilde{c}_1(\omega))-iY_1(\omega) - 2i\sum_{k=2}^\infty \frac{d_{k}(\omega)-2m\tilde{c}_k(\omega)}{k(k+1)(k^2 + k -2)} \in \bigoplus_{k=1}^\infty \mathcal{Y}_k.
\end{gathered}
\]
In view of Proposition \ref{prop:bODEconormal}, we obtain
\begin{equation}
\label{u_1 expansion S6}
    u_1 = \bigl(2i\tilde{d}_0 + \frac{i}{3}(d_1(\omega)-2m\tilde{c}_1(\omega))\bigr)\rho\log\rho + \tilde{u}_1,\quad \tilde{u}_1 \in \A^{((1,0),1+\alpha-)}(X).
\end{equation}
Writing $f_1(\sigma) = f_1(0) + \sigma f_1'$ with $f_1'=g^{00}u_0 \in \A^3(X)$ and letting $u_1' :=\wh\Box(0)^{-1}(f_1')\in\A^{1-}(X)$, we rewrite the expansion \eqref{resolvent expansion 1 S6} as follows
\begin{equation}
\label{resolvent expansion 2 S6}
    \begin{split}
    \wh\Box(\sigma)^{-1} f = u_0 &+ \sigma u_1 + \sigma^2 u_1' - \sigma\wh\Box(\sigma)^{-1} (\wh\Box(\sigma)-\wh\Box(0))u_1 - \sigma^2\wh\Box(\sigma)^{-1} (\wh\Box(\sigma)-\wh\Box(0))u_1'
\end{split}
\end{equation}
To examine the fourth term on the right, which contributes to $u_{\rm sing}$, we follow the computation in \eqref{f_1(sigma) defn S6} and use \eqref{u_1 expansion S6}:
\begin{equation}
\label{f_2 expansion S6}
\begin{gathered}
    f_2(\sigma) := -\sigma^{-1}\bigl(\wh\Box(\sigma)-\wh\Box(0)\bigr)u_1 = (-2i\rho(Q_0 + \rho^2\widetilde{Q})-g^{00}\sigma)u_1 = f_{2,0} + \tilde f_2(\sigma), \\
    f_{2,0} = -2i\rho Q_0 \bigl(2i\tilde{d}_0 + \frac{i}{3}(d_1(\omega)-2m\tilde{c}_1(\omega))\bigr)\rho\log\rho = \bigl(4 \tilde{d}_0 + \frac{2}{3}d_1(\omega) - \frac{4m}{3}\tilde{c}_1(\omega) \bigr)\rho^2,\\
    \tilde f_2(\sigma) = -2i\rho Q_0\tilde{u}_1 - 2i\rho^3\widetilde{Q}u_1 - \sigma g^{00} u_1 \in \A^{2+\alpha-}(X) + \sigma\A^{3-}(X),
\end{gathered}
\end{equation}
Here, we used the fact that $Q_0$ eliminates $\rho$, which is the $(1,0)$-leading term in $\tilde{u}_1$, so $Q_0\tilde{u}_1\in \A^{1+\alpha-}(X)$. We then apply Proposition \ref{Hintz Lemma 2.24} to $f_{2,0}$ and conclude
\begin{equation}
\label{EqPBoxInvf2Nonlin}
  \wh\Box(\sigma)^{-1}f_{2,0} \in \cA^{1-,((0,0),1-),((0,1),1-)}(X^+_{\rm res}),
\end{equation}
with leading order term at $\zface$ given by $-4\tilde{d}_0\log\frac{\sigma}{\rho}$, and leading order term at $\tface$ given by $\Tilde{u}_{\rm{mod}} = \widetilde\Box^{-1}\bigl(\hat\rho^2 (4 \tilde{d}_0 + \frac{2}{3}d_1(\omega) - \frac{4m}{3}\tilde{c}_1(\omega))\bigr)$ (in the notation of Lemma \ref{lem:2.23Hintz}). Here, we have used the facts that $\int_{\S^2} d_1(\omega)\,d\omega=0$ and $\int_{\S^2} \tilde{c}_1(\omega)\,d\omega = 0$, since $d_1(\omega),\,\tilde{c}_1(\omega)\in\mathcal{Y}_1$. %

\noindent
Recalling from Proposition \ref{prop:1} that $\wh\Box(\sigma)^{-1}: \A^{2+\beta}(X) \to \A^{\beta-,\beta-,((0,0),\beta-)}(X_{\rm{res}}^+)$, $0<\beta<1$, it then follows from \eqref{f_2 expansion S6} that
\[
\begin{gathered}
    \sigma^2\wh\Box(\sigma)^{-1}\Tilde{f}_2(0) \in \sigma^2 \A^{\alpha-,\alpha-,((0,0),\alpha-)}(X_{\rm{res}}^+) = \A^{\alpha-,2+\alpha-,((2,0),2+\alpha-)}(X_{\rm{res}}^+), \\
    \sigma^2\wh\Box(\sigma)^{-1}(\Tilde{f}_2(\sigma) - \Tilde{f}_2(0)) \in \sigma^3 \A^{1-,1-,((0,0),1-)}(X_{\rm{res}}^+) = \A^{1-,4-,((3,0),4-)}(X_{\rm{res}}^+).
\end{gathered}
\]
For the last term in \eqref{resolvent expansion 2 S6}, we note that
\[
    \sigma^{-1}\bigl(\wh\Box(\sigma)-\wh\Box(0)\bigr)u_1' = (2i\rho(Q_0 + \rho^2\widetilde{Q})+g^{00}\sigma)u_1' \in \A^{2-}(X) + \sigma \A^{3-}(X).
\]
Recalling from \cite[Lemma 2.17]{Hin22} that $\wh\Box(\sigma)^{-1}: \A^{2-}(X)\to \A^{1-,0-,0-}(X_{\rm{res}}^+)$, we therefore obtain $\sigma^2\wh\Box(\sigma)^{-1} (\wh\Box(\sigma)-\wh\Box(0))u_1' \in \A^{1-,3-,3-}(X_{\rm{res}}^+)$. We conclude that
\begin{equation}
\label{final expansion S6}
    \wh\Box(\sigma)^{-1} f \in u_0 + \sigma u_1 + \sigma^2 u_1' + \sigma^2\wh\Box(\sigma)^{-1}f_{2,0} + \A^{\alpha-,2+\alpha-,((2,0),2+\alpha-)}(X_{\rm{res}}^+).
\end{equation}
Noting that $u_0 + \sigma u_1 + \sigma^2 u_1'$ belongs to $\CI([0,1)_\sigma;\A^{1-}(X))$, and by combining \eqref{EqPBoxInvf2Nonlin} and \eqref{final expansion S6}, we obtain the desired result. The constant $d_X(f) := 4\tilde{d}_0 = 4mc_{(0)} - 2d_0$ is given by \eqref{eq:d(f)}. 
\end{proof} 

We now apply Theorem \ref{thm:resolvent expansion Sec 6} to $\wh{f}_0$, which satisfies \eqref{fhat(0) asymp Sec 6}. We note that $\wh\Box(\sigma)^{-1}\wh{f}_0$ has the same form as in \cite[Theorem 3.1]{Hin22}. Following the argument in the proof of \cite[Theorem 3.4]{Hin22}, we obtain
\begin{equation}
\label{eq:phi0hat Sec 6}
    \begin{split}
     \widehat{\phi_0}(\sigma)=\psi(\sigma) \wh\Box(\sigma)^{-1} \wh{f}_0 \in &-d_X(\wh{f}_0)\psi(\sigma)\sigma^2\log(\sigma+i0)+\CcI((-1,1)_\sigma;\CI(X^\circ))\\ &+\sum_{\pminus} \A^{3-}(\pminus[0,1)_\sigma;\CI(X^\circ)),
\end{split}
\end{equation}
where we replaced the conormal space $\A^{1-}(X)$ with $\CI(X^\circ)$.

\noindent
In view of the computation in Section \ref{subsubsection:low frequency part} for $\calF^{-1}(\psi(\sigma)\sigma\log(\sigma+i0))(\ts)$, we can derive 
\[
    \calF^{-1}(\psi(\sigma)\sigma^2\log(\sigma+i0)) \in -2\chi(\ts)\ts^{-3} + \mathscr{S}(\RR_{\ts}),
\]
with $\chi(\ts)\in\CI(\RR_{\ts})$ being identically $0$ when $\ts\leq 1/3$ and identically $1$ when $\ts\geq 1$.

\noindent
We also note that the inverse Fourier transform of the $\CcI((-1,1)_\sigma;\CI(X^\circ))$ part lies in $\mathscr{S}(\RR_{\ts})$; either of the $\A^{3-}$ terms in \eqref{eq:phi0hat Sec 6} has inverse Fourier transform that is bounded by $t_*^{-4+}$ together with all their  $t_*\pa_{\ts}$ and $\pa_x$ derivatives, by \cite[Lemma 3.6]{Hin22}. We can therefore conclude that for $x\in K\Subset X^\circ$,
\begin{equation}
\label{phi0 asymp Sec 6}
    |\pa_{\ts}^j \pa_x^\beta (\phi_0 - 2d_X(\wh{f}_0) t_*^{-3})|\leq C_{j\beta K} t_*^{-4-j+},\quad \forall\, j\in\mathbb{N},\ \beta\in\mathbb{N}^3.
\end{equation} 

\begin{remark}
\label{remark:Minkowski techs} 
In Minkowski spacetime, the solution to $\Box\phi=b\phi^p, \, p \geq 5,$ arising from smooth, compactly supported initial data, admits a radiation field expansion with a simpler pattern than that in Lemma \ref{lem:rad field small time}. Namely, $\phi = \sum_{k=1}^{p-1} \rho^k R_k(t_*,\omega) + \CI([0,\infty)_{t_*};\A^{p-}(X))$, where
\[
    R_k(t_*,\omega) = \frac{1}{2(k-1)}\int_0^{t_*}(\Delta_\omega - (k-2)(k-1))R_{k-1}(s,\omega)\,ds \in \bigoplus_{j=k-1}^\infty \mathcal{Y}_j,\quad 1< k< p-1,  
\]
\[
    R_{p-1}(t_*,\omega) = \frac{1}{2(p-2)} \int_0^{t_*} ((\Delta_\omega-(p-3)(p-2))R_{p-2}(s,\omega) -b_0 R_1^p(s,\omega))\,ds.
\]
One can then extend the resolvent expansion in Theorem \ref{thm:resolvent expansion Sec 6} to higher order, noting that the singular part is more regular (in $\sigma$) than $\sigma^2\log\sigma$ when restricted to the Minkowski setting with $p\geq 5$. For instance, when $p=5$, the most singular term in $\wh\Box(\sigma)^{-1}\wh{f}_0$ is $\sigma^3\log\sigma$, which generates the leading order term $Ct_*^{-4}$ in the asymptotic expansion in $t_*$, where
\[
    C= \frac{2}{\pi}\int_0^\infty  \int_{\mathbb{S}^2} b_0 R_1^5(t_*,\omega) \,d\omega\,dt_*.
\]
\end{remark}

\subsection{Finalizing the proof of Theorem \ref{thm:highorder}}
We consider the part $\phi_1$ given by
\[
    \widehat{\phi_1}(\sigma) = \sigma\psi(\sigma)\wh\Box(\sigma)^{-1}(\sigma^{-1}(\wh{f}(\sigma)-\wh{f}(0))).
\]
Using the Taylor expansion of $\wh{f}(\sigma)$ at $\sigma=0$, we write 
\begin{equation}
\label{phi1 Taylor Sec 6}
    \frac{\widehat{f}(\sigma)-\widehat{f}(0)}{\sigma} = i\wh{g} -\frac{\sigma}{2}\wh{h} - i\sigma^2\int_0^1 \frac{(1-s)^2}{2}\int_{\RR_{\ts}} e^{i s\sigma\ts} \ts^3 f(\ts,x)\,d\ts ds,
\end{equation}
\[
\text{where}\quad \wh{g}(x) = \int_{\RR} \ts f(\ts,x)\,d\ts,\quad \wh{h}(x) = \int_{\RR} \ts^2 f(\ts,x)\,d\ts.
\]
In view of \eqref{the forcing f Sec 6} and \eqref{phi conormal est high order} we have
\[
    \int_0^1 \frac{(1-s)^2}{2}\int_{\RR_{\ts}} e^{i s\sigma\ts} \ts^3 f(\ts,x)\,d\ts ds \in \A^0(\pminus[0,1)_\sigma;\A^3(X)).
\]
Recalling the mapping property $\wh\Box(\sigma)^{-1} : \A^0([0,1)_\sigma;\A^{3}(X)) \to \A^{1-,1-,0}(X_{\rm{res}}^+)$ from Lemma \ref{lem:aux}, we conclude from \eqref{phi1 Taylor Sec 6} that
\[
    \wh{\phi_1}(\sigma) \in i\sigma\psi(\sigma)\wh\Box(\sigma)^{-1}\wh{g} - \frac{\sigma^2}{2}\psi(\sigma)\wh\Box(\sigma)^{-1}\wh{h} + \sigma^3 \A^{1-,1-,0}(X_{\rm{res}}^{\pminus}) .
\]
By a direct computation as in Section \ref{subsection: low frequency Sec 6}, one can check that both $\wh{g}$ and $\wh{h}$ admit an asymptotic expansion of the form \eqref{fhat(0) asymp Sec 6} that satisfies the assumptions of Theorem \ref{thm:resolvent expansion Sec 6}. Therefore, the inverse Fourier transform of $\psi(\sigma)\wh\Box(\sigma)^{-1}\wh{g}$ or $\psi(\sigma)\wh\Box(\sigma)^{-1}\wh{h}$ satisfies the same estimate as $\phi_0$ in \eqref{phi0 asymp Sec 6}. The extra factor of $\sigma$ improves the decay in $\ts$ since
\[
    \mathcal{F}^{-1}(i\sigma\psi(\sigma)\wh\Box(\sigma)^{-1} \wh{g})
    = -\pa_{\ts}\mathcal{F}^{-1}(\psi(\sigma)\wh\Box(\sigma)^{-1} \wh{g}),
\]
\[
    \mathcal{F}^{-1}(-\sigma^2\psi(\sigma)\wh\Box(\sigma)^{-1} \wh{h})
    = \pa_{\ts}^2 \mathcal{F}^{-1}(\psi(\sigma)\wh\Box(\sigma)^{-1} \wh{h}) .
\]
Combining these with the fact that $\sigma^3 \A^{1-,1-,0}(X_{\rm{res}}^{\pminus})$ terms have inverse Fourier transforms that are bounded by $\mathcal{O}(\ts^{-4})$ together with all their derivatives in $\ts\pa_{\ts}$ and $\pa_x$, we conclude that for $x\in K\Subset X^\circ$,
\begin{equation}
\label{phi1 asymp Sec 6}
    |\pa_{\ts}^j \pa_x^\beta \phi_1 (\ts,x)|\leq C_{j\beta K} t_*^{-4-j+},\quad \forall\, j\in\mathbb{N},\ \beta\in\mathbb{N}^3 .
\end{equation} 

The high frequency part $\phi_2$ is harmless, as one can show that \eqref{phi2 rapid decay} still holds for the initial value problems \eqref{initial value problem Sec 6} with $p\geq 4$. It suffices to observe that our forcing term $f(\ts,x)$ given in \eqref{the forcing f Sec 6} still satisfies the estimate \eqref{eqn:forcing f decay}. In fact, it has better decay:
\[
    \pa_{t_*}^k(\ts\pa_{t_*})^m \Gamma_b^I f(\ts,x) = \mathcal{O}(\la t_*\ra^{-2p-k} \la t_*+2r\ra^{-p});
\]
we can therefore repeat the steps in subsection \ref{subsubsection:high frequency} to obtain 
\begin{equation}
\label{phi2 rapid decay Sec 6}
    \phi_2(t_*,x) \in \mathscr{S}(\mathbb{R}_{t_*};\A^{1-}(X)).
\end{equation}
Combining \eqref{phi0 asymp Sec 6}, \eqref{phi1 asymp Sec 6} and \eqref{phi2 rapid decay Sec 6}, we obtain \eqref{phi asymp spatial compact highorder}, which completes the proof of Theorem \ref{thm:highorder}.

\begin{remark}
\label{rmk:p=4 or larger}
Our approach actually allows us to establish Theorem \ref{thm:highorder} for more general initial data than the smooth and compactly supported functions as in \eqref{initial data finer assumption}. More specifically, our proof can be extended \textit{verbatim} to the case where 
\[
    \begin{gathered}
        u_0(x) - u_{0}^1(\omega)\rho - u_{0}^2(\omega)\rho^2 - u_{0}^3(\omega)\rho^3 \in \A^{4-}(X),
        \quad
        u_1(x) - u_{1}^1(\omega)\rho - u_{1}^2(\omega)\rho^2 \in \A^{3-}(X),
    \end{gathered}
\]
with $u_0^1, u_0^2, u_0^3, u_1^1, u_1^2 \in\CI(\S^2)$, under the assumptions (a) an integral condition $$   
\int_{\S^2} (2u_0^2(\omega)+\tilde{g}(\omega) u_1^1(\omega))\,d\omega = 0$$ on the sphere involving $u_{0}^2$, $u_{1}^1$ and $\tilde g$ holds, and (b) the solution to the corresponding initial value problem \eqref{initial value problem Sec 6} satisfies the \textit{a priori} estimate \eqref{phi conormal est high order} (or even the weaker \textit{a priori} estimate \eqref{phi conormal est}).
Assumption (a) on $u_0^2,\, u_1^1$ guarantees the validity of the key property \eqref{c(omega) vanishing avg Sec 6}, with the appearance of angular coefficients from the initial data in the formula \eqref{eq:c(omega) defn Sec 6}, of $c(\omega)$.
\end{remark}

\appendix
\section{Global existence and pointwise decay theorem from previous work for the power-type nonlinearities that are cubic and higher order}\label{app:GE}

The sharp decay estimates used as an assumption in this paper for power-type nonlinearities on asymptotically flat spacetimes have been established under specific smoothness and decay conditions on the initial data. Additionally, there are assumptions regarding the asymptotic flatness of the coefficients of the operator. We first state the result for asymptotically Minkowski spacetimes, followed by another result for black hole spacetimes.

We assume the existence of a global solution throughout this work. Suppose we are in the defocusing, constant-coefficient case $a \equiv1, b\equiv1$. Then in stationary spacetimes $g = g(x) $, the global existence of solutions for defocusing power-type nonlinearities with large initial data can be established by leveraging a combination of local-in-time Strichartz estimates and the conservation of the energy associated to \eqref{IVP intro} and \eqref{initial value problem Sec 6}. These tools and estimates are well-known. %
In the energy critical case where $p=5$, global existence for the initial value problem with large initial data can be proven by performing an estimate in the backward cone (domain of dependence) with an apex at a fixed time slice. Specifically, if $\|u\|_{L^6_x(\text{slices})} \to 0$ as one approaches the apex along constant-time slices within the cone, then global existence follows. %
Global existence is known for focusing $a\equiv-1, b \equiv -1$ power-type nonlinearities with sufficiently small initial data. %
In the variable-coefficient case with $a(t_*,x)$ and $b(t_*,x)$, global existence considerations would be related to the sign properties of these functions.

\begin{theorem}[Corollary of results from \cite{Looi22,Looi22.2}]\label{thm:power_nonlinearity}
Let the coefficient function $a(t_*, x)$ be smooth and satisfy the following estimates for all integers $k \in \N$ and multi-indices $I$:
\begin{equation}
\label{a(t*,x) estimate in app}
    |\pa_{\ts}^k \mathcal{V}_X^I a(\ts,x)|\leq C_{k, I} \la \ts \ra^{-k},\quad\forall k,I, \quad \mathcal{V}_X\in \mathrm{span}\{\pa_{r}, \pa_\rho, \Omega_1,\Omega_2,\Omega_3\}.
\end{equation}
Let $\phi$ be a global solution to the initial value problem
\begin{equation}
\begin{cases}
\Box_g \phi = a(t_*,x) \phi^p, & \quad p \geq 3 \textrm{ an integer}, \\
(\phi, \partial_{t_*}\phi)|_{t_*=0} = (\phi_0, \phi_1),
\end{cases}
\end{equation}
with initial data $(\phi_0, \phi_1) \in \mathcal{C}_c^\infty(\mathbb{R}^3) \times \mathcal{C}_c^\infty(\mathbb{R}^3)$. The existence of such a global solution is assumed, and is known to hold under standard conditions, such as a defocusing condition on $a(t_*,x)$ (e.g., constant $a \ge 0$) or a smallness assumption on the initial data in the focusing case.
Then $\phi$ satisfies the following decay estimate in $(t_*,x)$ coordinates:
\begin{equation}
|(\langle t_*\rangle\partial_{t_*})^k \Gamma_b^I \phi(t_*,x)| \leq C_{Ik} \langle t_*+2|x| \rangle^{-1} \langle t_*\rangle^{-\min(2,p-2)}.
\label{eq:A.2}
\end{equation}
\end{theorem}

Various cases of this result were proven in \cite{Looi22,Looi22.2} for a broad class of asymptotically Minkowski spacetimes: potentials and first-order terms were allowed. The work \cite{Looi22.2}, which derived decay estimates for a wider range of nonlinearities, also addressed some cases with variable coefficients in the nonlinearity terms. In \cite{Looi22}, the bound \eqref{eq:A.2} was proved for large initial data and $p=5$ and with $a \equiv 1$. The bounds for the power-type nonlinearities with $a\equiv1$ proved in \cite{Looi22,Looi22.2} are sharp. Since only a particular form of asymptotic behavior (as $r\to\infty$, or equivalently $\rho\to0$) of the metric is relevant here, we have restricted to that case when stating this theorem, omitting the statement for more general linear operators than $\Box_g$. A distinction from \cite{Toh2022} is that the results in \cite{Looi22,Looi22.2} apply to a large class of asymptotically Minkowski spacetimes and, in the case of the quintic nonlinearity ($p=5$) in \cite{Looi22}, allow for large initial data. %

The same bound as \eqref{eq:A.2} is proved for the $Z$ vector fields\footnote{We will introduce this space in the next section. In the references \cite{Looi22, Looi22.2, Toh2022}, the space of $ Z $-controlled functions, denoted $ S^Z $, is used. }. 
The proof of the estimate involving $k \in \mathbb{N}$ in \eqref{eq:A.2} proceeds as follows. Bounds are established for the solution and all its $ Z $-vector fields, including the scaling field $ S $. Additionally, it is shown that in the $(t, r, \omega)$ coordinate system, the radial derivative of the solution gains a factor of $ 1/(1+r) $. By combining these two results, the $\partial_t$ derivative of the solution thereby gains a factor of $ 1/(1+t) $. This gives us the part of \eqref{eq:A.2} involving $k \in \N$.

The main theorem of \cite{Toh2022}, which established global existence and pointwise bounds for solutions to the semilinear wave equations on Kerr backgrounds with small angular momenta, is the following:

\begin{theorem}[Theorem from \cite{Toh2022}]\label{thm:Toh2022}
Let $p \geq 3$ be an integer and $(\mathcal{M}, g)$ be a Kerr spacetime. Define $\mathcal{M} = \{ \tilde{t} \geq 0, r \geq r_e \}$, where $r_e$ is the event horizon radius. For any $T \geq 0$, let $\Sigma(T) = \mathcal{M} \cap \{ \tilde{t} = T \}$ and $\Sigma^- = \Sigma(0)$. 

Consider the Cauchy problem for the semilinear wave equation:
\begin{equation}
\begin{cases}
\Box_K \phi = \pminus \phi^p \\
\phi |_{\Sigma^-} = \phi_0, \quad \tilde{T} \phi |_{\Sigma^-} = \phi_1
\end{cases}
\end{equation}
where $\Box_K$ is the d'Alembertian on the Kerr spacetime, and $\tilde{T}$ is the future-directed unit normal to $\Sigma^-$.

Let $\kappa = \min\{2, p-2\}$. For any fixed $m \in \mathbb{N}$ and $R_1 > r_e$, there exist $N \gg m$ and $\varepsilon > 0$ such that for all initial data $(\phi_0, \phi_1)$ supported in $\{r_e \leq r \leq R_1\}$ satisfying
\[
\|\phi_0\|_{H^{N+1}(\Sigma^-)} + \|\phi_1\|_{H^N(\Sigma^-)} \leq \varepsilon,
\]
such that $\phi$ exists globally on $\mathcal{M}$. Moreover, for all $0 \leq |\alpha| \leq m$, 
\[
|Z^\alpha \phi(\tilde{t},r,\theta,\varphi)| \lesssim \varepsilon \langle \tilde{t} \rangle^{-1} \langle \tilde{t} - \tilde{r} \rangle^{-\kappa},
\]
where $\langle a \rangle = \sqrt{2 + |a|^2}$, and $\tilde{r}$ is a modified radial coordinate equal to $r$ in a compact region and asymptotic to the Regge-Wheeler coordinate $r^*$ near spatial infinity. 
\end{theorem}
The coordinate $\tilde{t}$ is chosen such that the hypersurface $\{\tilde{t} = 0\}$ is spacelike and coincides with the standard Boyer-Lindquist time coordinate $t$ away from the black hole. The theorem holds for both focusing and defocusing nonlinearities.

\subsection{Vector Fields and Function Spaces}
\label{app:vector_field_methods}

\begin{definition}[Vector fields]\label{def:vector_fields}
On $\mathbb{R}^{1+3}$, we denote by $Z$ the full set of vector fields, comprising the translations $\partial = (\partial_t, \partial_i)$, the rotations $\Omega = \{x^i \partial_j - x^j \partial_i\}$, and the scaling vector field $S = t\partial_t + x^i \partial_i$.
\end{definition}

\begin{definition}[Function spaces]
Let $f$ be a positive function on $X:=[1,\infty)_r \times \mathbb{S}^2$ and $Z$ denote the set of vector fields comprising partial derivatives, rotations $x^i\partial_j - x^j\partial_i$, and the scaling field $t\partial_t + \sum_i x^i\partial_i$.

\begin{enumerate}
    \item The space of $Z$-controlled functions is defined as
    $S^Z(f) := \{ g \in \CI(\mathbb{R}_{t_*} \times X) : |Z^\alpha g| \leq c_\alpha f, \quad \forall \alpha \},$
    where $c_\alpha$ are constants and $\alpha$ is a multi-index.
    
    \item The space of stationary $Z$-controlled functions is:
    $S^Z_{\text{stat}}(f) := \{ g \in \CI(X) : |Z^\alpha g| \leq c_\alpha f, \quad \forall \alpha \}.$
\end{enumerate}
\end{definition}

\begin{example}
An element of $\CI(X)$ can be expressed as
$a(x) = \sum_{j=0}^\infty r^{-j} a_j(\omega), r \geq 1,$
with $a_j \in \CI(\mathbb{S}^2)$ satisfying appropriate decay conditions. For instance, the defocusing and focusing nonlinear wave equations with constant coefficient nonlinearity correspond to $a = \pminus 1$, with $a_0 = \pminus 1$ and $a_j = 0$ for all $j \geq 1$.
\end{example}

\begin{proposition}[Relationships between function spaces]
\label{prop:function_space_relations}
For functions defined on $X$, we have $\CI(X) \subset S^Z_{\text{stat}}(1) \subset \mathcal{A}^0(X)$ and $S^Z_{\text{stat}}(r^{-\alpha}) = \mathcal{A}^\alpha(X)$ for $\alpha \geq 0$.
\end{proposition}

\begin{proof}
The inclusions in (1) follow from the definitions, noting that $Z$ applied to stationary functions includes all vector fields defining $\mathcal{A}^\alpha(X)$. The equality in (2) holds because for $\alpha \geq 0$, both spaces impose equivalent decay conditions.
\end{proof}

\begin{remark}\label{rem:b_vector_fields}
Control over $Z$ vector fields up to order $N \in \mathbb{N}$ implies control over all b-vector fields of the same order, as b-vector fields are generated by $\{r\partial_r, \Omega\}$, which are contained in the span of $Z$ when restricted to stationary functions.
\end{remark}

\end{document}